\documentclass[
    a4paper,
    DIV=14,
    abstract=true,
    numbers=noenddot
]
{scrartcl}

\usepackage
{
    amsmath,
    amssymb,
    amsthm,
    authblk,
    bm,
    enumitem,
    fontawesome,
    graphicx,
    mathtools,
    nicefrac,
    physics,
	todonotes,
    xcolor
}

\usepackage[bf,normal]{caption}
\usepackage{subcaption}

\usepackage[T1]{fontenc}

\usepackage[pdffitwindow=false,
            plainpages=false,
            pdfpagelabels=true,
            pdfpagemode=UseOutlines,
            pdfpagelayout=SinglePage,
            bookmarks=false,
            colorlinks=true,
            hyperfootnotes=false,
            linkcolor=blue,
            urlcolor=blue!30!black,
            citecolor=green!50!black]{hyperref}

\newtheorem{theorem}{Theorem}[section]
\newtheorem{proposition}[theorem]{Proposition}
\newtheorem{corollary}[theorem]{Corollary}
\newtheorem{lemma}[theorem]{Lemma}

\newtheorem{definition}[theorem]{Definition}
\theoremstyle{definition}
\newtheorem{example}[theorem]{Example}

\newtheorem{observation}{Observation}
\newtheorem{assumption}{Assumption}

\newcommand{\co}[1]{\overset{\circ}{#1}}
\newcommand{\red}[1]{{#1}}
\definecolor{DeepPurple}{RGB}{61, 0, 249}
\newcommand{\X}{\mathbb{X}}

\newcommand{\wh}[1]{\widehat{#1}}
\newcommand{\tl}[1]{\widetilde{#1}}
\newcommand{\tll}[1]{\widetilde{#1}}
\DeclareMathOperator{\mspan}{span}
\DeclareMathOperator*{\argn}{arg\,min}

\providecommand\rm{}
\renewcommand{\b}[1]{{\bm{#1}}}
\renewcommand{\rm}[1]{{\mathrm{#1}}}
\newcommand{\br}[1]{\left\langle#1\right\rangle}
\newcommand{\set}[1]{\left\{#1\right\}}

\newcommand{\qb}[1]{\left[#1\right]}
\newcommand{\Id}{\bm{I}}
\renewcommand{\norm}[1]{\left\lVert #1 \right\rVert}
\renewcommand{\abs}[1]{\left| #1 \right|}

\newcommand{\A}{\mathbb{A}}
\newcommand{\C}{\mathbb{C}}
\newcommand{\E}{\mathbb{E}}

\newcommand{\N}{\mathbb{N}}

\newcommand{\R}{\mathbb{R}}

\newcommand{\Aa}{\mathcal{A}}
\newcommand{\Bb}{\mathcal{B}}

\newcommand{\Dd}{\mathcal{D}}

\newcommand{\Ff}{\mathcal{F}}

\newcommand{\Hh}{\mathcal{H}}
\newcommand{\Kk}{\mathcal{K}}
\newcommand{\Ll}{\mathcal{L}}

\newcommand{\Oo}{\mathcal{O}}
\newcommand{\Pp}{\mathcal{P}}

\newcommand{\Tt}{\mathcal{T}}

\renewcommand{\d}{\,\mathrm{d}}
\newcommand\restr[2]{\left.#1\right|_{#2}}
\newcommand*\samethanks[1][\value{footnote}]{\footnotemark[#1]}
\newcommand*{\email}[1]{\faEnvelopeO~\href{mailto:#1}{#1}}

\newcommand\xqed[1]{\leavevmode\unskip\penalty9999 \hbox{}\nobreak\hfill \quad\hbox{#1}}
\newcommand{\exampleSymbol}{\xqed{$\triangle$}}

\allowdisplaybreaks

\begin{document}

\title{Data-driven approximation of Koopman operators and generators: Convergence rates and error bounds}
\author[1,2]{Liam Llamazares-Elias\thanks{L.L.\ and S.L.\ contributed equally.}}
\author[3]{Samir Llamazares-Elias\samethanks}
\author[4]{Jonas Latz}
\author[5]{Stefan Klus}

\affil[1]{School of Mathematics, University of Edinburgh, UK}
\affil[2]{Current address: Department of Mathematics and Statistics, Lancaster University, UK, \email{l.llamazares@lancaster.ac.uk}}
\affil[3]{Department of Mathematics, University of Salamanca, Spain, \email{samirllamazares@usal.es}}
\affil[4]{Department of Mathematics, University of Manchester, UK, \email{jonas.latz@manchester.ac.uk}}
\affil[5]{School of Mathematical \& Computer Sciences, Heriot--Watt University, UK, \email{s.klus@hw.ac.uk}}

\date{}

\maketitle

\begin{abstract}
	Global information about dynamical systems can be extracted by analysing associated infinite-dimensional transfer operators, such as Perron--Frobenius and Koopman operators as well as their infinitesimal generators. In practice, these operators typically need to be approximated from data. Popular approximation methods are \emph{extended dynamic mode decomposition} (EDMD)  and \emph{generator extended mode decomposition} (gEDMD).	We propose a unified framework that leverages Monte Carlo sampling to approximate the operator of interest on a finite-dimensional space spanned by a set of basis functions. Our framework contains EDMD and gEDMD as special cases, but can also be used to approximate more general operators. Our key contributions are proofs of the convergence of the approximating operator under {relaxed} conditions. \red{We also prove that in some cases eigenpairs of the approximating operators weakly converge to eigenpairs of the exact operator{, in others they do not}.} Moreover, we derive explicit convergence rates and account for the presence of noise in the observations. Whilst all these results are broadly applicable, they also refine previous analyses of EDMD and gEDMD. We verify the analytical results with the aid of several numerical experiments. \\[1ex]
\end{abstract}

\section{Introduction} \label{intro section}

Dynamical systems are a vital tool to describe deterministic and stochastic processes in science and engineering: the motion of celestial bodies, dynamics of molecules, or the development of the human brain. Even the most complex dynamical systems can often be analysed by studying certain associated linear operators such as the Koopman operator, the Perron--Frobenius operator, and their generators \cite{koopman1931hamiltonian, lasota1998chaos, budivsic2012applied, klus_numerical_2016}. These operators have been used in a wide range of fields, such as molecular dynamics \cite{schutte2013metastability, SKH23}, fluid dynamics \cite{froyland2014well, froyland2016optimal}, and engineering \cite{vaidya2010nonlinear, POR20}. Koopman and Perron--Frobenius operators allow us to study the evolution of \emph{observables}, such as the velocity, acceleration, or energy of the system, and \emph{probability densities}. The generators of these operators are used to study the temporal rate of change of observables and densities, respectively.

As a result, the investigation of our dynamical system reduces to having access to a linear operator $\Aa$, which acts on observables or probability densities of our system. The goal of \emph{data-driven methods} is to obtain an approximation $\wh{\Aa}_{NM}$ of $\Aa$ using the information derived from studying the system using only a finite amount of observables $\psi_1,\dots,\psi_N$ and training points $\bm{x}_1, \dots, \bm{x}_M$. These methods gained considerable interest in the literature in recent years, and various methods have been developed to address this problem; some of the most notable are \emph{dynamic mode decomposition} {(DMD) \cite{rowley2009spectral,schmid2010dynamic,tu2013dynamic,colbrook2024multiverse}}, which relies on a least-squares estimate of the Koopman operator using a set of linear basis functions, \emph{extended dynamic mode decomposition} EDMD \cite{williams_datadriven_2014, klus2015numerical}, which can be regarded as a nonlinear generalisation of DMD, and \emph{generator extended dynamic mode decomposition} (gEDMD) \cite{klus_data-driven_2020,klus2020kernel}, which approximates the infinitesimal generator of the Koopman operator.

More recently, the convergence of these methods has been analysed. Convergence properties of EDMD were studied in \cite{korda_convergence_2018}, though no error bounds were given. The rate of convergence of gEDMD has been studied in \cite{kurdila2018koopman,zhang2023quantitative, Nske2021FiniteDataEB}. However, all these works require strong assumptions that may be impossible to verify in practice. In this work, our goal is to resolve these limitations. Our main contributions are as follows:
\begin{enumerate}  \setlength{\itemsep}{0ex}
	\item \textbf{A common framework:}  We introduce a unified framework to study methods for the data-driven analysis of dynamical systems (such as, but not limited to, EDMD and gEDMD). Within this framework, a linear operator $\Aa$ is estimated using a Monte Carlo approximation, denoted by $\wh{\Aa}_{NM}$. This approximation utilises $M$ Monte Carlo samples to estimate $\Aa _N$, which is the projection of $\Aa$  onto the dictionary space  $\mathcal{F}_N = \mspan(\psi_1,\ldots, \psi_N)$.\label{part 1}

	\item \textbf{Convergence:} We show that $\wh{\Aa}_{NM}$ is the projection of $\Aa$ onto the space of empirical samples $\wh{\Ff}_{NM} $, almost sure convergence of $\wh{\Aa}_{NM}$ to $\Aa_N$, convergence of eigenvalues and weak convergence of eigenfunctions of $\wh{\Aa }_{NM} $ along a subsequence. \label{part 2}

	\item \textbf{Error bounds:} We derive {explicit} bounds for the approximation errors $\|\wh{\Aa}_{NM}-\Aa_N\|$ and $\|\wh{\Aa}_{NM}-\Aa\|$ and then extend these results to the case of noisy observations. \label{part 4}

\item {\textbf{Relaxed assumptions:} We derive our results without the restrictive conditions prevalent in prior literature. Specifically, we do not assume the data is sampled from an invariant measure or a single trajectory, nor do we require an orthonormal dictionary, bounded operators on $\mathcal{F}$, or an invertible empirical Gram matrix. The precise conditions are detailed at the end of Section \ref{general framework section}.}
\end{enumerate}
The outline of the article is as follows: In Section \ref{general framework section}, we introduce Koopman and Perron--Frobenius operators and the mathematical setting for our problem. In Sections \ref{convergence data section}, \ref{convergence Galerkin section} and \ref{eigen section}, we prove Contribution \ref{part 2}. In Section \ref{Joint convergence section}, we prove Contribution \ref{part 4}. In Section \ref{simulations section}, we provide numerical simulations to illustrate our results. We conclude with a discussion and some recommendations based on the theoretical and numerical results in Section \ref{conclusion section}.  Our approach in Sections \ref{convergence data section} and \ref{convergence Galerkin section} is heavily influenced by the results in \cite{korda_convergence_2018}, and our work in Section \ref{Joint convergence section} is more closely related to that of \cite{Nske2021FiniteDataEB,zhang_effective_2016}. However, we develop these results under a more general framework in which an arbitrary operator is estimated and under weaker assumptions.

\section{A general framework for data-based recovery of dynamics} \label{general framework section}

In this section, we will introduce the required mathematical concepts and the notation used throughout the paper.

\subsection{Notation} \label{notation section}

Let $E_1,E_2$ be generic vector spaces and vectors $\Psi_1=\set{\psi_n}_{n=1}^{N_1} \subseteq E_1, \Psi_2=\set{\phi_n}_{n=1}^{N_2} \subseteq E_2$. We define $\mspan(\Psi_1),\mspan(\Psi_2)$ to be the smallest vector space containing $\Psi_1,\Psi_2$ respectively. An operator $\Tt \colon \mspan(\Psi_1) \to \mspan(\Psi_2)$ can be represented by any matrix $\bm{T}^{\Psi_1 \to \Psi_2} \in \mathbb{R}^{N_2 \times N_1}$ verifying
\begin{equation*}
	\Tt \psi_j = \sum_{i=1}^{N_2}\bm{T}^{\Psi_1 \to \Psi _2}_{ij}\phi_i, \quad j=1,\dots, N_1.
\end{equation*}
The matrix $\bm{T}^{\Psi_1 \to \Psi _2}$ always exists and if $\Psi_1$ and $\Psi _2$  are a collection of independent vectors, $\bm{T}^{\Psi_1 \to \Psi _2}$ is unique. If $\Psi _1=\Psi _2$ we write $\bm{T}^{\Psi_1}:=\bm{T}^{\Psi_1 \to \Psi _1}$.

We denote the Euclidean norm of a vector $\bm{v}\in\C^N$ by $\abs{\bm{v}}$. Given an operator $\Tt\colon\mathcal{D}\to\mathcal{F}$ between normed spaces and a matrix $\bm{T} \in \C^{N\times N}$, we denote the induced operator norms of $\Tt$ and $\bm{T}$ and the Frobenius norm of $\bm{T}$ by
\begin{equation*}
	\|\Tt\| := \sup_{\norm{\phi}_{\mathcal{D}}=1}\norm{\Tt\phi}_{\Ff}, \qquad \|\bm{T}\| := \sup_{\abs{\bm{v}}=1}\abs{\bm{T} \bm{v}}, \qquad 	\norm{\bm{T}}_F := \left(\sum_{i,j=1}^N\abs{\bm{T}_{ij}}^2\right)^{\nicefrac{1}{2}},
\end{equation*}
respectively.

We define random objects, such as random variables, stochastic processes, and random flows on a common underlying probability space $(\Omega, \mathfrak{A}, \mathbb{P})$, but usually simplify our presentation by ignoring this dependence. Moreover, we let  $(\X, \mathfrak{B}(\X), \mu)$ be an additional probability space with $\mathfrak{B}(\X)$ being the Borel $\sigma $-algebra on the topological space $\X$  and denote the space of $\mu$ square-integrable functions on $\X$ by $\Ff := L^2(\X\to\C, \mu)$.

In what follows, we will denote vectors by bold lowercase letters, matrices by bold uppercase letters, and  operators by calligraphic letters.  Moreover, $\bm{x}$ are elements of $\mathbb{X}$ and $\psi_1, \dots, \psi_N$ denote dictionary functions. Finally, we will always use $i,j,n \in \set{1, \dots,N}$ to index the dictionary and $m \in \set{1, \dots,M}$ to index samples where $M,N \in \N$ . For the sake of convenience, we include an overview of the notation used throughout the paper in Appendix~\ref{notation table section}.

\subsection{Dynamical systems and linear operators} \label{Koopman section}

Linear operators are a powerful tool for studying dynamical systems. Two such operators are the Koopman and Perron--Frobenius operators, along with their generators. Let $ \mathbb{X} $ be the state space and consider a deterministic discrete-time dynamical system $ \Phi \colon \mathbb{X} \to \mathbb{X} $ defined by
\begin{equation} \label{discrete}
	\bm{x}_{\ell+1} = \Phi(\bm{x}_\ell), \quad \ell \in \mathbb{N}_0,
\end{equation}
with an appropriate initial condition $\bm{x}_0 \in \mathbb{X}$. We refer to $\Phi$ as the \emph{flow} of the dynamical system. In practice, we may not have access to $\Phi$. In order to recover some information about the dynamical system, we may measure some quantity {$f\colon\X \to\C$} of the system, such as its velocity, acceleration, or energy and investigate how it evolves. That is, we study
\begin{equation}\label{Koopman operator}
	\Kk f(\bm{x}) := f(\Phi(\bm{x})),
\end{equation}
where the operator 
\red{$\Kk  \colon  C(\X) \to C(\X)$ is known as the \emph{Koopman operator}.}
Given a measure $\mu$ on $\X$, if $\Phi$ preserves $\mu$-null sets, one can extend the Koopman operator to $\Kk \colon L^\infty(\X) \to L^\infty(\X)$. {Its preadjoint operator, acting on $L^1(\X)$ and often denoted $\Kk_*$, is the \emph{Perron--Frobenius operator}.}

These ideas translate to continuous-time dynamical systems, such as
\begin{equation*}
	\bm{x}_t = \Phi^{t-s}(\bm{x}_s) \in \mathbb{X}, \quad s,t\in\R^+, s < t,
\end{equation*}
where now $\Phi^t \colon \mathbb{X} \to \mathbb{X}$ defines the flow in continuous time, and we study the time evolution of observables $f$ through the \emph{semigroup of Koopman operators} $\Kk^t$, $t \geq 0$, defined by
\begin{equation*}
	\Kk^t f(\bm{x}) := f(\Phi^t(\bm{x})).
\end{equation*}
The \emph{infinitesimal generator} of the Koopman semigroup $\Kk^t$ is then defined on \red{suitable $f$} as
\begin{equation*}
	\mathcal{L} f(\bm{x}) := \lim _{t \downarrow 0} \frac{1}{t}\left(\Kk^t f(\bm{x})-f(\bm{x})\right).
\end{equation*}
The above operators can also be extended to stochastic dynamical systems, i.e., to systems in which the flow $\Phi$ or $\Phi^t$ is a random object mapping from $\Omega \times \mathbb{X}$ to $\mathbb{X}$. In this setting, we still refer to the Koopman operators and semigroups as well as the associated Perron--Frobenius operators and their generators by $\Kk$, $\Kk^t$, $\Kk_*$, $(\Kk^t)_*$, $\Ll$, and $\Ll^*$, respectively. In this case, the Koopman operators and semigroups are given by
\begin{align} \label{eq:stochKK}
	\Kk f(\bm{x})   := \E[f(\Phi(\bm{x}))], \qquad
	\Kk^t f(\bm{x}) := \E[f(\Phi^t(\bm{x}))].
\end{align}
The corresponding Perron--Frobenius operators and the generators are then defined in an analogous fashion. A typical case is when $\bm{X}_t\in \X \red{ \subset \R^d}$ is a continuous time dynamical system evolving according to the stochastic differential equation
\begin{equation} \label{SDE}
	\mathrm{d} \bm{X}_t = \bm{b}(\bm{X}_t) \mathrm{d} t + \bm{\sigma}(\bm{X}_t) \, \mathrm{d} \bm{W}_t, \quad \bm{X}_0=\bm{x}.
\end{equation}

By applying Itô's formula, one can show that the generator $\Ll$ associated with the stochastic differential equation \eqref{SDE} \red{is defined on $f \in C^2(\mathbb{X})$ as}
\begin{equation} \label{Koopman generator}
	\mathcal{L} f= \bm{b}\cdot \nabla f+\frac{1}{2} \rm{Tr}\qty(\bm{\Sigma} \nabla^2 f),
\end{equation}
where $\bm{\Sigma} := \bm{\sigma \sigma}^\top$ \cite{cinlar2011probability}. In physical terms, $\Ll$ describes the infinitesimal rate of change of observables $f$ evolving under our dynamical system. That is, writing $ u(t, \bm{x}) := \Kk^t f(\bm{x})$, we have
\begin{equation*}
	\partial _t u = \Ll u.
\end{equation*}
The above is called the Kolmogorov backward equation. Similarly, $\Ll^*$ describes the evolution of probability distributions of our system. If $X_t$ has density $\nu(t, \bm{x}) \in C^1_tC^2_{\bm{x}}([0, T]\times \mathbb{X})$, i.e., is once differentiable in time and twice in space with bounded derivatives, it holds that
\begin{equation*}
	\partial _t \nu = \Ll^* \nu,
\end{equation*}
see \cite[Theorem 5.3.2]{pavliotis2016stochastic}.
This is known as the \emph{forward Kolmogorov equation} or \emph{Fokker--Planck} equation. As a result, knowledge of $\Ll$ gives us complete knowledge of the evolution of $\bm{X}$.

Through the expressions above, we see that the Koopman and Perron--Frobenius operators, along with their generators, are critical tools to describe how a dynamical system evolves.

\subsection{Mathematical framework} \label{continuous subsection}

We now move on to the data-driven approximation of an  operator of interest, which we denote by~$\Aa$. In what follows, this operator can be one of the operators we have introduced in the context of dynamical systems, i.e., $\Kk$ or $\Kk_*$ in discrete time and $\Kk^t$, $(\Kk^t)_* $, $\Ll$, or  $\Ll^*$ in continuous time, \red{in which case it is necessary that the flow $\Phi  $ preserves $\mu $-null sets}. However, our theory is not limited to these operators. As before, we consider the state space $(\mathbb{X}, \Bb, \mu)$ and write $\Ff := L^2(\mathbb{X}, \mu)$. The observables on which $\Aa $ can be evaluated define the \emph{domain} of $\Aa$, i.e.,
\begin{equation*}
	\Dd := \set{ f\in \Ff: \Aa f \in \Ff}.
\end{equation*}
We say that $\Aa$ is a \emph{closed operator} if the graph of $\Aa$,
\begin{equation*}
	\{(f, \Aa f) \in \Ff  \times \Ff : f \in \mathcal{D}\},
\end{equation*}
is a closed subspace of $\Ff\times\Ff$. If $\Aa$ is a closed operator, $\Dd$ is a Hilbert space with the inner product
\begin{equation*}
	\br{f,g}_{\Dd} := \br{f,g}_\Ff+\br{\Aa f, \Aa g}_\Ff, \quad \norm{f}_\Dd^2 := \br{f,f}_{\Dd}, \quad \forall f,g \in \Dd,
\end{equation*}
see \cite{stengl2023existence} for more details. A direct consequence of the above definitions is that $\norm{f}_{\Ff}\leq \norm{f}_{\Dd}$ and $\Aa\colon\Dd\to\Ff$ is a continuous operator with $\norm{\Aa}\leq 1$.

\begin{example}
Let $\X = \R^d$ and let $\mu = \mathcal{N} (0,\b{I}) $ be the unit Gaussian measure on $\R^d$. Consider the dynamical system
\begin{align*}
	\b{x} '(t) = -\b{x}, \quad \b{x}(0)= \b{x}_0.
\end{align*}
Given continuous  $f \in C_c(\X)$, by the change of variables $\bm{y} = \bm{x} e^{-t}$, we have
\begin{align*}
\norm{\Kk^t f}_\Ff^2 &=\frac{1}{(2 \pi)^{\nicefrac{d}{2}}}  \int_{\R^d}\abs{f\left(\bm{x} e^{-t}\right)}^2 e^{-\bm{x}^2 / 2} \d \bm{x}=\frac{e^{dt}}{(2 \pi)^{\nicefrac{d}{2}}}  \int_{\R^d}\abs{f(\bm{y})}^2 e^{-\bm{y}^2 e^{2t}/ 2} \d \bm{y}\leq e^{dt}\norm{f}_{\Ff}^2.
\end{align*}
As a result, we may extend $\Kk ^t$ by density to $\Kk ^t \colon \Ff \to \Ff $. {By construction, we have the bound $\|\Kk^t f\|_\Ff \le e^{dt/2} \|f\|_{\Ff}$ for all $t \ge 0$. Furthermore, $\lim_{t \downarrow 0} \|\Kk^t f - f\|_\Ff = 0$ by the dominated convergence theorem, establishing that $\Kk^t$ is a strongly continuous semigroup.} By the Hille--Yosida theorem {\cite[Theorem 3.8]{engel2000one}}, the infinitesimal generator $\Ll$ is closed and densely defined. This shows that we may take both $\Aa =\Kk^{t}$ for some fixed $t$ and $\Aa =\Ll$ in our mathematical framework.
\end{example}

\begin{example}\label{SDE example}
Let $\X=[0,1]^d$ and let $\mu $ be the Lebesgue measure on $\X$. Consider the SDE \eqref{SDE} with drift and diffusion terms
\begin{equation*}
		      \bm{b}(\bm{x})=\left[\begin{array}{c}
				      4 x_1-4 x_1^3 \\
				      -2 x_2
			      \end{array}\right], \quad  \bm{\sigma(x)}=\left[\begin{array}{cc}
				      0.7 & x_1 \\
				      0   & 0.5
			      \end{array}\right],
	      \end{equation*}
and absorbing boundaries. Since $\bm{\Sigma} = \b{\sigma }\b{\sigma }^\top$ is elliptic, the Fokker--Planck equation is well-posed. As a result, the push-forward $ (\Phi^t)^* \mu $ has a density, so $\Phi $ preserves sets of Lebesgue measure zero and $\Kk^t$ is well defined on $\Ff $. Applying typical energy estimates for second-order PDEs and the Hille--Yosida theorem, $\Ll $ is closed, see \cite[Theorem 5]{evans2022partial}. As a consequence, we may take $\Aa = \Kk^{t}$ for some fixed $t$  and $\Aa = \Ll $.
\end{example}

In practice, there are two main cases. First, if $\Aa$ is a bounded operator on $\Ff$ ({e.g., the Koopman operator $\Kk^t$ provided $\Phi$ preserves $\mu$-null sets}), then it is defined on all of $\Ff$ and closed, so $\Dd=\Ff$, and the norms $\norm{\cdot}_\Dd$ and $\norm{\cdot}_\Ff$ are equivalent. Second, $\Aa$ may be an unbounded operator, such as the infinitesimal generator $\Ll$ of a semigroup (such as in \eqref{Koopman generator}). {For $\Ll$ to be a closed operator, it is sufficient that the associated semigroup $\Kk^t$ is strongly continuous; standard semigroup theory dictates that the generator of any strongly continuous semigroup is densely defined and closed.} In this case, $\Dd$ is a dense proper subspace of $\Ff$, see \cite{evans2022partial} for details.

A common example is where \red{$\X$ is an open subset of $\R^d$} and $\Dd$ is a weighted Sobolev space $H^k(\X, \mu )$ \red{supplemented by appropriate boundary conditions}~\cite{klus_data-driven_2020}. \red{Specifically, if $\mu$ is absolutely continuous with respect to the Lebesgue measure with a locally integrable density, the weighted Sobolev space $H^k(\X,\mu)$ is defined as the space of functions in $L^2(\X,\mu)$ whose classical weak derivatives $D^\alpha$ up to order $k$ remain in $L^2(\X,\mu)$,}
equipped with the norm \cite{turesson2000nonlinear, gol2009weighted}
\begin{equation*}
	\left\|\psi \right\|_{H^k(\X, \mu )}=\sum_{|\alpha|\leq k}\left(\int_{\X}\left|D^\alpha \psi\right|^2 \d \mu \right)^{\nicefrac{1}{2}}.
\end{equation*}
For example, ignoring boundary conditions, let $\mu$ be absolutely continuous with respect to the Lebesgue measure and let $\Aa$ be the generator in \eqref{Koopman generator}, then $\Dd= H^2(\X, \mu)$. If the system is deterministic, on the other hand, then $\Dd=H^1(\X, \mu)$.

{We emphasize that this is strictly a concrete example for differential operators on Euclidean domains. For systems with singular measures (e.g., supported on attractors), classical weak derivatives are undefined, and the domain $\Dd$ is instead defined abstractly via the generator. Provided the Koopman semigroup is strongly continuous on $L^2(\X, \mu)$, standard semigroup theory guarantees that this abstract domain $\Dd$ is dense in $L^2(\X, \mu)$, ensuring our general framework applies without requiring absolute continuity.}

Our goal is to study a given dynamical system through the action of $\Aa $ on observables $f \in \Dd $. A practical difficulty is that we will typically not know the value of $\Aa f(\bm{x})$ for every observable $f$ and every point $\bm{x} \in \mathbb{X}$, as this would correspond to having an infinite amount of information about our system. The most we can hope for is to know this information for a finite amount of observables ${\psi_1, \dots, \psi_N}\subset \Dd$, called the \emph{dictionary}, and a finite amount of training data points $\bm{x}_1, \dots, \bm{x}_M$, which are assumed to be sampled independently from some arbitrary probability measure $\mu $. That is, our total information is
\begin{equation}\label{info}
	\big\{\psi_n(\bm{x}_m), \Aa \psi_n (\bm{x}_m)\big\}_{m,n=1}^{M,N}.
\end{equation}
Here, the terms $\Aa \psi_n (\bm{x}_m)$ may be evaluated exactly or approximated using existing techniques (e.g., finite differences, Kramers--Moyal expansions, etc.), see \cite{klus_data-driven_2020,brunton2016discovering}. Since $\Aa$ is infinite-dimensional, this finite amount of information cannot fully describe $\Aa$ except in simple cases. In general, the best we can hope is to obtain some approximation  $\wh{\Aa}$ of ${\Aa}$.

Many approaches for the data-driven description of dynamics, such as DMD, EDMD, and gEDMD, can be described as trying to obtain the best approximation $\wh{\Aa}$ of $\Aa$ given the limited data in \eqref{info}. In fact, all of them can be subsumed under the framework that we will introduce below.

There are two aspects that need to be considered: a finite number of dictionary functions $\psi_1,\ldots,\psi_N$ that may not span the space $\mathcal{F}$ and a finite number of samples $\bm{x}_1,\ldots,\bm{x}_M$ that are not space-filling in $\mathbb{X}$. We now discuss the implications of these aspects, starting with the finite dictionary.
Using the information in \eqref{info}, the operator $\Aa$ can be approximated on
\begin{equation*}
	\Ff_N := \mspan(\Psi) = \mspan(\psi_1, \dots, \psi_N)
\end{equation*}
by taking sampling-based approximations of the projection of $\mathcal{A}$ onto $\Ff_N$. In all the preceding methods, this is done as follows: We would like our finite-dimensional approximation $\Aa_N\colon \Ff_N \to \Ff_N$ to $\Aa$ to satisfy
\begin{equation} \label{Galerkin matrices}
	\big[\bm{C}_N\big]_{ij} :=\br{\Aa \psi_i, \psi_j}_{L^2(\mu )}=\br{\Aa_N\psi_i, \psi_j}_{L^2(\mu )}= \qty[\qty(\bm{A}^\Psi_N)^\top \bm{G}_N]_{ij},
\end{equation}
for all $i,j$, where
\begin{equation*}
	\big[\bm{C}_N\big]_{ij} :=\br{\Aa \psi_i, \psi_j}_{L^2(\mu )},\quad	\big[\bm{G}_N\big]_{ij} := \br{\psi_i, \psi_j}_{L^2(\mu )}
\end{equation*}
are the \emph{structure matrix} and the \emph{Gram matrix} of $\Aa$ and $\Psi$, respectively. If we could evaluate the above integrals exactly, we would obtain our approximation via
\begin{equation} \label{exact Galerkin}
	\big(\bm{A}_N^\Psi\big)^\top = \bm{C}_N\bm{G}_N^{-1},
\end{equation}
where the invertibility of the Gram matrix is equivalent to $\Psi$ being linearly independent. The operator $\mathcal{A}_N\colon\Ff_N\to\Ff_N$ defined through \eqref{exact Galerkin} satisfies \eqref{Galerkin matrices} and, as a result, is necessarily the projection of $\Aa$ onto $\Ff_N$. That is, if we write $\Pp_{\Ff_N}$ for the projection onto $\Ff_N$,
\begin{equation} \label{Galerkin operator}
	\Aa_N=\restr{\Pp_{\Ff_N} \Aa}{\Ff_N} .
\end{equation}
Due to the connection with the finite element method (FEM), \eqref{Galerkin operator} is often called the \emph{Galerkin approximation} (or projection) of $\Aa$. \red{This is also known in some contexts as the finite section of $\Aa$.}

However, since we only have access to the information in \eqref{info}, the best we can do is to use the samples $\bm{x}_1,\ldots,\bm{x}_M \sim \mu$ and define the \emph{empirical structure matrix} and the \emph{empirical Gram matrix} as
\begin{equation} \label{approximate matrices}
	\big[\wh{\bm{C}}_{NM}\big]_{ij} := \frac{1}{M} \sum_{m=1}^M \overline{\psi_j(\bm{x}_m)}\Aa\psi_i (\bm{x}_m), \quad \big[\wh{\bm{G}}_{NM}\big]_{ij} := \frac{1}{M}\sum_{m=1}^{M} \psi_i(\bm{x}_m) \overline{\psi_j(\bm{x}_m)},
\end{equation}
respectively.
Let us denote the \emph{empirical measure} associated with $\{\bm{x}_1,\ldots,\bm{x}_M\}$ by
\begin{equation*}
	\wh{\mu}_M := \frac{1}{M} \sum_{m=1}^M \delta _{\bm{x}_m},
\end{equation*}
then the above can also be written as
\begin{equation*}
	\big[\wh{\bm{C}}_{NM}\big]_{ij}= \br{\Aa \psi_i, \psi_j}_{L^2(\wh{ \mu }_M)} , \quad \big[\wh{\bm{G}}_{NM}\big]_{ij}= \br{\psi_i, \psi_j}_{L^2(\wh{ \mu }_M)} .
\end{equation*}
We define, analogously to \eqref{exact Galerkin},
\begin{equation} \label{approximate Galerkin}
	\wh{\bm{A}}_{NM}^\top := \wh{\bm{C}}_{NM}\wh{\bm{G}}_{NM}^+.
\end{equation}
Here, $\wh{\bm{G}}_{NM}^+$ denotes the \emph{Moore--Penrose pseudoinverse} \cite{Penrose} of $\wh{\bm{G}}_{NM}$ and by construction is such that \eqref{approximate Galerkin} minimises the empirical error,
\begin{equation*}
	\norm{\wh{\bm{A}}_{NM}^\top\wh{\bm{G}}_{NM} - \wh{\bm{C}}_{NM}}^2_F := \sum_{i,j=1}^N \abs{\big[\wh{\bm{A}}_{NM}^\top\wh{\bm{G}}_{NM}\big]_{ij} - \big[\wh{\bm{C}}_{NM}\big]_{ij}}^2.
\end{equation*}
The Gram matrix $\bm{G}_N$ is always invertible as the basis functions contained in $\Psi$ are assumed to be linearly independent. However, its Monte Carlo approximation $\wh{\bm{G}}_{NM}$ may not be invertible.  For this reason, the pseudoinverse is required and ensures that we obtain the matrix $\wh{\bm{A}}_{NM} \in \C^{N\times N}$. This can give rise to theoretical issues, such as the discontinuity of the map from a matrix to its pseudoinverse. These problems can be resolved with careful analysis. We will touch upon this in more detail in Theorem \ref{projection theorem}.

By the strong law of large numbers and the definitions in \eqref{Galerkin matrices} and \eqref{approximate matrices}, almost surely
\begin{equation*}
	\lim_{M \to \infty}\wh{\bm{C}}_{NM} = \bm{C}_{N}, \qquad
	\lim_{M \to \infty}\wh{\bm{G}}_{NM} = \bm{G}_N .
\end{equation*}
As a result, we expect $\wh{\bm{A}}_{NM}$ to converge to $\bm{A}^\Psi_{N}$ as the number of data points goes to infinity. This will be studied in Section~\ref{convergence data section}. Convergence to $\Aa$, when the number of basis functions goes to infinity, is studied in Sections~\ref{convergence Galerkin section} and \ref{Joint convergence section}, where also error bounds are established. Convergence to the spectrum of $\Aa $ is studied in Section \ref{eigen section}.  In practice, one may not have access to the exact values of $ \psi_n(\bm{x}_m),\Aa \psi_n(\bm{x}_m)$, but only to an approximation or a noisy measurement. To deal with this case, we also study in Subsection \ref{Noise section} the case where our data has some noise $\bm{\eta}, \bm{\xi}$, i.e., we have access to the mappings
\begin{equation*}
	\b{x}_m \to \psi_n(\bm{x}_m)+\eta_N^{m,n}, \quad \b{x}_m \to \Aa \psi_n(\bm{x}_m)+\xi_N^{m,n}, \quad n=1,\dots,N.
\end{equation*}
\red{We show how a direct implementation in this case introduces some bias, whereas by splitting the evaluations of the basis functions into two batches, convergence of the approximations to the true dynamics is obtained.}

We stress once more that all the previously mentioned methods for data-driven recovery of dynamics (DMD, EDMD, and gEDMD) fall into this framework and are thus covered by the analysis. Here, we use minimal assumptions. More precisely, we require that the observables be linearly independent, sufficiently regular and that the data be i.i.d.\ (Assumptions \ref{assumption} and \ref{moments assumption}). To show convergence and error bounds, it is also necessary that, as $N$ goes to infinity, the basis functions $\psi_1,\dots,\psi_N$ cover the whole space (Assumption \ref{projection assumption} and \ref{projection operator assumption}) and are bounded (Assumption \ref{bn assumption}). Finally, for the case where the measurements are noisy, we assume that the noise is centred at zero and independent (Assumption \ref{noise assumption}).

\section{Data-driven approximation as a projection} \label{convergence data section}
The following section is inspired by \cite{korda_convergence_2018}. However, in our analysis we do not require the \emph{empirical Gram matrix} $\wh{\bm{G}}_{NM}$  to be invertible. We discuss natural situations in which $\wh{\bm{G}}_{NM}$ may not be invertible throughout this section. We begin by imposing the basic assumptions that will be used in what follows.

\begin{assumption} \label{assumption}
	We assume the following:
	\begin{enumerate}[label=(\alph*), ref=\theassumption(\alph*)]  \setlength{\itemsep}{0ex}
		\item The basis functions $\Psi=\set{\psi_1,\dots,\psi_N}\subset \Dd$ are linearly independent. \label{a1}
		\item The functions $\{\psi_n, \Aa\psi_n\}_{n=1}^N $ are continuous $\mu $ almost everywhere. \label{continuous}
		\item The points $\set{\bm{x}_m}_{m=1}^M\subset \X$ are i.i.d.\ samples from $\mu$. \label{a3}
	\end{enumerate}
\end{assumption}
Point $\ref{a1}$ of Assumption \ref{assumption} is necessary to ensure that the Gram matrix $\bm{G}_N$ is invertible.  Point \ref{continuous} is required so that the pointwise evaluation in the Monte Carlo approximations \eqref{approximate matrices} is well-defined. Lastly, Point~\ref{a3} is the basic assumption underlying Monte Carlo approximations.

\begin{observation}\label{smoothness observation}
	In the case where $\Aa$ is a differential operator of order $k$ and $\X$ is an open subset of $\R^d$ with uniformly Lipschitz boundary (for example if $\X=\R^d$ or if $\X$ is any open subset with Lipschitz continuous boundary), the continuity of \ref{continuous} requires continuity of the derivatives of order $k$ of the observables. By Sobolev embedding, a sufficient condition is $\Psi \subset H^{s}(\X,\mu)$ for $s>k+d/2$ (see \cite[Section 12.3]{leoni2024first}).    For example, if $\Aa $ is the generator $\Ll$ of the Koopman operator defined in~\eqref{Koopman generator}, then, in general, we require that the second derivatives of $\Psi$ be continuous almost everywhere. However, if the stochastic dynamics in \eqref{SDE} are reversible with respect to $\mu$ (for example, if $\mu$ is the Gibbs distribution), given smooth $\psi, \varphi$, we have
	\begin{equation*}
		\br{\Aa \psi, \varphi}_{L^2(\mu )}=-\frac{1}{2} \br{\bm{\Sigma} \nabla \psi, \nabla \varphi} _{L^2(\mu)}
	\end{equation*}
	as shown in \cite{zhang_effective_2016}. As a result, we can define
	\begin{equation*}
		\big[\bm{C}_{N}\big]_{i j} := -\frac{1}{2} \br{\bm{\Sigma} \nabla \psi_i, \nabla \psi_j} _{L^2(\mu)}, \quad
		\big[\wh{\bm{C}}_{NM}\big]_{ij} := -\frac{1}{2} \br{\bm{\Sigma} \nabla \psi_i, \nabla \psi_j} _{L^2(\wh{\mu }_M)}.
	\end{equation*}
	This amounts to an integration by parts and can also be carried out when $\mu $ is the Lebesgue measure. Consequently, we then only need the first derivatives of $\Psi$ to be continuous almost everywhere. This is useful if one wants to use piecewise linear basis functions as in FEM, see Section \ref{simulations section} and, for example,~\cite{zhang2023quantitative}.
\end{observation}

As we will see in Theorem \ref{projection theorem}, the approximation $\wh{\bm{A}}_{NM}$ is best defined on the \emph{empirical spaces}
\begin{equation*}
	\wh{\Ff}_{M} := L^2(\X\to\mathbb{C},\wh{\mu}_M), \quad \red{\wh{\Dd }_M:=  \wh{ \Ff} _M \times \wh{\Ff} _M}.
\end{equation*}
Elements in $\Ff$ are not in $\wh{\Ff}_M$ as functions defined $\mu $ almost everywhere are not in general well-defined  $\wh{\mu }_M $ almost everywhere. However, if $\phi \in \Ff $ is continuous almost everywhere, we can view $\phi $  as an element of $\wh{\Ff}_M$ through the (non-injective) mapping
\begin{equation} \label{identification 0}
	\phi \mapsto \phi^{\red{\wh{\Ff }}} := \sum_{m=1}^M \phi(\bm{x}_m)\delta_{\bm{x}_m}\in \wh{\Ff }_M.
\end{equation}
Likewise, given $\psi \in \Dd $ such that $\Aa  \psi$   
is continuous almost everywhere, we can view $\psi $  as an element of $\red{\wh{\Dd}}_M$ through 
\begin{equation} \label{identification}
	\psi \mapsto \psi^{\red{\wh{\Dd }}} := \sum_{m=1}^M \qty(\psi(\bm{x}_m), \red{\Aa \psi (\b{x}_m)})\delta_{\bm{x}_m}\in \red{\wh{\Dd}}_M.
\end{equation}
We also use the notation
\begin{align*}	\Psi^{\red{\wh{\Ff}} } &:= \{\psi_n^{\red{\wh{\Ff}}}\}_{n=1}^N \subset \wh{\Ff}_M, \quad \red{\wh{\Ff}}_{NM} := \mspan(\Psi^{\red{\wh{\Ff}}})\subset \wh{\Ff}_{M}, \\ 
    \red{\Psi^{\wh{\Dd}} } &:= \{\psi_n^{\wh{\Dd}}\}_{n=1}^N \subset \wh{\Dd}_M, \quad \wh{\Dd}_{NM} := \mspan(\Psi^{\wh{\Dd}})\subset \wh{\Dd}_{M}
\end{align*}
to denote the basis $\Psi$ and subspace $\Ff_N$ when viewed as objects in $\wh{\Ff}_M$ \red{and $\wh{\Dd}_M$} using \eqref{identification 0} and \eqref{identification}, respectively. We denote by $\Aa^{\red{\mathrm{emp}}}$  the operator induced by $\Aa$, i.e.,
\begin{equation}\label{induced operator}
\red{\Aa^{\red{\mathrm{emp}}} \colon \wh{\Dd }_M \to \wh{\Ff }_M}, \quad   	\Aa^{\red{\mathrm{emp}}}\psi^{\red{\wh{\Dd }}} := \qty(\Aa \psi)^{\red{\wh{\Ff }}}.
\end{equation}
\red{There may be various functions $\psi \in \Dd $ with the same empirical representative  $\psi ^{\wh{\Dd } }$. However, by construction of ${\cdot }^{\wh{\Dd }}  $, all of them map to the same value $\qty(\Aa \psi)^{\wh{\Ff }} $. As a result, $\Aa^{\red{\mathrm{emp}}} $ is well defined.}

The functions $\psi^{\red{\wh{\Dd}} }$ and $(\Aa\psi)^{\red{\wh{\Ff}} }$ are only well-defined if $\psi$ and $\Aa\psi$ are continuous almost everywhere. In our case, this is satisfied for all $\psi\in \Ff_N$ by Point \ref{continuous} of Assumption \ref{assumption}. Using this notation, the Monte Carlo approximation of the structure and Gram matrices in \eqref{approximate matrices} become
\begin{equation} \label{approximate matrices prod}
	\big[\wh{\bm{C}}_{NM}\big]_{ij}=\br{\Aa^{\red{\mathrm{emp}}} \psi^{\red{\wh{\Dd}} }_i, \psi_j^{\red{\wh{\Ff}} }}_{L^2(\wh{\mu }_M)}, \quad \big[\wh{\bm{G}}_{NM}\big]_{ij}=\br{\psi_i^{\red{\wh{\Ff}} }, \psi_j^{\red{\wh{\Ff}} }}_{L^2(\wh{\mu }_M)}.
\end{equation}
Here, the Gram matrix $\wh{\bm{G}}_{NM}$ with respect to $\Psi^{\red{\wh{\Ff}} }$ is in general \emph{not} invertible as $\Psi^{\red{\wh{\Ff}} }$ may no longer be linearly independent (consider for example the case $M=1$ and $N=2$). By construction, $\wh{\Ff}_M$ is a Hilbert space, so we can define the projection of $\wh{\Ff}_M$ onto $\wh{\Ff}_{NM}$ by
\begin{equation*}
	\Pp_{\wh{\Ff}_{NM}}\colon\wh{\Ff}_M\to \wh{\Ff}_{NM}.
\end{equation*}
We prove that our data-driven approximation $\wh{\bm{A}}_{NM}$ corresponds to the projection of $\Aa^{\red{\mathrm{emp}}}$ onto $\wh{\Ff}_{NM}$.

\begin{theorem}[Empirical projection] \label{projection theorem}
	Let $\Psi $ satisfy Assumption \ref{continuous}. Then the matrix $\wh{\bm{A}}_{NM}$ which approximates $\Aa^{\red{\mathrm{emp}}}$ is, with probability $1$, a matrix representation of the projection of $\Aa^{\red{\mathrm{emp}}}$ onto $\wh{\Ff}_{NM}$. That is,
	\begin{equation*}
		\wh{\Aa}_{NM}^{\red{\mathrm{emp}}}=\restr{\Pp_{\wh{\Ff}_{NM}}\Aa^{\red{\mathrm{emp}}}}{\red{\wh{\Dd}}_{NM}},
	\end{equation*}
	where $\wh{\Aa}_{NM}^{\red{\mathrm{emp}}}: \wh{\Dd }_{NM} \to \wh{\Ff }_{NM}  $ is the operator that has matrix  $\wh{\bm{A}}_{NM}$ in the basis $\Psi^{\red{\wh{\Dd}}}, \Psi^{\red{\wh{\Ff}}}$.
\end{theorem}

\begin{proof}
	By Assumption \ref{continuous}, with probability $1$, the matrices $\wh{\bm{C}}_{NM}$ and $\wh{\bm{G}}_{NM}$ in \eqref{approximate matrices prod}, and thus $\wh{\bm{A}}_{NM}$ in \eqref{approximate Galerkin}, are well-defined. By construction, the data-driven approximation $\wh{\bm{A}}_{NM}$ minimizes the empirical error, i.e.,
	\begin{equation} \label{gEDMD error}
		\wh{\bm{A}}_{NM} \in \argn_{\wh{\bm{A}} \in \C^{N\times N}} \norm{\wh{\bm{C}}_{NM}-{\wh{\bm{A}}}^\top \wh{\bm{G}}_{NM}}_F.
	\end{equation}
	Now, by definition, we have
	\begin{equation*}
		\restr{\Pp_{\wh{\Ff}_{NM}}\Aa^{\red{\mathrm{emp}}}}{\red{\wh{\Dd}}_{NM}}\colon\red{\wh{\Dd}}_{NM}\to \wh{\Ff}_{NM}.
	\end{equation*}
	Furthermore, using basic properties of the projection, for all $ i,j\in \{1,\dots, N\}$, we have
	\begin{equation*}
		\big[\wh{\bm{C}}_{NM}\big]_{ij} := \br{\Aa^{\red{\mathrm{emp}}} \psi^{\red{\wh{\Dd}} }_i, \psi^{\red{\wh{\Ff}} }_j}_{L^2(\wh{\mu }_M)}=\br{\restr{\Pp_{\wh{\Ff}_{NM}}\Aa^{\red{\mathrm{emp}}}}{\red{\wh{\Dd}}_{NM}}\psi_i^{\red{\wh{\Dd}} }, \psi_j^{\red{\wh{\Ff}} }}_{L^2(\wh{\mu }_M)}.
	\end{equation*}
	Equivalently, any matrix representation $\Big(\restr{\Pp_{\wh{\Ff}_{NM}}\Aa^{\red{\mathrm{emp}}}}{\red{\wh{\Dd}_{NM}}}\Big)^{\red{\Psi^{\wh{\Dd } } \to  \Psi ^{\wh{\Ff } }}} \in \C^{N\times N}$ of $\restr{\Pp_{\wh{\Ff}_{NM}}\Aa^{\red{\mathrm{emp}}}}{\red{\wh{\Dd}_{NM}}}$ in the basis $\Psi ^{\wh{\Dd } }, \Psi ^{\wh{\Ff } }$  satisfies \red{(we recall the notation introduced in Section \ref{notation section})}
	\begin{equation} \label{0 error}
		\norm{\wh{\bm{C}}_{NM} - \qty{ \Big(\restr{\Pp_{\wh{\Ff}_{NM}}\Aa^{\red{\mathrm{emp}}}}{\wh{\Dd}_{NM}}\Big)^{\wh{\Psi}}}^\top \wh{\bm{G}}_{NM}}_F=0.
	\end{equation}
	From \eqref{gEDMD error} and \eqref{0 error} we deduce that $\wh{\bm{C}}_{NM}={\wh{\bm{A}}_{NM}}^\top \wh{\bm{G}}_{NM}$. That is,
	\begin{equation} \label{equality}
		\br{\Aa^{\red{\mathrm{emp}}} \psi^{\red{\wh{\Dd}} }_i, \psi^{\red{\wh{\Ff}} }_j}_{L^2(\wh{\mu }_M)} = \big[\wh{\bm{A}}_{NM}^\top \wh{\bm{G}}_{NM}\big]_{ij}, \quad\forall i,j\in\set{1, \dots,N}.
	\end{equation}
	Define the data-driven operator on the empirical space $\wh{\Aa}_{NM}^{\red{\mathrm{emp}}} \colon \wh{\Dd}_{NM} \to \wh{\Ff}_{NM}$ as
	\begin{equation*}
		\wh{\Aa}_{NM}^{\red{\mathrm{emp}}} v := \sum_{i,k=1}^N c_i \big[\wh{\bm{A}}_{NM}\big]_{ki} \psi_k^{\red{\wh{\Dd } }}, \quad\forall v = \sum_{i=1}^Nc _i \psi_i ^{\red{\wh{\Dd } }}\in \wh{\Dd}_{NM} .
	\end{equation*}
	We use Lemma \ref{operato lemma} to see that $\wh{\Aa}_{NM}^{\red{\mathrm{emp}}}$ is well-defined. To this end, consider $(c_1, \dots, c_N) \in \C^N$ such that $\sum_{i=1}^N c_i \psi_i^{\red{\wh{\Dd } }} = 0$. Then, for $j\in\set{1, \dots, N}$, we have
	\begin{align*}
		&\sum_{i,k=1}^N \br{c_i \big[\wh{\bm{A}}_{NM}\big]_{ki} \psi_k^{\red{\wh{\Dd } }}, \psi_j^{\red{\wh{\Ff } }}}_{L^2(\wh{\mu }_M)} = \sum_{i,k=1}^N c_i \big[\wh{\bm{A}}_{NM}\big]_{ki}\big[\wh{\bm{G}}_{NM}\big]_{kj} =\sum_{i=1}^Nc_i[\wh{\bm{A}}_{NM}^\top \wh{\bm{G}}_{NM}]_{ij}
		\\&=\sum_{i=1}^Nc_i \br{\Aa^{\red{\mathrm{emp}}} \psi_i^{\red{\wh{\Dd } }}, \psi_j^{\red{\wh{\Ff } }}}_{L^2(\wh{\mu }_M)}                                                                                 =\br{\Aa^{\red{\mathrm{emp}}} \sum_{i=1}^Nc_i \psi_i^{\red{\wh{\Dd } }}, \psi_j^{\red{\wh{\Ff } }}}_{L^2(\wh{\mu}_M)}=0,
	\end{align*}
	where in the first equality, we used the definition of $\wh{\bm{G}}_{NM}$, in the third, we used \eqref{equality}, and in the last, we used that $\sum_{i=1}^Nc_i\psi_i^{\red{\wh{\Dd } }}=0$. Since $j$ was arbitrary, $\wh{\Aa}_{NM}^{\red{\mathrm{emp}}}$ is a well-defined operator with matrix representation $\wh{\bm{A}}_{NM}$ by Lemma \ref{operato lemma}. Furthermore, by construction of $\wh{\Aa}_{NM}^{\red{\mathrm{emp}}}$  and \eqref{equality}, $\wh{\Aa}_{NM}^{\red{\mathrm{emp}}}$ satisfies
	\begin{equation} \label{equality 2}
		\br{\Aa^{\red{\mathrm{emp}}}\psi_i^{\red{\wh{\Dd } }}, \psi_j^{\red{\wh{\Ff } }}}_{L^2(\wh{\mu }_M)}=\br{\wh{\Aa}_{NM}^{\red{\mathrm{emp}}}\psi_i^{\red{\wh{\Dd } }}, \psi_j^{\red{\wh{\Ff } }}}_{L^2(\wh{\mu }_M)}, \quad\forall i,j=1,\dots,N.
	\end{equation}
	Due to the general theory of Hilbert spaces, the only operator $\wh{\Aa}_{NM}^{\red{\mathrm{emp}}}\colon\red{\wh{\Dd}}_{NM} \to \wh{\Ff}_{NM}$ satisfying \eqref{equality 2} is $\restr{\Pp_{\wh{\Ff}_{NM}}\Aa^{\red{\mathrm{emp}}}}{\red{\wh{\Dd}}_{NM}}$. As a result, $\wh{\Aa}_{NM}^{\red{\mathrm{emp}}}=\restr{\Pp_{\wh{\Ff}_{NM}}\Aa^{\red{\mathrm{emp}}}}{\red{\wh{\Dd}}_{NM}}$, which completes the proof.
\end{proof}
We now show that if $\Ff _N$ is invariant under the action of $\Aa$, then $\wh{\Aa}_{NM}^{\red{\mathrm{emp}}}$ is an exact approximation of the Galerkin projection $\Aa_N$ defined in \eqref{Galerkin operator} of $\Aa $  with probability $1$. \red{As in \eqref{induced operator}, we define
\begin{align*}
	\Aa _N^{\red{\mathrm{emp}}} : \wh{\Dd }_{NM} \to \wh{\Ff }_{NM}, \quad  \psi^{\wh{\Dd } } \to (\Aa_N \psi)^{\wh{\Ff } }.
\end{align*}

}

\begin{corollary}[Exact approximation] \label{exact approximation}
	If  $\Aa \Ff _N \subset \Ff_N$, then, with probability $1$,
	\begin{equation*}
		\wh{\Aa }_{NM}^{\red{\mathrm{emp}}} = \Aa_N^{\red{\mathrm{emp}}}.
	\end{equation*}
	Let $\bm{A}_N^{\Psi}$ be the matrix of $\Aa_N=\restr{\Aa }{\Ff_N}$ in the basis $\Psi$. If additionally, $\wh{\bm{G}}_{NM}$ is invertible, then with the notation of Section \ref{notation section}
	\begin{equation*}
		\wh{\bm{A}}_{NM}=\bm{A}_N^{\Psi}.
	\end{equation*}
\end{corollary}
\begin{proof}
	Taking $\Aa _N= \restr{\Aa }{\Ff_N}$  in place of $\Aa$ in Theorem \ref{projection theorem} and using that $\Aa_N^{\red{\mathrm{emp}}} \colon \wh{\Dd}_{NM}\to \wh{\Ff}_{NM}$,
	\begin{equation*}
		\wh{\Aa}_{NM}^{\red{\mathrm{emp}}} =\wh{\Aa_N}_{NM}=  \restr{\Pp_{\wh{\Ff}_{NM}}\Aa_N^{\red{\mathrm{emp}}}}{\wh{\Ff}_{NM}} = \Aa_N^{\red{\mathrm{emp}}}.
	\end{equation*}
	This proves the first part of the corollary. To see the second, note that if $\wh{\bm{G}}_{NM}$ is invertible, then $\Psi^{\wh{\Dd }},\Psi^{\wh{\Ff }}$ form a basis of $\wh{\Dd }_{NM}$ and $\wh{\Ff }_{NM}$ respectively and matrix representations of operators from $\wh{\Dd}_{NM}$ to $\wh{\Ff }_{NM}$  and endomorphisms of $\Ff_N$ are the same. As a result, $\wh{\bm{A}}_{NM}=\bm{A}^\Psi_N$. This completes the proof.
\end{proof}

Theorem \ref{projection theorem} is a generalisation of Theorem 1 in \cite{korda_convergence_2018}, where it was also required that $\Psi^{\red{\wh{\Ff } }}$ be linearly independent in $\wh{\Ff}_M$ (or equivalently, that $\wh{\bm{G}}_{NM}$ be invertible). However, this assumption is not altogether benign. The assumption will always fail when $M< N$ (as $\wh{\Ff}_M$ has dimension $M$) and also in many cases of interest, such as the following example.

\begin{example}[Finite element basis of degree $k$] \label{FEM example}
	We want to approximate $\mathcal{F}$ by the space of piecewise polynomials of degree up to $k$. That is, we decompose $\X$ into disjoint subdomains $\set{T_1,\dots,T_L}$, i.e.,
	\begin{equation*}
		\X=\bigsqcup_{l=1}^L T_l,
	\end{equation*}
	and take
	\begin{equation*}
		\mathcal{F}_{N}=\left\{v \in C(\X), ~ \restr{v}{T_l} \in P_k, \forall l \in \set{1,\dots,L}\right\},
	\end{equation*}
	where $P_k$ is the set of polynomials of degree up to $k$ in each variable. This space has dimension $N=Lk^d$.
	A basis $\Psi$ can be obtained by taking $k^d$ nodes in each $T_l$ to obtain a total of $N$ nodes $\{\bm{a}_1,\dots,\bm{a}_N\}$ and taking the unique functions $\Psi=\set{\psi_n}_{n=1}^N\subset \Ff_N$ such that
	\begin{equation*}
		\psi_i(\bm{a}_j)=\delta_{ij}, \quad\forall i,j\in \set{1,\dots,N}.
	\end{equation*}
	Here, typically, $L$ is chosen so that each $T_l$ has a diameter at most $h$. In this case, $L=\mathcal{O}(h^{-d})$ and $N=\mathcal{O}\qty((k/h)^{d})$. \exampleSymbol
\end{example}

In the above example, if there is a cell $T_{l_0}$ which holds no data point, then $\wh{\psi}_{l_0}=0$. As a result, $\wh{\Psi}$ will not be linearly independent in general.
We now study the convergence of $\wh{\Aa}_{NM}$ to the projection of $\Aa$ onto $\Ff_N$. This convergence uses the law of large numbers and requires that the empirical matrices have finite variances.
\begin{assumption} \label{moments assumption}
	The basis functions $\Psi=\set{\psi_n}_{n=1}^N$ satisfy $\psi_j\Aa\psi_i\in \Ff$ as well as $\psi_i^2\in \Ff$ for all $i,j\in \set{1,\dots,N}$.
\end{assumption}
For example, the above assumption will hold by Assumption \ref{assumption} if $\X$ is a compact domain. Though $\wh{\Psi}$ are not linearly independent in general, they will be (almost surely) when we have a large enough training data set. This is shown in the following lemma.

\begin{lemma}
	\label{invertiblelemma}
	Suppose Assumptions \ref{assumption} and \ref{moments assumption} hold, then there exists an $\mathbb{N}$-valued random variable $m_0$ such that $\Psi^{\red{\wh{\Ff }} }\subset\wh{\Ff}_M$ are almost surely linearly independent for all $M\ge m_0$.
\end{lemma}

\begin{proof}
	By Assumption \ref{assumption}, we may apply the strong law of large numbers to deduce that, for all $i,j\in \{1,\dots, N\}$, when $M\to\infty$
	\begin{equation*}
		\big[\wh{\bm{G}}_{NM}\big]_{ij} = \frac{1}{M}\sum_{m=1}^{M} \psi_i(\bm{x}_m) \overline{\psi_j(\bm{x}_m)}~\overset{a.s.}{\to}~ \int \psi_i(\bm{x}) \overline{\psi_j(\bm{x})} \mathrm{d} \mu(\bm{x})=\big[{\bm{G}}_{N}\big]_{ij}.
	\end{equation*}
	Since the determinant is a continuous function and $\bm{G}_N $ is the Gram matrix of a linearly independent set of functions, we have
	\begin{equation*}
		\operatorname{det}({\wh{\bm{G}}_{NM}})\to\operatorname{det}({{\bm{G}}_{N}})\ne 0\quad\text{a.s.\ for $M\to\infty$},
	\end{equation*}
	from which it follows that there exists $m_0\in\mathbb{N}$ such that
	\begin{equation*}
		\abs{\operatorname{det}({\wh{\bm{G}}_{NM}})-\operatorname{det}(\bm{G}_{N})}<\abs{ \operatorname{det}({\bm{G}}_{N})}\quad\text{a.s.\ when }M\ge m_0.
	\end{equation*}
	Consequently,
	\begin{equation*}
		\operatorname{det}({\wh{\bm{G}}_{NM}})\ne 0\quad\text{a.s.\ when } M\ge m_0,
	\end{equation*}
	which proves the result as $\wh{\bm{G}}_{NM}$ is the Gram matrix of $\Psi ^{\wh{\Ff } }$ in $\wh{\Ff }_{M}$.  
\end{proof}

Theorem 3.2 establishes that $\wh{\bm{A}}_{NM}$ represents a projection operator, $\wh{\Aa}_{NM}^{\mathrm{emp}}$, in the discrete space of the data samples. This guarantees our approximation is optimal with respect to the given data.

\red{For convergence and error analysis, however, we must compare our approximation to the true operator $\Aa_N$, which acts on the continuous space $\Ff_N$. We therefore use the empirically derived matrix to define our primary object of analysis.

\begin{definition}
We define the \emph{data-driven operator} $\wh{\Aa }_{NM}\colon \Ff _N \to \Ff _N$ as the unique linear operator whose matrix representation with respect to the basis $\Psi \subset \Ff_N $ is $\wh{\bm{A}}_{NM}$ as defined in \eqref{approximate Galerkin}.
\end{definition}

This definition is our object of interest for the remainder of this paper and is valid for all $M$ and $N$. We now connect it back to the empirical setting.}

\begin{observation}
	\label{ObservationIdentification}
	Lemma \ref{invertiblelemma} shows that, for sufficiently large $M$, \red{$\Psi^{\wh{\Ff }}$} forms a basis of $\wh{\Ff}_{NM}$. As a result, for large $M$,  ${\Ff}_{N}$ is isomorphic to $\wh{\Ff}_{NM}$ and $\wh{\Dd}_{NM}$ through the mapping \eqref{identification 0} and \eqref{identification} respectively. Under this isomorphism, the data-driven operator $\wh{\Aa}_{NM}$ is the continuous-space counterpart to the empirical projection operator $\wh{\Aa}_{NM}^{\mathrm{emp}}$.
\end{observation}

Using Lemma \ref{invertiblelemma}, we now prove that in the infinite sample limit, the data-driven operator $\wh{\Aa}_{NM}$ converges to $\Aa_N$, i.e., the Galerkin projection of $\Aa$ onto $\Ff_N$.

\begin{theorem}[Convergence in data limit] \label{convergence to projection theorem}
	If Assumptions \ref{assumption} and \ref{moments assumption} hold, almost surely,
	\begin{equation*}
		\lim_{M\to\infty}{\wh{\bm{A}}_{NM}}=\bm{A}_N^{{\Psi}},
	\end{equation*}
	where $\bm{A}_N^{{\Psi}}$ is the matrix  of $\Aa_N$ with respect to $\Psi$.
\end{theorem}

\begin{proof}
	Let $m_0$ be as in Lemma \ref{invertiblelemma}. For $M>m_0$, since $\wh{\bm{G}}_{NM}$ is invertible, its pseudoinverse is equal to its inverse, and, by the strong law of large numbers, we have
	\begin{equation*}
		\qty({\wh{\bm{A}}_{NM}})^\top=\wh{\bm{C}}_{NM}\wh{\bm{G}}_{NM}^{-1}\overset{a.s.}{\to}\bm{C}_N \bm{G}_N^{-1}=\qty(\bm{A}_N^{{\Psi}})^\top\quad\text{for } M\to\infty.
	\end{equation*}
	In other words, the matrix of $\wh{\Aa}_{NM}$ converges to that of $\Aa_N$.
\end{proof}

\begin{corollary} \label{convergence data corollary}
	If Assumptions \ref{assumption} and \ref{moments assumption} hold, then, with probability $1$,
	\begin{equation*}
		\lim _{M \rightarrow \infty}\left\|\mathcal{\wh{A}}_{NM} - {\Aa_N }\right\|=0,
	\end{equation*}
	where $\|\cdot\|$ is the operator norm. In particular, for all $\psi \in \Ff_N$ and with probability $1$
	\begin{equation*}
		\lim _{M \rightarrow \infty}\left\|\mathcal{\wh{A}}_{NM}\psi -{\Aa_N \psi}\right\|_{\Ff_N}=0,
	\end{equation*}
	where $\norm{\cdot }_{\Ff_N}$ is any norm on $\Ff_N$.
\end{corollary}
\begin{proof}
	The convergence in the operator norm is a direct consequence of the convergence of the matrix representations of Theorem \ref{convergence to projection theorem}. The second point follows from the fact that convergence in the operator norm implies pointwise convergence and that all norms on finite-dimensional spaces are equivalent.
\end{proof}

\section{Convergence of the projections} \label{convergence Galerkin section}

In the previous section, we have shown that the data-driven approximation in \eqref{approximate Galerkin} defines an operator $\wh{\Aa}_{NM}$ that converges to the Galerkin projection $\Aa_N=\restr{ \Pp_{\Ff_N} \Aa}{\Ff_N}$. In this section, our goal is to show that $\Aa_N$ converges to $\Aa$. This was also done in \cite{korda_convergence_2018}, where it was assumed that the (Koopman) operator is bounded on $\Ff$ and that $\set{\psi_n}_{n=1}^\infty$ form an orthonormal basis of $\Ff$. In many cases, however, $\Aa$ will \emph{not} be bounded (for example, if $\Aa$ is the generator $\Ll$ of the Koopman operator). Furthermore, the requirement that the basis be orthonormal is restrictive, and in practice, it is often preferable to work with a dictionary that does not have to be orthonormalised; a finite element basis, for instance, will often produce sparse operators, but it is not orthonormal. Additionally, the sampling measure $\mu $ is typically unknown, so it is not possible to orthonormalise with respect to the norm on $\Ff $.

We will work in a more general setting. Firstly, we require that $\mathcal{F}_N$ converges to $\Ff$ as $N \rightarrow \infty$. That is, as we take more basis functions, we fill $\Ff$.

\begin{assumption} \label{projection assumption}
	Assumption \ref{assumption} holds and
	\begin{equation*}
		\lim_{N \to \infty} \norm{\Pp_{\Ff_N} \phi - \phi}_{\Ff} =0 , \quad\forall \phi \in \Ff,
	\end{equation*}
	where $\Pp_{\Ff_N}$ is the projection of $\Ff$ onto $\Ff_N$ using the inner product on $\Ff$.
\end{assumption}

We will also need to approximate functions in the domain $\Dd$. This necessitates the following assumption.

\begin{assumption} \label{projection operator assumption}
	Assumption \ref{assumption} holds, $\Aa$ is a closed operator, and
	\begin{equation*}
		\lim_{N \to \infty} \norm{\Pp_{\Dd_N} f - f}_{\Dd} = 0 , \quad\forall f \in \Dd.
	\end{equation*}
	Here, $\Pp_{\Dd_N}$ is the projection of $\Dd$ onto $\Ff_N$ using the inner product on $\Dd$.
\end{assumption}

If $\Aa$ is bounded on $\Ff$, then $\Ff=\Dd$ and Assumptions \ref{projection assumption} and \ref{projection operator assumption} are equivalent. In general, though, this is not the case. One assumption may imply that functions are approximated well in $L^2(\X, \mu)$ and the other in $H^r(\X, \mu)$.

\begin{example}[Orthonormal basis]
Let $\set{\psi_n}_{n=1}^\infty$ be an orthonormal basis of $\Ff$ and set $\Psi_N := \set{\psi_{1}, \dots, \psi_{N}} $. Given $f(x) = \sum_{n=1}^{\infty} c_n \psi_n(x) \in \Ff$,
\begin{equation*}
    \norm{\Pp_{\Ff_N} \phi - \phi}_{\Ff}^2=\sum_{n=N+1}^\infty c_n^2 \to 0 \quad (N \rightarrow \infty).
\end{equation*}
If $\Aa$ is continuous on $\Ff$, then $\Dd=\Ff$ so that Assumption~\ref{projection assumption} holds. \exampleSymbol
\end{example}

\red{\begin{example}
Consider the dynamical system in Example~\ref{SDE example} on $\X=[0,1]^d$ with the Lebesgue measure, then $\Dd = H^2(\X)\cap H_0^1(\X)$. Examples of dense bases in $\Dd $:
\begin{itemize}
	\item Cutoff Gaussians $\psi(\bm{x})=\exp\qty(-{\norm{\bm{x}-\bm{p}}^2}/{2 \theta^2})\prod_{i=1}^d x_i\left(1-x_i\right)$, where $\b{p}, \theta$ varies over a dense set in $\X$ and $\theta >0$ is fixed. 
	\item FEM basis functions (e.g., Lagrange) of higher order.
\end{itemize}
\exampleSymbol
\end{example}

To see an example motivated by spectral approximation of the Koopman operator when $\mu $ is invariant, see \cite[Corollary 11]{valva2024physics}.}

The above assumptions only require that the subspaces $\Ff_N$ are good approximations of $\Ff$ and~$\Dd$, with the error vanishing in the limit. We now introduce the notation
\begin{equation*}
	\Ff_\infty := \bigcup_{n=1}^\infty \Ff_n.
\end{equation*}
Under the above assumptions, we can prove the convergence of the Galerkin approximation.

\begin{theorem}[Convergence in dictionary limit] \label{convergence of projection theorem}
	Let $\Psi_N$ satisfy Assumption \ref{projection assumption}, then
	\begin{equation*}
		\lim_{N \to \infty}\norm{\Aa_N\Pp_{\Dd_N} \psi- \Aa \psi}_\Ff=0 , \quad\forall \psi \in \Ff_\infty .
	\end{equation*}
	If additionally Assumption \ref{projection operator assumption} holds, then
	\begin{equation*}
		\lim_{N \to \infty}\norm{\Aa_N \Pp_{\Dd_N} f- \Aa f}_\Ff=0 , \quad\forall f \in \Dd.
	\end{equation*}
\end{theorem}

\begin{proof}
	Using basic algebra and the fact that by definition $\Aa_N = \restr{\Pp_{\Ff_N} \Aa}{\Ff_N}$, we obtain
	\begin{align*}
		\Aa_N \Pp_{\Dd_N} - \Aa & =(\Aa_N - \Aa)\Pp_{\Dd_N} +\Aa \Pp_{\Dd_N} -\Aa                                                           \\
		                        & =(\Pp_{\Ff_N} - \rm{Id})\Aa\Pp_{\Dd_N} +\Aa (\Pp_{\Dd_N} -\rm{Id})                                        \\
		                        & =(\Pp_{\Ff_N} - \rm{Id})\Aa +(\Pp_{\Ff_N} - \rm{Id})\Aa(\Pp_{\Dd_N}-\rm{Id}) +\Aa (\Pp_{\Dd_N} -\rm{Id}).
	\end{align*}
	Consider now $f\in\Dd$. Applying the triangle inequality and the fact that $\Aa$ is continuous on its domain gives
	\begin{equation} \label{convergence}
		\begin{split}
			\norm{\Aa_N \Pp_{\Dd_N} f- \Aa f}_\Ff & \leq \norm{(\Pp_{\Ff_N} - \rm{Id})\Aa f}_\Ff +\norm{(\Pp_{\Ff_N} - \rm{Id})\Aa}\norm{\Pp_{\Dd_N} f-f}_\Dd \\&~~~+\norm{\Aa} \norm{\Pp_{\Dd_N} f-f}_\Dd.
		\end{split}
	\end{equation}
	If $f\in \Ff_\infty$, then $\Pp_{\Dd_N}f=f$ for $N$ large enough, and using Assumption \ref{projection assumption} with $\phi := \Aa f$ proves the first part of the theorem. If $f\in \Dd$, then combining Assumption \ref{projection assumption} and Assumption \ref{projection operator assumption} concludes the proof.
\end{proof}

The first part of Theorem \ref{convergence of projection theorem} is useful in the case where we have a finite-dimensional space of observables we are interested in. In this case, these can be incorporated directly into $\Ff_N$. The second part of the theorem is useful when we want to know the evolution of every possible observable. The proof also shows that the order of convergence depends completely on the rate of convergence of $\mathcal{P}_{\Ff_N} f$ and $\Pp_{\Dd_N}f$ to $f$. We summarise this result in the following corollary.

\begin{corollary} \label{order corollary}
	Consider $\Psi_N$ satisfying Assumption \ref{assumption} and such that $ \norm{\Pp_{\Ff_N} \phi-\phi}_\Ff = \mathcal{O}(N^{-\alpha})
	$ for all $\phi\in \Ff$. Then
	\begin{equation*}
		\norm{\Aa_N \Pp_{\Dd_N}\psi- \Aa \psi}_\Ff=\mathcal{O}(N^{-\alpha}), \quad\forall \psi \in \Ff_\infty.
	\end{equation*}
	If additionally, $\Aa$ is closed and $ \norm{\Pp_{\Dd_N} f-f}_\Dd=\mathcal{O}(N^{-\alpha})$ for all $f\in\Dd$, then
	\begin{equation*}
		\norm{\Aa_N \Pp_{\Dd_N} f - \Aa f}_\Ff=\mathcal{O}(N^{-\alpha}), \quad\forall f\in \Dd,
	\end{equation*}
	where, in both cases, the hidden constant depends only linearly on $\norm{\Aa},\norm{\psi},\norm{f}$.
\end{corollary}

\begin{proof}
	This follows from inequality \eqref{convergence} and the observation that if $f\in \Ff_\infty$ then the second and third terms in this equation are identically zero for large enough $N$.
\end{proof}

\section{Joint limit in data and dictionary} \label{Joint convergence section}

In Sections \ref{convergence data section} and \ref{convergence Galerkin section}, we studied the iterated limits of $\wh{\Aa}_{NM}$ when the size of the training data set $M$ and the number of basis functions $N$ go to infinity. In this section, we study the behaviour of $\wh{\Aa}_{NM}$ when $M$ and $N$ increase simultaneously.

{Some concentration bounds for EDMD were derived in \cite{colbrook2024beyond}.} The projection error was studied for various approximation spaces, such as reproducing kernel Hilbert spaces and those generated by finite-dimensional bases of wavelets, in \cite{kurdila2018koopman}. In \cite{zhang2023quantitative}, the authors work with the generator of an ordinary differential equation and derive a projection error in the context of a finite element basis. They also provide a finite-data error bound on the approximation of the generator in the case where the data is sampled from the Lebesgue measure. In \cite{Nske2021FiniteDataEB}, the authors derive an error bound for the approximation error of gEDMD under the assumptions that the Koopman semigroup is exponentially stable and the points are sampled from a probability measure invariant under the flow and from a single ergodic trajectory. The error bounds in \cite{zhang2023quantitative, Nske2021FiniteDataEB} both require $M$ to be ``sufficiently large'' so that the empirical Gram matrix is invertible. However, no bound on how large $M$ must be is given. In fact, a deterministic bound on $M$ is impossible, but rather a probabilistic one is needed.

We work under relaxed assumptions, for example, we do not impose that the dictionary functions be orthonormal, we do not impose that $\bm{x}_1, \dots, \bm{x}_M$ be sampled from the Lebesgue measure, a measure invariant under the flow or even from the same trajectory, and we do not impose that the empirical Gram matrix be invertible. We first formulate an existence result of the following type:

\begin{theorem}[Convergence in joint limit]
	\label{joint limit theorem}
	Let $\Psi$ satisfy Assumptions \ref{assumption}, \ref{moments assumption}, and \ref{projection assumption}, then there exists a sequence $\set{(N,M_N)}_{N=1}^\infty$ such that for any $M'_N \geq M_N$ almost surely
	\begin{equation*}
		\lim_{N \to \infty}\norm{\wh{\Aa}_{NM'_N} {\psi}- \Aa \psi }_\Ff=0, \quad \forall \psi \in \Ff_\infty.
	\end{equation*}
	If additionally Assumption \ref{projection operator assumption} is satisfied, then almost surely
	\begin{equation*}
		\lim_{N \to \infty}\norm{\wh{\Aa}_{NM'_N}{\Pp_{\Dd_N} f}- \Aa f }_\Ff=0, \quad \forall f \in \Dd.
	\end{equation*}
\end{theorem}

\begin{proof}
	Let $f\in\Dd$ and $\varepsilon >0$. Using the triangle inequality, we have
	\begin{equation} \label{triangle0}
		\norm{\wh{\Aa}_{NM_N}\Pp_{\Dd_N} f- \Aa f }_\Ff \leq\norm{\wh{\Aa}_{NM_N}- \Aa_N }\norm{\Pp_{\Dd_N} f}_\Dd+\norm{\Aa_{N}\Pp_{\Dd_N} f- \Aa f }_\Ff.
	\end{equation}
	By Corollary \ref{convergence data corollary}, for any $N \in \N$, there exists almost surely $M_N$ such that for all $M'_N \geq M_N$
	\begin{equation} \label{ineq 1}
		\norm{\wh{\Aa}_{NM'_N} - \Aa_N} < \frac{1}{2N}.
	\end{equation}
	Let us set $N_0>\varepsilon^{-1}\|f\|_{\Dd}$ and such that for $N\geq N_0$ we have
	\begin{equation} \label{ineq 2}
		\norm{\Aa_{N}\Pp_{\Dd_N} f- \Aa f }_\Ff<\frac{\varepsilon}{2},
	\end{equation}
	where \eqref{ineq 2} is possible due to Theorem \ref{convergence of projection theorem} in both the case $f\in \Ff_\infty$ as well as $f\in \Dd$. As a result, we obtain from \eqref{triangle0}, \eqref{ineq 1}, and \eqref{ineq 2} that for all $N\geq N_0$
	\begin{equation*}
		\norm{\wh{\Aa}_{NM'_N}\Pp_{\Dd_N} f - \Aa f }_\Ff <\frac{\varepsilon}{2}+\frac{\varepsilon}{2}=\varepsilon,
	\end{equation*}
	where it was also used that $\norm{\Pp_{\Dd_N}}=1$. Since $\varepsilon >0$ and $f$ were arbitrary, this proves the theorem.
\end{proof}

This existence result forms a good foundation. In practice, however, we may need an explicit dependence between $M$ and $N$ and explicit error bounds. Before we can proceed, we first need to state additional assumptions.

\begin{assumption}[Boundedness of the basis functions]\label{bn assumption}
	We assume that:
	\begin{enumerate}[label=(\alph*), ref=\theassumption(\alph*)]  \setlength{\itemsep}{0ex}
		\item \label{basis norm assumption} There exists $\gamma_N^{}$ such that $\mu$ almost everywhere $\abs{{\Psi}_N(\bm{x})}^2< \gamma_N^{}$.
		\item \label{basis norm assumption 2} There exists $\gamma_N^{}$ such that $\mu $ almost everywhere $ \abs{\Aa{\Psi}_N(\bm{x})}^2 < \gamma_N^{}$,
	\end{enumerate}
	where in \ref{basis norm assumption 2} the operator is applied componentwise.
\end{assumption}

Let us illustrate these assumptions with two examples.

\begin{example}[Bounded dictionary] \label{infinity example}
	Suppose that for all $i\in\set{1,\dots,N}$ it holds that
	\begin{equation*}
		\norm{\psi_i}^2_{\infty}<\gamma, \quad \norm{\Aa\psi_i}^2_\infty < \gamma.
	\end{equation*}
	Then Assumption \ref{bn assumption} holds with $\gamma_N^{}=N\gamma$. \red{For example, let $\X =[0,1]^d$ and:
	\begin{enumerate}
		\item Let $\psi_\b{n} = e^{2\pi i \b{n} \cdot \bm{x}}$ be the Fourier basis and $\Aa $ the Koopman operator of any dynamical system, then $\gamma =1$. 
		\item Let $\psi_\b{n}=\qty(1+\norm{2\pi \b{n}}+\norm{2\pi\b{n}}^2)^{-1}e^{2\pi i \b{n} \cdot \bm{x}}$ be the Fourier basis orthonormalized in $H^2(\X)$ and let $\Aa= \nabla - \Delta$, then the bound holds with $\gamma =1$. \exampleSymbol
	\end{enumerate}
	}
\end{example}

\begin{observation}
{Though it is possible to scale basis functions to enforce a uniform bound on $\norm{\Aa\psi_n}_\infty$ as in Example \ref{infinity example}, this is not recommended for higher-order operators. As shown in \eqref{order of convergence}, the sample complexity scales with $\norm{\bm{G}_N^{-1}}^4$. For the $H^2$-orthonormalized basis, $\norm{\bm{G}_N^{-1}}$ grows rapidly with $N$ due to vanishing $L^2$ norms. Instead, for higher-order operators, it is optimal to use well-conditioned bases and allow the bound $\gamma_N$ to grow with $N$.}
\end{observation}

\begin{example}[Locally supported dictionary] \label{fem bound example}
	Consider $\Psi_N$ as before and assume that there exists some $K\in\mathbb{N}$ for which
	\begin{equation*}
		\mu\qty(\bigcap_{k=1}^{K}\operatorname{supp}(\psi_{i_k}))=0, \quad \forall i_{1},\dots,i_{K}\in \set{1,\dots,N}.
	\end{equation*}
	Then Assumption \ref{bn assumption} holds with $\gamma_N=K\gamma$ (that is constant in $N$). This is, for example, the case if $\Psi$ is a finite element basis and $\Aa$ has order $0$. \exampleSymbol
\end{example}
The first ingredient for our error analysis is an error bound for the Gram matrix $\bm{G}$.
\begin{lemma}[Error estimate $\bm{G}$] \label{chernoff G}
	Under Assumptions \ref{assumption} and \ref{basis norm assumption}, given any $0<\delta <\frac{1}{2} \norm{\bm{G}^{-1}_N}^{-1}$ and $p\in (0,1)$ and for all
	\begin{equation*}
		M>(3 \norm{\bm{G}_N}+2 \delta ) \frac{2 \gamma_N^{}}{3 \delta ^2}\log \left(\frac{2
			N}{1-p}\right),
	\end{equation*}
	it holds that
	\begin{equation*}
		\mathbb{P}\qb{ \bm{\wh{G}}_{NM} \text{ is invertible and } \norm{\bm{\wh{G}}^{-1}_{NM}-\bm{G}_{N}^{-1}}<2{\norm{\bm{G}^{-1}_N}^2\delta}} \ge p.
	\end{equation*}
\end{lemma}

\begin{proof}
	Given $m \in \set{1,\dots,M} $, we define
	\begin{equation} \label{g as vector}
		\bm{S}_m := \frac{1}{M}\qty(\Psi_N(\bm{x}_m)\Psi_N^\dagger(\bm{x}_m) -\bm{G}_N),
	\end{equation}
	where $\dagger$ denotes the Hermitian adjoint.
	By construction, we are in the conditions of Bernstein's inequality for the covariance in Corollary \ref{Bernstein Lemma} where, using the notation of this inequality, $\bm{g}_m=\bm{c}_m=\Psi_N(\bm{x}_m), \bm{G}=\bm{C}=\bm{T}=\bm{G}_N, \gamma =\gamma_N^{}$ and
	\begin{equation*}
		\bm{Z} := \sum_{m=1}^M \bm{S}_m =\bm{\wh{G}}_{NM}-\bm{G}_N .
	\end{equation*}
	As a result, for $M$ as in the statement of the proposition, we obtain that
	\begin{equation*}
		\mathbb{P}\big[ \norm{\bm{Z}} <\delta \big]\geq p.
	\end{equation*}
	Now, since $\delta < 1/\norm{\bm{G}^{-1}_N}$ and $\bm{G}_N$ is invertible, we deduce that, with probability greater or equal to~$p$, the approximation $\wh{\bm{G}}_{NM}=\bm{G}_N+\bm{Z}$ is also invertible with inverse given by the Neumann series
	\begin{equation*}
		\wh{\bm{G}}_{NM}^{-1}= \bm{G}^{-1}_N\sum_{k=0}^{\infty} \qty(-\bm{Z}\bm{G}_N^{-1})^k.
	\end{equation*}
	Taking norms shows that, with probability greater or equal to $p$, $\wh{\bm{G}}_{NM}$ is invertible and
	\begin{equation}\label{neumann}
		\norm{\bm{\wh{G}}_{NM}^{-1}-\bm{G}_{N}^{-1}}\leq \norm{\bm{G}^{-1}_N}\sum_{k=1}^{\infty} \qty(\norm{\bm{G}^{-1}_N}\delta)^k= \frac{\norm{\bm{G}^{-1}_N}^2 \delta}{1-\norm{\bm{G}_N^{-1}}\delta}< 2\norm{\bm{G}^{-1}_N}^2\delta.
	\end{equation}
	This concludes the proof.
\end{proof}

In an analogous fashion, we can bound the error due to using $\bm{\wh{C}}_{NM}$ with an arbitrarily large probability. In fact, the result is more straightforward as we no longer have to deal with the matrix inversion necessary for $\bm{\wh{G}}_{NM}$.

\begin{lemma}[Error estimate $\bm{C}$] \label{chernoff c}
	Let $\Psi$ satisfy Assumptions \ref{assumption} and \ref{bn assumption} and let $\delta >0$ and $p\in (0,1)$ be arbitrary. Write $[\bm{T}_N]_{ij} := \br{\Aa \psi_i, \Aa \psi_j}$. Then, for all
	\begin{equation*}
		M>(3 \max \set{\norm{\bm{G}_N}, \norm{\bm{T}_N}}+2 \delta ) \frac{2 \gamma_N^{}}{3 \delta ^2}\log \left(\frac{2
			N}{1-p}\right),
	\end{equation*}
	it holds that
	\begin{equation*}
		\mathbb{P}\qty[ \norm{\wh{\bm{C}}_{NM}-\bm{C}_{N}}<\delta]\geq p.
	\end{equation*}
\end{lemma}

\begin{proof}
	The proof is similar to that of Lemma \ref{chernoff G}. Given $m \in \set{1, \dots,M}$, we define
	\begin{equation} \label{c as vector}
		\bm{S}_m := \frac{1}{M}\qty( \Aa \Psi_N(\bm{x}_m)\Psi_N^\dagger(\bm{x}_m) -\bm{C}_N).
	\end{equation}
	By construction, we are now able to apply Bernstein's inequality for the covariance defined in Corollary \ref{Bernstein Lemma}, where $\bm{g}_m=\Psi_N(\bm{x}_m), \bm{c}_m= \Aa \Psi_N(\bm{x}_m), \bm{G}=\bm{G}_N, \bm{C}= \bm{C}_N, \bm{T}=\bm{T}_N, \gamma =\gamma_N^{}$ and
	\begin{equation*}
		\bm{Z} := \sum_{m=1}^M \bm{S}_m =\bm{\wh{C}}_{NM}-\bm{C}_N .
	\end{equation*}
	As a result, for $M$ as in the statement of the proposition, we obtain that
	\begin{equation*}
		\mathbb{P}\qb{\norm{\bm{\wh{C}}_{NM}-\bm{C}_N}<\delta }=\mathbb{P}\big[\norm{\bm{Z}} < \delta \big]\geq p.
	\end{equation*}
	This concludes the proof.
\end{proof}

Having bounded the error due to using $\wh{\bm{G}}_{NM}$ and $\wh{\bm{C}}_{NM}$ instead of $\bm{G}_{N}$ and $\bm{C}_{N}$, we can now bound the error in approximating $\Aa_N$ by $\wh{\Aa}_{NM}$. To do so, we need to relate the bounds in the matrix norms back to the operator norms. As shown in the Lemma \ref{norm matrix operator lemma}, this depends on the condition number of $\bm{G}_N$, which can be written as
\begin{equation*}
	\kappa (\bm{G}_N) = \frac{\lambda_{\max}(\bm{G}_N)}{\lambda_{\min}(\bm{G}_N)},
\end{equation*}
since $\bm{G}_N$ is Hermitian.
The condition number is $1$ for an orthonormal basis but can become very large for other bases, such as the basis of monomials. This will be discussed in more detail below.

\begin{proposition}[Error estimate projection] \label{bound error projection}
	Let $\Psi_N$ satisfy Assumptions \ref{assumption} and \ref{bn assumption}. Let $0<\delta < \frac{1}{2}\norm{\bm{G}^{-1}_N}^{-1}$ and $p\in (0,1)$. Then, for all
	\begin{equation} \label{M condition}
		M > (3 \max \set{\norm{\bm{G}_N},\norm{\bm{T}_N}}+2 \delta ) \frac{2 \gamma^{}_N}{3 \delta ^2}\log \left(\frac{4
			N}{1-p}\right),
	\end{equation}
	it holds that
	\begin{equation*}
		\mathbb{P}\left[	\norm{\wh{\Aa}_{NM}- \Aa_N}\leq 2\sqrt{\kappa(\bm{G}_N)} \qty(1+\|\bm{C}_N\|\|\bm{G}_N^{-1}\|)\norm{\bm{G}_N^{-1}}\delta\right]\geq p.
	\end{equation*}
\end{proposition}

\begin{proof}
	By Lemmas \ref{chernoff G} and \ref{chernoff c}, we deduce that for $M$ as defined above it holds that
	\begin{equation}\label{lemma bound}
		\bm{\wh{G}}_{NM} \text{ is invertible, }\norm{\bm{\delta}_{\bm{G}^{-1}}}<2\norm{\bm{G}^{-1}_N}^2\delta\text{, and }\norm{\bm{\delta}_{\bm{C}}}<\delta,
	\end{equation}
	with probability greater or equal to $p$, where, in order to simplify the notation, we used
	\begin{equation*}
		\bm{\delta}_{\bm{G}^{-1}} := \wh{\bm{G}}_{NM}^{-1}-{\bm{G}}_{N}^{-1}, \quad \bm{\delta}_{\bm{C}} := \wh{\bm{C}}_{NM}-{\bm{C}}_{N}.
	\end{equation*}
	Let us now restrict ourselves to the region of the probability space where \eqref{lemma bound} holds. First, let $\Tt := \wh{\Aa}_{NM}- \Aa_N$, then
	\begin{equation*}
		\qty(\bm{T}^{\Psi_N})^T=\wh{\bm{C}}_{NM}\wh{\bm{G}}_{NM}^{-1}-\bm{C}_N \bm{G}_N^{-1}=\bm{C}_N\bm{\delta}_{\bm{G^{-1}}}+\bm{\delta}_{\bm{C}}\bm{G}_N^{-1}+\bm{\delta }_{\bm{C}}\bm{\delta }_{\bm{G}^{-1}}.
	\end{equation*}
	As a result, by \eqref{lemma bound} and collecting the terms, we have
	\begin{equation*}
		\norm{\Tt^{\Psi_N}}\leq \qty(2+2\|\bm{C}_N\|\|\bm{G}_N^{-1}\|)\norm{\bm{G}_N^{-1}}\delta.
	\end{equation*}
	Now, applying Lemma \ref{norm matrix operator lemma} yields
	\begin{equation*}
		\norm{\Tt} \leq 2\sqrt{\kappa(\bm{G}_N)}\qty(1+\|\bm{C}_N\|\|\bm{G}_N^{-1}\|)\norm{\bm{G}_N^{-1}}\delta.
	\end{equation*}
	This concludes the proof.
\end{proof}

The bound on the error in approximating the projection $\Aa_N$ in the just proved proposition can be combined with an error in the approximation of $\Aa$ through $\Aa_N$. This gives the following result.

\begin{theorem}[Order of convergence]\label{OC theorem}
	Let $\Psi_N$ satisfy Assumptions \ref{assumption} and \ref{bn assumption} and $N \in \N$ be arbitrary. Define $\rho_N:= \sqrt{\kappa(\bm{G}_N)} \qty(1+\|\bm{C}_N\|\|\bm{G}_N^{-1}\|)$, let $\varepsilon \in (0, \rho_N)$ be arbitrary and write $\delta_N:=\varepsilon/\qty(2\rho_N\norm{\bm{G}_N^{-1}})$. Then, for all $p\in (0,1)$ and
	\begin{equation*}
		M > (3 \max \set{\norm{\bm{G}_N},\norm{\bm{T}_N}}+2 \delta_N ) \frac{2 \gamma^{}_N}{3 \delta_N ^2}\log \left(\frac{4
			N}{1-p}\right),
	\end{equation*}
	it holds that
	\begin{equation*}
		\mathbb{P}\left[	\norm{\wh{\Aa}_{NM}- \Aa_N}\leq \varepsilon \right]\geq p.
	\end{equation*}
	Furthermore, if $f \in \Ff_\infty$ and
	$\norm{(\Pp_{\Ff_N}- Id)\phi}_\Ff = \mathcal{O}( N^{-\alpha})$ for all $\phi \in \Ff$, or if $f\in \Dd$ and additionally $ \norm{(\Pp_{\Dd_N}- Id)f}_\Dd = \mathcal{O}( N^{-\alpha}) $, then, for $N =\mathcal{O}(\varepsilon^{-\nicefrac 1\alpha})$ and $M$ as defined above, it holds that
	\begin{equation*}
		\mathbb{P}\left[{\norm{\wh{\Aa}_{NM}\Pp_{\Dd_N} f- \Aa f }_\Ff}\leq \norm{\Aa}\norm{f}_{\Dd}\varepsilon \right] \geq p.
	\end{equation*}
\end{theorem}

\begin{proof}
	By construction, we have that $\delta_N < \frac{1}{2}\norm{\bm{G}_N^{-1}}^{-1}$. Applying Proposition \ref{bound error projection} with $\delta_N$ in the place of $\delta$ proves the first part of the theorem. To prove the second part, we use the triangle inequality to obtain
	\begin{equation} \label{triangle}
		\norm{\wh{\Aa}_{NM}\Pp_{\Dd_N} f- \Aa f }_\Ff \leq\norm{\wh{\Aa}_{NM} - \Aa_N}\norm{f}_\Dd+\norm{\Aa_{N}\Pp_{\Dd_N} f- \Aa f }_\Ff .
	\end{equation}
	To bound the first term on the right-hand side of \eqref{triangle}, we use what was just proved. To bound the second term, we use Corollary \ref{order corollary} to obtain that, for $\Aa$ and $f$ nonzero,
	\begin{equation*}
		\norm{\wh{\Aa}_{NM}\Pp_{\Dd_N} f- \Aa f }_\Ff \leq \varepsilon \norm{f}_{\Dd}+C\varepsilon\norm{\Aa}\norm{f}\lesssim \norm{\Aa}\norm{f}_{\Dd}\varepsilon.
	\end{equation*}
	If $\Aa$ or $f$ are zero, on the other hand, it is clear that the error is zero. This concludes the proof.
\end{proof}

This theorem allows us to estimate the cost of an approximation of $\mathcal{A}$ with accuracy $\varepsilon > 0$ that holds with probability $p$. In terms of the order of convergence of $\wh{\Aa}_{NM}$ to $\Aa_N$,  we have that
\begin{equation}\label{order of convergence}
	M =\mathcal{O}\qty(\gamma_N^{}  \max \set{\norm{\bm{G}_N},\norm{\bm{T}_N}}\kappa \qty(\bm{G}_N)\norm{\bm{C}_N}^2 \norm{\bm{G}_N^{-1}}^4 \log \left(\frac{
		N}{1-p}\right) \varepsilon ^{-2}).
\end{equation}
{An important consequence of Theorem \ref{OC theorem} is that it provides an \emph{explicit} lower bound on $M$ as a function of $N$, which allows one to choose $M$ adaptively depending on $N$. The ability to prescribe $M$ as an explicit increasing function of $N$, rather than merely requiring $M$ to be ``sufficiently large'' without quantification, is a rare and highly non-trivial feature in the operator approximation literature. See \cite{colbrook2024limits} for a related discussion in the context of Koopman learning.}
\begin{observation}[Extension to other data-driven methods]
	The main technical tool to obtain the bounds on the approximation error is the Bernstein inequality \ref{Bernstein Lemma}. This inequality is quite flexible and can be directly extended to other data-driven matrices. For example, \emph{residual dynamic mode decomposition} \cite{colbrook2023residual} is based on estimating $\bm{T}_N$ using
	\begin{align*}
		\wh{\bm{T}}_{NM}:= \frac{1}{M}\sum_{m=1}^M \Aa \Psi_N(\bm{x}_m)(\Aa \Psi_N(\bm{x}_m))^\dagger.
	\end{align*}
	An application of Bernstein's inequality for the covariance in Corollary \ref{Bernstein Lemma} with $\bm{g}_m=\bm{c}_m= \Aa \Psi_N(\bm{x}_m), \bm{G}=\bm{C}=\bm{T}=\bm{T}_N, \gamma =\gamma_N^{}$ gives
	\begin{align*}
		\Pp \qty[\norm{\wh{\bm{T}}_{NM}-\bm{T}_N} < \varepsilon] \geq p, \quad \text{for} \quad   M= \mathcal{O}\qty(\gamma_N^{}\norm{\bm{T}_N} \log(\frac{N}{1-p})\varepsilon ^{-2}).
	\end{align*}
	Other methods to derive these error estimates exist. For example, in \cite[Theorem 3]{colbrook2024beyond}, the authors used a concentration type inequality to obtain under some further assumptions bounds of the form
	\begin{align*}
		\Pp \qty[\norm{\wh{\bm{T}}_{NM}-\bm{T}_N}_F < \varepsilon] \geq p, \quad \text{for} \quad   M= \mathcal{O}\qty({\alpha^2_N \gamma ^{}_N} \log(\frac{N}{1-p})\varepsilon ^{-2}),
	\end{align*}
	where $\psi_j$ are assumed to be Lipschitz with constant $c_j$ and $\alpha^2_N := \sum_{j=1} ^N c_j^2$. If, for example, $\norm{T_N}$ and $c_N^{}$ are bounded in $N$ with $\lim_{N\to\infty} c_N = C$, the bound given by Bernstein's inequality is sharper.
\end{observation}

Let us now consider two examples.

\begin{example} \label{dense bound example}
	If $\bm{C}_N$, $\bm{G}_N$, and $\kappa(\bm{G}_N)$ are bounded and $\gamma_N^{}=\mathcal{O}(N)$, then, to obtain an error of $\varepsilon$,
	\begin{equation*}
		M=\mathcal{O}\qty(-\frac{1}{\alpha}\log(\frac{\varepsilon}{1-p})\varepsilon ^{-2-\nicefrac 1\alpha}), \quad N=\Oo(\varepsilon^{-\nicefrac 1\alpha}).
	\end{equation*}
	The cost of obtaining $\wh{\bm{A}}_{NM}$ is
	\begin{equation*}
		C(\varepsilon, p)=\Oo(M N^2+N^{3})=\Oo(MN^2)=\mathcal{O}\qty(-\log(\frac{\varepsilon}{1-p})\varepsilon ^{-2-\nicefrac 3\alpha}). \tag*{\exampleSymbol}
	\end{equation*}
\end{example}

\begin{example}
	If $\bm{G}_N$ is sparse, it is likely that $\norm{\bm{G}_N^{-1}}$ is not bounded. For example for a $1$-dimensional FEM basis $\norm{\bm{G}_N^{-1}}= \Oo\qty(h_N^{-1})= \Oo\left(N^{\nicefrac12} \right) $ and from \eqref{order of convergence}, if the remaining quantities are bounded, we obtain
	\begin{equation*}
		M=\mathcal{O}\qty(-\frac{1}{\alpha}\log(\frac{\varepsilon}{1-p})\varepsilon ^{-2-\nicefrac2\alpha}), \quad N=\Oo(\varepsilon^{-\nicefrac 1\alpha}).
	\end{equation*}
	As a result, the cost of obtaining $\wh{\bm{A}}_{NM}$ is
	\begin{equation*}
		C(\varepsilon, p)=\Oo\qty(M N)=\mathcal{O}\qty(\log(\frac{\varepsilon}{1-p})\varepsilon ^{-2-\nicefrac3\alpha}).
	\end{equation*}
	This is equal to the cost in the previous example.\exampleSymbol
\end{example}
\subsection{Accounting for measurement error within the operators} 
In this section we show the effect of measurement error in the data-driven approximation of the operator $\Aa$. In the first subsection we show how a direct implementation leads to \red{bias}. In the second subsection we show how to correct for this bias by dividing the evaluation of the basis functions into two batches.
\label{Noise section}
\subsubsection{Data-driven approximation with noisy measurements: The issue of bias} \label{Noise section2}
In practice, it is often not possible to evaluate the dictionary $(\psi_n)_{n=1}^N$ and operator applied to the dictionary functions $(\mathcal{A}\psi_n)_{n=1}^N$  precisely. The evaluations may be perturbed by measurement errors, or we may need to approximate the operator $\mathcal{A}$, as mentioned in the beginning of Section~\ref{general framework section}. This error will influence the accuracy of the approximation of $\mathcal{A}$. We will assume in what follows that instead of \eqref{info}, we only have access to
\begin{equation} \label{data22}
	\{\psi_n(\bm{x}_m)+\eta_N^{m,n}, \Aa \psi_n(\bm{x}_m) +\xi_N^{m,n}\}_{m,n=1}^{M,N},
\end{equation}
where  $(\eta_N^{m,1},\dots, \eta_N^{m, N}) =: \bm{\eta }_N^m$ and $(\xi^{m,1},\dots,\xi^{m, N}) =: \bm{\xi }_N^m$ represent the measurement or evaluation error in dictionary and operator, respectively. Given this perturbed data, a direct implementation of the data-driven approximation is 
\begin{equation*}
	\tl{{\bm{A}}}_{NM}^\top := \tl{{\bm{C}}}_{NM}\tl{{\bm{G}}}_{NM}^+,
\end{equation*}
with perturbed structure and Gram matrices given by
\begin{align}\label{noisy matrices2}
	[\tl{\bm{{C}}}_{NM}]_{ij} & :=\frac{1}{M}\sum_{m=1}^M\qty(\Aa\psi_i(\bm{x}_m)+\bm{\xi}_N^{m,i})\qty(\psi_j(\bm{x}_m)+\bm{\eta}_N^{m,j})^\dagger, \\
	[\tl{\bm{{G}}}_{NM}]_{ij} & :=\frac{1}{M}\sum_{m=1}^M\qty(\psi_i(\bm{x}_m)+\bm{\eta}_N^{m,i})\qty(\psi_j(\bm{x}_m)+\bm{\eta}_N^{m,j})^\dagger,
\end{align}
respectively. In this section, we show how the error bounds in Lemmas \ref{chernoff G} and \ref{chernoff c} can be modified to account for this noise. This leads to an analogous error bound on the data-driven operator as in Proposition~\ref{bound error projection} and Theorem~\ref{OC theorem} \red{where a bias term appears.} We make the following assumptions on the noise.
\begin{assumption} \label{noise assumption2}
	The random variables $\set{(\bm{\eta}_N^m, \bm{\xi}_N^m)}_{m=1}^M$ have mean $\bm{0}$, are symmetric, independent, and independent of $\set{\bm{x}_m }_{m=1}^M$. 
\end{assumption}{Given $\tl{\gamma }_N \geq 0 $ we will use the notation}
\begin{equation*}
	\tl{p}_N:=\mathbb{P}\qty[\abs{(\bm{\eta}_N^m,\bm{\xi}_N^m)}^2\leq \tl{\gamma}_N^{}].
\end{equation*}
\red{Let $\mathbb{E}[\cdot \mid \cdot]$ denote the conditional expectation and write
\begin{align*}
	\b{\Sigma}_{N}^\b{\eta}&:=\mathbb{E}\qty[\bm{\eta}_N^m(\bm{\eta}_N^{m})^\dagger\mid \left\{\abs{(\bm{\eta}_N^m,\bm{\xi}_N^m)}^2\leq \tl{\gamma}_N^{}, ~\forall m\right\} ],\\ \b{\Sigma}_{N}^\b{\xi}&:=\mathbb{E}\qty[\bm{\xi}_N^m(\bm{\xi}_N^{m})^\dagger\mid \left\{\abs{(\bm{\eta}_N^m,\bm{\xi}_N^m)}^2\leq \tl{\gamma}_N^{}, ~\forall m\right\}],
	\\
	\b{\Sigma}_{N}^{\b{\xi},\b{\eta}}&:=\mathbb{E}\qty[\bm{\xi}_N^m(\bm{\eta}_N^{m})^\dagger\mid\left\{\abs{(\bm{\eta}_N^m,\bm{\xi}_N^m)}^2\leq \tl{\gamma}_N^{}, ~\forall m\right\}],
\end{align*}
for the covariance matrices of the noise knowing that its norm squared is bounded by $\tl{\gamma }_N$. We also write the perturbed exact matrices as
\begin{align}\label{noisy matrices3}
\tl{\b{G}}_N & := \bm{G}_N + \red{\Sigma_N^\b{\eta }}, \quad \tl{\b{C}}_N:=  \bm{C}_N+ \red{\Sigma_N^{\b{\xi},\b{\eta}}}, \quad \tl{\b{T}}_N:= \bm{T}_N + \red{\Sigma_N^\b{\xi }}
\end{align}
and define the perturbed matrix of the operator as 
 \begin{align*}
	\tl{\b{A}}_N ^T& := \tl{\b{C}}_N \tl{\b{G}}_N^+.
\end{align*}
}

\begin{proposition}[\red{Biased} error estimate with noise] \label{noise bound2}
	Let $\Psi, \bm{\eta}_N, \bm{\xi }_N $ satisfy Assumptions \ref{assumption}, \ref{bn assumption}, and \ref{noise assumption2}, and let $p\in (0,1), \tl{p}_N \in (p^\red{\nicefrac{1}{M}},1)$. Suppose $\delta >0, \tl{\gamma }_N>0$ are such that  $\delta+\red{\norm{\b{\Sigma }_N^\b{\eta}}}< \frac{1}{2}\norm{\bm{G}^{-1}_N}^{-1}$. Then, for all
	\begin{equation*}
		M > (3 \max \set{\norm{\bm{G}_N},\norm{\bm{T}_N}}+ \red{3 \sigma _N^2}+2 \delta ) \frac{4 (\gamma_N^{}+ \tl{\gamma}_N^{})}{3 \delta ^2}\log \left(\frac{4
			N}{1-p / \tl{p}_N^{\red{M}}}\right),
	\end{equation*}
	it holds that
		\begin{equation*}
		\mathbb{P}\left[    \norm{\tl{{\Aa}}_{NM}- \Aa_N}\leq 2\sqrt{\kappa(\bm{G}_N)} \qty(1+\|\bm{C}_N\|\|\bm{G}_N^{-1}\|)\norm{\bm{G}_N}^{-1}(\delta+ \red{\sigma_N ^2})\right]\geq p,
	\end{equation*}
	where \red{$\sigma _N^2:= \max\set{\norm{ \b{\Sigma }_N^\b{\eta}}, \norm{ \b{\Sigma }_N^{\b{\xi },\b{\eta}}}}$}.
\end{proposition}

\begin{proof} 
	We begin by restricting ourselves to realizations where $\abs{\bm{\eta}^m_N}^2\leq \tl{\gamma}_N^{}, \quad \abs{\bm{\xi}^m_N}^2\leq \tl{\gamma}_N^{}$. More formally, we modify our probability space to $(\Omega, \mathcal{E}, \tl{\mathbb{P}})$ where $\tl{\mathbb{P}}$ is the conditional probability
	\begin{equation*}
		\tl{\mathbb{P}}(\mathbb{A}) := {\tl{p}_N}^{\red{-M}} {\mathbb{P}\qty[\A\cap \left\{\abs{(\bm{\eta}_N^m,\bm{\xi}_N^m)}^2\leq \tl{\gamma}_N^{}, ~\forall m\right\}]}.
	\end{equation*}
Given $k \in \set{1,\dots,N}$, by the independence of $\bm{x}_k$ of $\set{(\bm{\eta}_N^{_m}, \bm{\xi}_N^{m})}_{m=1}^M$, for all $\A \in \Bb(\R^d)$, we have
	\begin{equation*}
		\tl{\mathbb{P}}(\bm{x}_k\in \mathbb{A}) =\tl{p}_N^\red{-M} \mathbb{P}\qty[\bm{x}_k\in \mathbb{A}\cap \left\{\abs{(\bm{\eta}_N^{_m}, \bm{\xi}_N^{m})}^2\leq \tl{\gamma}_N^{}, ~\forall m\right\}]=\mathbb{P}[\bm{x}_k\in \A]\cdot 1=\mu(\A).
	\end{equation*}
	That is, also $\bm{x}_k\sim \mu$ under the probability measure $\tl{\mathbb{P}}$. Additionally, the family $\set{\bm{x}_m}_{m=1}^M$ is independent under $\tl{\mathbb{P}}$ as well. This is because, given $\A_1,\dots, \A_M \in \Bb(\R^d)$, we have
	\begin{align*}
		\tl{\mathbb{P}}(\set{\bm{x}_m\in \A_m, ~  \forall m}) & =\tl{p}_N^\red{-M}{\mathbb{P}\qty[\set{\bm{x}_m\in \A_m, ~  \forall m}\cap \left\{\abs{(\bm{\eta}_N^{_m}, \bm{\xi}_N^{m})}^2\leq \tl{\gamma}_N^{}, ~\forall m\right\}]}\\
		                                                      & =\mathbb{P}[\set{\bm{x}_m\in \A_m, ~  \forall m}]\cdot 1=\mu(\A_1)\cdots \mu (\mathbb{A}_M),
	\end{align*}
	where in the second and last equality we used the independence stated in Assumption \ref{noise assumption2}. Similarly, one can also show that $\set{\bm{x}_m}_{m=1}^M$ are independent of $\set{(\bm{\eta}_N^{_m}, \bm{\xi}_N^{m})}_{m=1}^M$ and that $\set{(\bm{\eta}_N^{_m}, \bm{\xi}_N^{m})}_{m=1}^M$ are i.i.d.\ with distribution
	\begin{equation*}
		\tl{\mathbb{P}}[(\bm{\eta}_N^{_m}, \bm{\xi}_N^{m}) \in \mathbb{D}]=\tl{p}_N^\red{-M}\mathbb{P}\qty[\set{ (\bm{\eta}_N^{_m}, \bm{\xi}_N^{m}) \in \mathbb{D}}\cap \set{\abs{(\bm{\eta}_N^{_m}, \bm{\xi}_N^{m})}^2\leq \tl{\gamma}_N^{}}], \quad m =1,\dots,M.
	\end{equation*}
	Using this result, since $(\bm{\eta}_N^m, \bm{\xi }_N^m)$ is symmetric and centred at $\bm{0}$ under $\mathbb{P}$, we obtain that $(\bm{\eta}_N^m, \bm{\xi }_N^m)$ also has mean $\bm{0}$ under $\tl{\mathbb{P}}$.    In conclusion,
	\begin{equation} \label{gcdef2}
		\set{(\bm{g}_m,\bm{c}_m)}_{m=1}^M := \set{\Psi_N(\bm{x}_m)+\bm{\eta}_N^m,\Aa\Psi_N(\bm{x}_m)+\bm{\xi}_N^m}_{m=1}^M,
	\end{equation}
	are i.i.d.\ (under $\tl{\mathbb{P}}$) with mean 
	\begin{equation} \label{gcmatdef2}
	\red{\tl{\b{G}}_N } = \tl{\E}\qty[\bm{g}\bm{g^\dagger}], \quad \red{\tl{\bm{T}}_N}= \tl{\E}\qty[\bm{c}\bm{c}^\dagger], \quad \red{\tl{\bm{C}}_N} := \tl{\E}[\bm{c}\bm{g}^\dagger],
	\end{equation}
	and by definition of $\tl{\mathbb{P}}$ and the triangle inequality for all $m \in \set{1,\dots,M} $, it holds that
	\begin{equation} \label{bernbound2}
		\tl{\mathbb{P}}[\abs{\bm{g}_m}^2 \leq 2(\gamma_N^{} + \tl{\gamma}_N^{}) \text{ and } \abs{\bm{c}_m}^2 \leq 2(\gamma_N^{} + \tl{\gamma}_N^{})]=1.
	\end{equation}
	Now, analogously to \eqref{g as vector} and \eqref{c as vector}, we define
	\begin{equation} \label{sum def2}
		\begin{split}
			\bm{S^G}_{m} & := \frac{1}{M} (\Psi_N(\bm{x}_m)+\bm{\eta}_N^m)(\Psi_N(\bm{x}_m)+\bm{\eta}_N^m)^\dagger- \red{\tl{ \bm{G}}_N},    \\
			\bm{S^C}_{m} & := \frac{1}{M} (\Aa\Psi_N(\bm{x}_m)+\bm{\xi}_N^m)(\Psi_N(\bm{x}_m)+\bm{\eta}_N^m)^\dagger- \red{\tl{ \bm{C}}_N}.
		\end{split}
	\end{equation}
	By the independence of $\bm{g}_m$ and $\bm{c}_m$ in \eqref{gcdef2}, their mean value in \eqref{gcmatdef2} and the bound in \eqref{bernbound2}, we may apply Bernstein's inequality for the covariance defined in Corollary \ref{Bernstein Lemma} to \eqref{sum def2}, which implies that
	\begin{align*}
		\tl{\mathbb{P}}\qb{\norm{\tl{\bm{G}}_{NM}-\red{\tl{ \bm{G}}_N}} \geq\delta } & \leq 2 N \exp(\frac{-M \delta ^2 /2 }{ 2(\gamma_N^{}+ \tl{\gamma}_N^{})\qty(\norm{\red{\tl{ \bm{G}}_N}} + {2 \delta }/{3})}),                             \\
		\tl{\mathbb{P}}\qb{\norm{\tl{\bm{C}}_{NM}-\red{\tl{ \bm{C}}_N}} \geq\delta } & \leq 2 N \exp(\frac{-M \delta ^2 /2 }{ 2(\gamma_N^{}+ \tl{\gamma}_N^{})\qty(\max \set{\norm{\red{\tl{ \bm{T}}_N}}, \norm{\red{\tl{\bm{G}}_N}}} +{2 \delta }/{3})}) .
	\end{align*}
	Solving for $M$ we obtain that, for $M$ as in the problem statement    \begin{equation} \label{chernoff noise2}
		p\leq\tl{p}_N^M\tl{\mathbb{P}}\qb{\norm{\tl{\bm{G}}_{NM}-\red{\tl{ \bm{G}}_N}} <\delta, \text{ and } \norm{\tl{\bm{C}}_{NM}-\red{\tl{ \bm{C}}_N}} <\delta}.
	\end{equation}
	Now, by definition of $\tl{\mathbb{P}}$, we have that $\tl{p}_N^{M}\tl{\mathbb{P}}\leq \mathbb{P}$. Using this in \eqref{chernoff noise2} shows that, for $M$ as above
	\begin{equation}\label{lemma bound2}
		p\leq\mathbb{P}\qb{\norm{\tl{\bm{G}}_{NM}-\red{\tl{ \bm{G}}_N}} <\delta, \text{ and } \norm{\tl{\bm{C}}_{NM}-\red{\tl{ \bm{C}}_N}} <\delta}.
	\end{equation}
	The proof now follows the same lines as the proof of Proposition \ref{bound error projection}.
We restrict ourselves to the region of the probability space where \eqref{lemma bound2} holds. To simplify the notation, write
\begin{equation*}
		\bm{\delta}_{\bm{G}} := \tl{\bm{G}}_{NM}-{\bm{G}}_{N}, \quad \bm{\delta}_{\bm{C}} := \tl{\bm{C}}_{NM}-{\bm{C}}_{N}.
	\end{equation*} 
	By \eqref{noisy matrices3}, the triangle inequality and \eqref{lemma bound2}, we have
\begin{align*}
	\norm{\delta _{\b{G}}} \leq \delta +\red{\norm{\b{\Sigma }_N^\b{\eta}}<\frac{1}{2} \norm{\bm{G}_N^{-1}}^{-1}} , \quad \norm{\delta _{\b{C}}} \leq \delta +\red{\norm{\b{\Sigma }_N^{\b{\xi },\b{\eta}}}}.
\end{align*}
As a result, $\tl{\bm{G}}_{NM}$ is invertible. Write $\bm{\delta}_{\bm{G}^{-1}} := \tl{\bm{G}}_{NM}^{-1}-{\bm{G}}_{N}^{-1}$. Using the Neumann series as in  \eqref{neumann}
\begin{align}\label{lemma bound3}
	\norm{\delta_{\b{G}^{-1}}}\leq \frac{\norm{\bm{G}^{-1}_N}^2 \norm{\delta_{\b{G}}}}{1-\norm{\bm{G}_N^{-1}}\norm{\delta_{\b{G}}}}< 2\norm{\bm{G}^{-1}_N}^2 \norm{\delta_{\b{G}}}.
\end{align}
Let $\tl{\Tt} := \tl{\Aa}_{NM}-{\Aa}_{N}$. Then, its matrix representation is given by
	\begin{equation*}
		\qty(\tl{\bm{T}}^{\Psi_N})^T=\tl{\bm{C}}_{NM}\tl{\bm{G}}_{NM}^{-1}-\bm{C}_N \bm{G}_N^{-1}=\bm{C}_N\bm{\delta}_{\bm{G^{-1}}}+\bm{\delta}_{\bm{C}}\bm{G}_N^{-1}+\bm{\delta }_{\bm{C}}\bm{\delta }_{\bm{G}^{-1}}.
	\end{equation*}
	As a result, by \eqref{lemma bound3} and collecting the terms, we have
	\begin{equation*}
		\norm{\tl{\Tt}^{\Psi_N}}\leq 2\qty(\norm{\delta_\b{C} }+\|\bm{C}_N\|\|\bm{G}_N^{-1}\|\norm{\delta_\b{G}})\norm{\bm{G}_N^{-1}}.
	\end{equation*}
	Now, applying Lemma \ref{norm matrix operator lemma} and the definition of $\sigma_N ^2$  yields
	\begin{equation*}
		\norm{\tl{\Tt}} \leq 2\sqrt{\kappa(\bm{G}_N)}\qty(1+\|\bm{C}_N\|\|\bm{G}_N^{-1}\|)\norm{\bm{G}_N^{-1}}(\delta +\red{\sigma_N^2}) .
	\end{equation*}
	This concludes the proof.
\end{proof}
Proposition \ref{noise bound2} shows that, in this setting and stressing the dependence of $\delta $ on $M$, the error is of the form $\lambda_N(\delta_M+\sigma _N^2)$. As a result, the error will not be small unless the variance of $\b{\eta }$ and covariance of $\b{\eta }$ and $\b{\xi }$ are small. Additionally, the mass matrix $\tl{\b{G}}_{NM}$ may cease to be invertible unless $ \delta +\red{\norm{\b{\Sigma }_N^\b{\eta}}<\norm{\bm{G}_N^{-1}}^{-1}}$. This is in contrast to Proposition \ref{OC theorem} and Proposition \ref{noise bound} of the \red{next subsection} where the error is of the form $\lambda_N \delta_M$. In these cases, the error can be made small by taking $M$ large enough. 
\begin{example}
Write $\b{I} \in \R^{n \times n}$ for the identity matrix and consider the simple case where the data driven algorithm without noise is exact with $\b{C}_N= \wh{\b{C}}_{NM}=c \b{I}, \b{G}_N= \wh{\b{G}}_{NM}=g \b{I}$ for some $c \in \R, g>0$. Let $\b{\Sigma }_N^{\b{\eta} }= \sigma ^2_\b{\eta }\b{I}$  and $\b{\Sigma }_N^{\b{\xi},\b{\eta } }= v_{\b{\xi},\b{\eta } }\b{I}$ for some $\sigma_{\b{\eta }} ^2>0, v_{\b{\xi},\b{\eta } }\in \R$ .  Then, we have that 
\begin{align*}
	\lim_{M \to \infty} \tl{\b{G}}_{NM}= (g+\sigma_\b{\eta } ^2)\b{I} \quad \text{and } \lim_{M \to \infty} \tl{\b{C}}_{NM}= (c+v_{\b{\xi},\b{\eta }})\b{I}.
\end{align*}
As a result, a calculation shows that
\begin{align*}
	\lim_{M \to \infty}\tl{\b{A}}_{NM}=\b{A}_N+ \frac{g v_{\b{\xi},\b{\eta } }-c\sigma ^2_\b{\eta }  }{g(g+\sigma ^2_\b{\eta })}\b{I}.
\end{align*}
In consequence, as $M$ becomes large we expect an error of the form 
\begin{align*}
	\norm{\tl{\b{A}}_{NM}-\b{A}_N}\approx \frac{\abs{v_{\b{\xi},\b{\eta } }}}{g+\sigma ^2_\b{\eta }}+ \frac{\abs{c}\sigma ^2_\b{\eta } }{g(g+\sigma ^2_\b{\eta })},
\end{align*}
which is independent of $M$ and is not small unless the variance terms $\sigma ^2_\b{\eta }$ and $v_{\b{\xi},\b{\eta } }$ are small relative to $\b{G}_N$ and $\b{C}_N$. 
\end{example}

\subsection{Unbiased estimation with noisy measurements via batching} \label{Noise section3}
In this section we show how to correct for the bias in the data-driven approximation of the operator $\Aa$. The idea is to evaluate the dictionary functions twice independently, and then use the average of these evaluations to form the data-driven approximation. Consider the noisy maps
\begin{align*}
	\b{x} \to \psi_n(\bm{x})+\eta_N^{\b{x},n}, \quad \b{x} \to \Aa \psi_n(\bm{x})+\xi_N^{\b{x},n}, \quad n=1,\dots,N,
\end{align*}
where $\bm{\eta}_N^{\b{x}}= (\eta_N^{\b{x},1},\ldots,\eta_N^{\b{x},N})$ and $\bm{\xi}_N^{\b{x}}=(\xi_N^{\b{x},1},\ldots,\xi_N^{\b{x},N})$ represent the measurement or evaluation error in dictionary and operator, respectively.
Consider now $\b{x}_1, \cdots, \b{x}_M \sim \mu $. For each $\b{x}_m$ we perform these noisy evaluations, evaluating the dictionary twice independently to obtain the data
\begin{equation} \label{data2}
	\{\psi_n(\bm{x}_m)+\eta_N^{m,n}, \psi_n(\bm{x}_m)+\co{\eta}_N^{m,n},\Aa \psi_n(\bm{x}_m)+\xi_N^{m,n}\}_{m,n=1}^{M,N}
\end{equation}  
where, we wrote $\bm{\eta }_N^m:=(\eta_N^{m,1},\dots, \eta_N^{m, N}), \co{\bm{\eta}}_N^m:=(\co{\eta}_N^{m,1},\dots, \co{\eta}_N^{m, N})$, and $\bm{\xi }_N^m:=(\xi_N^{m,1},\dots, \xi_N^{m, N})$. Given the perturbed data \eqref{data2}, we form the data-driven approximation
\begin{equation*}
	\co{{\bm{A}}}_{NM}^\top := \co{{\bm{C}}}_{NM}\co{{\bm{G}}}_{NM}^+,
\end{equation*}
with perturbed structure and Gram matrices defined respectively by
\begin{align}\label{noisy matrices}
	[\co{\bm{{C}}}_{NM}]_{ij} & :=\frac{1}{M}\sum_{m=1}^M\qty(\Aa\psi_i(\bm{x}_m)+{\xi}_N^{m,i})\qty(\psi_j(\bm{x}_m)+\co{{\eta}}_N^{m,j})^\dagger, \\
	[\co{\bm{{G}}}_{NM}]_{ij} & :=\frac{1}{M}\sum_{m=1}^M\qty(\psi_i(\bm{x}_m)+{\eta}_N^{m,i})\qty(\psi_j(\bm{x}_m)+\co{{\eta}}_N^{m,j})^\dagger.
\end{align}
 In the same line as Assumption \ref{noise assumption2} we make the following assumption on the noise
\begin{assumption} \label{noise assumption}
  We assume the following:
  \begin{enumerate}[label=(\alph*), ref=\theassumption(\alph*)]  \setlength{\itemsep}{0ex}
	  \item The random variables $\set{(\bm{\eta}_N^{_m}, \bm{\xi}_N^{m},\co{\bm{\eta }}_N^m)}_{m=1}^M$  have mean $\bm{0}$, are symmetric, independent, and independent of $\set{\bm{x}_m }_{m=1}^M$.
	  \item The random variables $\set{(\bm{\eta}_N^{_m}, \bm{\xi}_N^{m})}_{m=1}^M$ and $\set{\co{\b{\eta }}_N^m }$ are independent.
	  \item The random variables $\bm{\eta}_N^m$ and $\co{\bm{\eta}}_N^m$ are identically distributed.
  \end{enumerate}
\end{assumption}
{Given $\co{\gamma }_N \geq 0 $ we will use the notation}
\begin{equation*}
	\co{p}_N:=\mathbb{P}\qty[\abs{(\bm{\eta}_N^m,\bm{\xi}_N^m,\co{\bm{\eta}}_N^m)}^2\leq \co{\gamma}_N^{}],
\end{equation*}
and define the conditioned noise covariances
\begin{align*}
	\b{\Sigma}_{N}^\b{\eta}&:=\mathbb{E}\qty[\bm{\eta}_N^m(\bm{\eta}_N^m)^\dagger\mid \left\{\abs{(\bm{\eta}_N^{_m}, \bm{\xi}_N^{m},\co{\bm{\eta }}_N^m)}^2\leq \co{\gamma}_N^{}, ~\forall m\right\}],\\
	\b{\Sigma}_{N}^\b{\xi}&:=\mathbb{E}\qty[\bm{\xi}_N^m(\bm{\xi}_N^m)^\dagger\mid \left\{\abs{(\bm{\eta}_N^{_m}, \bm{\xi}_N^{m},\co{\bm{\eta }}_N^m)}^2\leq \co{\gamma}_N^{}, ~\forall m\right\}].
\end{align*}
Note that since $\bm{\eta}_N^m$ and $\co{\bm{\eta}}_N^m$ are identically distributed, the conditioned covariance of $\co{\bm{\eta}}_N^m$ equals $\b{\Sigma}_{N}^\b{\eta}$.

Assumption \ref{noise assumption} is satisfied by Gaussian measurement error and, by studying the cumulative distribution function of the $\chi^2$ distribution, an explicit expression for $\co{p}_N $ can be given. We do so in Example~\ref{noise example 1} below, after adapting our previous estimates to the setting of noisy data, starting with the result analogous to Proposition~\ref{bound error projection}.

\begin{proposition}[Error estimate noise] \label{noise bound}
	Let $\Psi, \bm{\eta}_N, \bm{\xi }_N $ satisfy Assumptions \ref{assumption}, \ref{bn assumption}, and \ref{noise assumption}, and let $0<\delta < \frac{1}{2}\norm{\bm{G}^{-1}_N}^{-1}$ and $p\in (0,1), \co{p}_N \in (p^\red{\nicefrac{1}{M}},1)$. Then, for all
	\begin{equation*}
		M > \qty(3 \max \set{\norm{\bm{G}_N}+\norm{\b{\Sigma}_{N}^{\b{\eta}}},\norm{\bm{T}_N+ {\Sigma_N^\b{\xi }}}}+2 \delta ) \frac{4 (\gamma_N^{}+ \co{\gamma}_N^{})}{3 \delta ^2}\log \left(\frac{4
			N}{1-p / \co{p}_N^{M}}\right),
	\end{equation*}
	it holds that
	\begin{equation*}
		\mathbb{P}\left[    \norm{\co{{\Aa}}_{NM}- \Aa_N}\leq 2\sqrt{\kappa(\bm{G}_N)} \qty(1+\|\bm{C}_N\|\|\bm{G}_N^{-1}\|)\norm{\bm{G}_N^{-1}}\delta\right]\geq p,
	\end{equation*}
\end{proposition}

\begin{proof}
	The proof is similar to that of Proposition \ref{noise bound2}. We  modify our probability space to $(\Omega, \mathcal{E}, \co{\mathbb{P}})$ where $\co{\mathbb{P}}$ is the conditional probability
	\begin{equation*}
		\co{\mathbb{P}}(\mathbb{A}) :=\co{p}_N^\red{-M} \mathbb{P}\qty[\A\cap \left\{\abs{(\bm{\eta}_N^{_m}, \bm{\xi}_N^{m},\co{\bm{\eta }}_N^m)}^2\leq \co{\gamma}_N^{}, ~\forall m\right\}].
	\end{equation*}
	Given $k \in \set{1,\dots,N}$, by the independence of $\bm{x}_k$ of $\set{(\bm{\eta}_N^{_m}, \bm{\xi}_N^{m},\co{\bm{\eta }}_N^m)}_{m=1}^M$, for all $\A \in \Bb(\R^d)$, we have
	\begin{equation*}
		\co{\mathbb{P}}(\bm{x}_k\in \mathbb{A}) =\co{p}_N^\red{-M} \mathbb{P}\qty[\bm{x}_k\in \mathbb{A}\cap \left\{\abs{(\bm{\eta}_N^{_m}, \bm{\xi}_N^{m},\co{\bm{\eta }}_N^m)}^2\leq \co{\gamma}_N^{}, ~\forall m\right\}]=\mathbb{P}[\bm{x}_k\in \A]\cdot 1=\mu(\A).
	\end{equation*}
	That is, also $\bm{x}_k\sim \mu$ under the probability measure $\co{\mathbb{P}}$. Additionally, the family $\set{\bm{x}_m}_{m=1}^M$ is independent under $\co{\mathbb{P}}$ as well. This is because, given $\A_1,\dots, \A_M \in \Bb(\R^d)$, we have
	\begin{align*}
		\co{\mathbb{P}}(\set{\bm{x}_m\in \A_m, ~  \forall m}) & =\co{p}_N^\red{-M}{\mathbb{P}\qty[\set{\bm{x}_m\in \A_m, ~  \forall m}\cap \left\{\abs{(\bm{\eta}_N^{_m}, \bm{\xi}_N^{m},\co{\bm{\eta }}_N^m)}^2\leq \co{\gamma}_N^{}, ~\forall m\right\}]}\\
		                                                      & =\mathbb{P}[\set{\bm{x}_m\in \A_m, ~  \forall m}]\cdot 1=\mu(\A_1)\cdots \mu (\mathbb{A}_M),
	\end{align*}
	where in the second and last equality we used the independence stated in Assumption \ref{noise assumption}. Similarly, one can also show that $\set{\bm{x}_m}_{m=1}^M$ are independent of $\set{(\bm{\eta}_N^{_m}, \bm{\xi}_N^{m},\co{\bm{\eta }}_N^m)}_{m=1}^M$, that $\set{(\bm{\eta}_N^{_m}, \bm{\xi}_N^{m})}_{m=1}^M$ are independent of $\set{\co{\bm{\eta }}_N^m }_{m=1}^M$, and that $\set{(\bm{\eta}_N^{_m}, \bm{\xi}_N^{m},\co{\bm{\eta }}_N^m)}_{m=1}^M$ are i.i.d.\ with distribution
	\begin{equation*}
		\co{\mathbb{P}}[(\bm{\eta}_N^{_m}, \bm{\xi}_N^{m},\co{\bm{\eta }}_N^m) \in \mathbb{D}]=\co{p}_N^\red{-M}\mathbb{P}\qty[\set{ (\bm{\eta}_N^{_m}, \bm{\xi}_N^{m},\co{\bm{\eta }}_N^m) \in \mathbb{D}}\cap \set{\abs{(\bm{\eta}_N^{_m}, \bm{\xi}_N^{m},\co{\bm{\eta }}_N^m)}^2\leq \co{\gamma}_N^{}}], \quad m =1,\dots,M.
	\end{equation*}
	Using this result, since $(\bm{\eta}_N^{_m}, \bm{\xi}_N^{m},\co{\bm{\eta }}_N^m)$ is symmetric and centred at $\bm{0}$ under $\mathbb{P}$, we obtain that $(\bm{\eta}_N^{_m}, \bm{\xi}_N^{m},\co{\bm{\eta }}_N^m)$ also has mean $\bm{0}$ under $\co{\mathbb{P}}$.    In conclusion,
	\begin{equation} \label{gcdef}
		\set{(\bm{g}_m,\bm{c}_m,\co{\b{g}}_m)}_{m=1}^M := \set{\Psi_N(\bm{x}_m)+\bm{\eta}_N^m,\Aa\Psi_N(\bm{x}_m)+\bm{\xi}_N^m,\Psi_N(\bm{x}_m)+\co{\bm{\eta}}_N^m}_{m=1}^M,
	\end{equation}
	are i.i.d.\ (under $\co{\mathbb{P}}$) with
	\begin{equation} \label{gcmatdef}
		\bm{G}_N = \E\qty[\bm{g}_m\qty(\co{\bm{g} }_m)^\dagger], \quad \bm{T}_N+ \red{\Sigma_N^\b{\xi }}= \E\qty[\bm{c}_m\bm{c}_m^\dagger], \quad \bm{C}_N := \E\qty[\bm{c}_m\qty(\co{\bm{g} }_m)^\dagger],
	\end{equation}
	and by definition of $\co{\mathbb{P}}$ and the triangle inequality for all $m \in \set{1,\dots,M} $, it holds that
	\begin{equation} \label{bernbound}
		\co{\mathbb{P}}[\abs{\bm{g}_m}^2 \leq 2(\gamma_N^{}+ \co{\gamma}_N^{}) \text{ and } \abs{\bm{c}_m}^2 \leq 2(\gamma_N^{}+ \co{\gamma}_N^{})\text{ and } \abs{\co{\bm{g}}_m}^2 \leq 2(\gamma_N^{}+ \co{\gamma}_N^{})]=1.
	\end{equation}
	Now, analogously to \eqref{g as vector} and \eqref{c as vector}, we define
	\begin{equation} \label{sum def}
		\begin{split}
			\bm{S^G}_{m} & := \frac{1}{M} (\Psi_N(\bm{x}_m)+\bm{\eta}_N^m)(\Psi_N(\bm{x}_m)+\co{\bm{\eta}}_N^m)^\dagger - \bm{G}_N =\bm{g}_m(\co{\bm{g}}_m)^\dagger - \bm{G}_N,    \\
			\bm{S^C}_{m} & := \frac{1}{M} (\Aa\Psi_N(\bm{x}_m)+\bm{\xi}_N^m)(\Psi_N(\bm{x}_m)+\co{\bm{\eta}}_N^m)^\dagger - \bm{C}_N=\bm{c}_m(\co{\bm{g}}_m)^\dagger- \bm{C}_N.
		\end{split}
	\end{equation}
	By the independence of $\bm{g}_m$ and $\bm{c}_m$ in \eqref{gcdef}, their mean value in \eqref{gcmatdef} and the bound in \eqref{bernbound}, we may apply Bernstein's inequality for the covariance defined in Corollary \ref{Bernstein Lemma} to \eqref{sum def}, which implies that
	\begin{align*}
		\co{\mathbb{P}}\qb{\norm{\co{\bm{G}}_{NM}-\bm{G}_N} \geq\delta } & \leq 2 N \exp(\frac{-M \delta ^2 /2 }{ 2(\gamma_N^{}+ \co{\gamma}_N^{})\qty(\norm{\bm{G}_N+\b{\Sigma}_{N}^{\b{\eta}}} + {2 \delta }/{3})}),                             \\
		\co{\mathbb{P}}\qb{\norm{\co{\bm{C}}_{NM}-\bm{C}_N} \geq\delta } & \leq 2 N \exp(\frac{-M \delta ^2 /2 }{ 2(\gamma_N^{}+ \co{\gamma}_N^{})\qty(\max \set{\norm{\bm{T}_N+ {\Sigma_N^\b{\xi }}}, {\norm{\bm{G}_N+\b{\Sigma}_{N}^{{\b{\eta}}}}}} +{2 \delta }/{3})}) .
	\end{align*}
	Solving for $M$ we obtain that, for $M$ as in the problem statement    \begin{equation} \label{chernoff noise}
		p\leq\co{p}_N^\red{M}\co{\mathbb{P}}\qb{\norm{\co{\bm{G}}_{NM}-\bm{G}_N} <\delta, \text{ and } \norm{\co{\bm{C}}_{NM}-\bm{C}_N} <\delta}.
	\end{equation}
	Now, by definition of $\co{\mathbb{P}}$, we have that $\co{p}_N^\red{M}\co{\mathbb{P}}\leq \mathbb{P}$. Using this in \eqref{chernoff noise} shows that, for $M$ as above
	\begin{equation*}
		p\leq\mathbb{P}\qb{\norm{\co{\bm{G}}_{NM}-\bm{G}_N} <\delta, \text{ and } \norm{\co{\bm{C}}_{NM}-\bm{C}_N} <\delta}.
	\end{equation*}
	The result follows identically to the proof of Proposition \ref{bound error projection}.
\end{proof}

The generalisation of Theorem \ref{OC theorem} to the case with noise can be proved in the same way as said theorem, but this time using Proposition~\ref{noise bound} instead of Proposition~\ref{bound error projection}. This gives the following result.

\begin{theorem}[Order of convergence with noise]\label{OC theorem noise}
	Let $\Psi, \bm{\eta}_N, \bm{\xi }_N $ satisfy Assumptions \ref{assumption}, \ref{bn assumption}, and \ref{noise assumption} and $N \in \N$ be arbitrary. Define $\rho_N:= \sqrt{\kappa(\bm{G}_N)} \qty(1+\|\bm{C}_N\|\|\bm{G}_N^{-1}\|)$, let $\varepsilon \in (0, \rho_N)$ be arbitrary and write $\delta_N:=\varepsilon/\qty(2\rho_N\norm{\bm{G}_N^{-1}})$. Then, for all $p\in (0,1), \co{p}_N\in(p^\red{\nicefrac{1}{M}},1)$ and
	\begin{equation*}
		M > \qty(3 \max \set{\norm{\bm{G}_N}+\norm{\b{\Sigma}_{N}^{\b{\eta}}},\norm{\bm{T}_N+ {\Sigma_N^\b{\xi }}}}+2 \delta_N ) \frac{4(\gamma_N^{}+ \co{\gamma}_N^{})}{3 \delta_N ^2}\log \left(\frac{4
			N}{1-p / \co{p}_N ^{{M}}}\right),
	\end{equation*}
	it holds that
	\begin{equation*}
		\mathbb{P}\left[	\norm{\co{\Aa}_{NM}- \Aa_N}\leq \varepsilon \right]\geq p.
	\end{equation*}
	Furthermore, if $f \in \Ff_\infty$ and
	$\norm{(\Pp_{\Ff_N}- Id)\phi}_\Ff = \mathcal{O}( N^{-\alpha})$ for all $\phi \in \Ff$, or if $f\in \Dd$ and additionally $ \norm{(\Pp_{\Dd_N}- Id)f}_\Dd = \mathcal{O}( N^{-\alpha}) $, then, for $N =\mathcal{O}(\varepsilon^{-\nicefrac 1\alpha})$ and $M$ as defined above, it holds that
	\begin{equation*}
		\mathbb{P}\left[{\norm{\co{\Aa}_{NM}\Pp_{\Dd_N} f- \Aa f }_\Ff}\leq \norm{\Aa}\norm{f}_{\Dd}\varepsilon \right] \geq p.
	\end{equation*}
\end{theorem}
Thus, Theorem \ref{OC theorem noise} shows an order of convergence of $\co{\Aa }_{NM} $ to $\Aa_N$ of
\begin{equation}\label{order of convergence noise}
	\begin{split}
	M &=\mathcal{O}\left((\gamma_N^{}+\co{\gamma}_N^{})  \max \set{\norm{\bm{G}_N}+\norm{\b{\Sigma}_{N}^{\b{\eta}}},\norm{\bm{T}_N+ {\Sigma_N^\b{\xi }}}}\kappa \qty(\bm{G}_N)\norm{\bm{C}_N}^2 \norm{\bm{G}_N^{-1}}^4\right.\\&\cdot\left.   \log \left(\frac{
		N}{1-p/\co{p}_N^M}\right) \varepsilon ^{-2}\right).
	\end{split}
\end{equation}

We now discuss this result in the context of Gaussian noise.

\begin{example}\label{noise example 1}
	Let $(\bm{\eta}_N, \bm{\xi}_N,\co{\bm{ \eta}}_N)\sim \mathcal{N}(0, \sigma^2\Id_{3N\times 3N})$. Then, $\sigma^{-2}\abs{(\bm{\eta}_N^m,\bm{\xi}_N^m,\co{\bm{\eta}}_N^m)}^2\sim \chi_{3N}^2$ . Suppose the basis functions are bounded by $1$ so  $\Psi_N $ satisfies Assumption \ref{basis norm assumption} with  $ \gamma_N^{}=N$. Let $M>1$ and write $\co{\gamma}_N^{}:=3\sigma^2 N \red{\log(M)}$. Write $F_k$ for the CDF of the $\chi^2_k$ distribution.   Then, for all $m$
	\begin{equation*}
		1-\co{p}_N=  1-	F_{3N}(3N \log(M)) \leq\left(\log(M) e^{1-\log(M)}\right)^{\nicefrac{3N}{2}}=(\nicefrac{\log(M)e}{M})^{\nicefrac {3N}{2}}
	\end{equation*}
	where in the inequality we used a tail bound property of the $\chi_k^2$ distribution
	\begin{align*}
				1-F_k(z k) \leq\left(z e^{1-z}\right)^{k / 2}, \quad\forall z>1, k \in \N.
	\end{align*}
	As a result, by a Taylor expansion we obtain that for $M \gg 1$ 
	\begin{align*}
		\co{p}_N^\red{M} \geq \qty(1-\qty(\frac{\log(M)e}{M})^{\nicefrac {3N}{2}})^M= 1-M\left(\frac{e \log M}{M}\right)^{3 N / 2}+\mathcal{O}\left(M^2\left(\frac{\log M}{M}\right)^{3 N}\right) \approx 1.
	\end{align*}
	This shows that, in this case, the effect of the noise on the theoretical error bound is small, effectively multiplying it by a factor of $\qty(1+ \sigma ^2)^2 \log(M)$. To see this compare \eqref{order of convergence} with \eqref{order of convergence noise} where $\co{\gamma}_N^{} = \sigma ^2\gamma _N $ and $\norm{\b{\Sigma }^\b{\eta }_N} < \sigma ^2 $. \exampleSymbol
\end{example}
\begin{observation}
The motivation behind the double evaluation of the basis functions is so that the data driven matrices $\co{\b{G}}_{NM}, \co{\b{C}}_{NM}$ have the correct expectation $\b{G}_N, \b{C}_N$. If it is possible to evaluate the basis exactly so that $\b{\eta}_N=0$ then the double evaluation of the basis functions is not necessary. One may work directly with the matrices 
\begin{align*}
	[\co{\bm{{C}}}_{NM}]_{ij} & :=\frac{1}{M}\sum_{m=1}^M\qty(\Aa\psi_i(\bm{x}_m)+\bm{\xi}_N^{m,i})\qty(\psi_j(\bm{x}_m)+{\bm{\eta}}_N^{m,j})^\dagger, \\
	[\co{\bm{{G}}}_{NM}]_{ij} & :=\frac{1}{M}\sum_{m=1}^M\qty(\psi_i(\bm{x}_m))\qty(\psi_j(\bm{x}_m))^\dagger,
\end{align*}
to obtain that under the same conditions as in Theorem \ref{OC theorem noise} also
	\begin{equation*}
		\mathbb{P}\left[    \norm{\co{{\Aa}}_{NM}- \Aa_N}\leq 2\sqrt{\kappa(\bm{G}_N)} \qty(1+\|\bm{C}_N\|\|\bm{G}_N^{-1}\|)\norm{\bm{G}_N^{-1}}\delta\right]\geq p.
	\end{equation*}
\end{observation}

\section{{Weak spectral} convergence}
\label{eigen section}

In Sections \ref{convergence data section}, \ref{convergence Galerkin section}, and \ref{Joint convergence section}, we established the convergence of the data-driven operators $\wh{\Aa}_{NM}$ when $M$ goes to infinity, the convergence of the Galerkin approximations $\Aa_{N}$ when $N$ goes to infinity, and the joint convergence of $\hat{\Aa}_{NM}$ when $N,M$ go to infinity, respectively. This section establishes convergence results for the eigenvalues and eigenfunctions of the approximations $\wh{\bm{A}}_{NM}$ and $\wh{\Aa}_{NM}$. 

Eigenvalues and eigenfunctions of transfer operators play a vital role in the global analysis of complex dynamical systems. The dominant eigenvalues are related to the slowest timescales of the underlying system and the corresponding eigenfunctions contain important information about slowly evolving spatiotemporal patterns and have been, for instance, used to detect stable conformations of molecules or gyres in the ocean, see \cite{SKH23, froyland2014well} for more details.

\red{For a general non-normal operator $\mathcal{A}$, the spectrum can be highly sensitive to the perturbations introduced by the numerical approximation. This is characterized by the operator's pseudospectrum, which can explain the appearance of spurious eigenvalues (spectral pollution) \cite{colbrook2024rigorous}. Additional challenges arise when the underlying dynamical system has a continuous spectrum, which can cause approximate eigenfunctions to converge weakly to zero \cite[Example 4.2]{mezic2022numerical}. For practical approaches designed to mitigate spectral pollution by discarding spurious eigenpairs, we refer the reader to residual-based methods, see \cite{colbrook2023residual} and the approximation of the Koopman generator of continuous time measure preserving flows by operators with discrete spectrum \cite{giannakis2024consistent}.}

\red{Conversely, the spectrum of a normal operator is stable and insensitive to small perturbations. However, this stability is only realized computationally if the discretization scheme preserves this normality. Standard Galerkin methods do not generally satisfy this property. As a result, spectral pollution can still appear even for well-behaved normal operators, as demonstrated in \cite{colbrook2024rigorous}. The most reliable cases for spectral approximation are therefore situations where the operator is normal  and the chosen numerical method is designed to preserve this structure.  }

We again discuss large data and large dictionary limits, first separately and then jointly, establishing three convergence results below. We first remind the reader of the definition of weak convergence in a Hilbert space. Let $(f_n)_{n=1}^\infty$ be a sequence taking values in a Hilbert space $(\mathcal{H}, \langle\cdot, \cdot\rangle_\Hh)$ and $f \in \Hh$. Then $(f_n)_{n=1}^\infty$ converges weakly to $f$, denoted by $ f_n \stackrel{w}{\to} f $, if
\begin{equation*}
	\left\langle f_n, g\right\rangle_\Hh \rightarrow\langle f, g\rangle_\Hh \text { for every } g\in \mathcal{H}.
\end{equation*}
{In what follows it will be important that the weak limit $f$ of the eigenfunctions is non-zero, see Example~2 in \cite{mezic2022numerical}.}

We now establish a convergence result for eigenpairs in the large data limit.
\begin{theorem}[Convergence of eigensystem in data limit] \label{convergence eigenfunctions data theorem}
	Suppose $\Psi$ satisfies Assumption \ref{continuous} and let $({\lambda}_{M}, f_{M})$ be a sequence of eigenvalue and normalised eigenfunction pairs of $\wh{\Aa}_{NM}$, i.e.,
	\begin{equation*}
		\norm{f_{M}}_{\Dd}=1, \quad\wh{\Aa}_{NM} f_{M}=\lambda_Mf_{M}.
	\end{equation*}
	Then there almost surely exists a subsequence of eigenvalue and normalised eigenfunction pairs $\left(\lambda_{M_i}, f_{ M_i}\right)_{i\in\N}$ such that
	\begin{equation*}
		\lambda_{M_i}\to\lambda, \quad f_{M_i} \stackrel{w}{\to} f\in \Ff_N\subset\Dd, \quad\text{when $M_i\to\infty$},
	\end{equation*}
	where the weak convergence is with respect to the inner product on $\Dd$. Furthermore, if $f\ne 0$ and Assumptions \ref{assumption} and \ref{moments assumption} hold, then $(\lambda,f)$ is an eigenvalue and eigenfunction pair of $\Aa_N := \restr{\mathcal{P}_{\Ff_N}\Aa}{\Ff_N}$.
\end{theorem}

\begin{proof}
	It holds that $\left(f_{M}\right)$ is a bounded subsequence of $\Ff_N\subset \Dd$, so by the Banach--Alaoglu Theorem it has a weakly convergent subsequence, see \cite[Section 16.2, Theorem 6]{royden1968real}. In our finite-dimensional case, this weak convergence is equivalent to convergence in the norm $\norm{.}_\Dd$. Additionally, for each $M$, $\wh{\Aa}_{NM}$ is bounded and thus the sequence of eigenvalues $(\lambda_{M})$ is also bounded and has a convergent subsequence. Consequently, there exists a subsequence of eigenvalue and normalised eigenfunction pairs $\left(\lambda_{M_i}, f_{M_i}\right)$ such that
	\begin{equation*}
		\lambda_{M_i}\to\lambda, \quad f_{M_i} \stackrel{}{\to} f, \quad\text{when } M_i\to\infty.
	\end{equation*}
	Supposing that Assumptions \ref{assumption} and \ref{moments assumption} hold, we want to show that $\Aa_N f=\lambda f$. We have
	\begin{equation} \label{c0}
		\Aa_N f=\Aa_N( f-f_{M_i})+\Aa_N f_{M_i}.
	\end{equation}
	The first summand converges to zero because $f_{M_i} \stackrel{}{\to}  f$ and $\Aa_N$ is bounded on $\Dd$.
	Expanding the second summand, we have
	\begin{equation*}
		\lim_{i\to\infty}\Aa_N f_{M_i} =\lim_{i\to\infty}(\Aa_N-\wh{\Aa}_{NM_i}) f_{M_i}+\lim_{i\to\infty}\wh{\Aa}_{NM_i} f_{M_i} =\lim_{i\to\infty}\lambda_i f_{M_i}=\lambda f,
	\end{equation*}
	where in the second equality we used Corollary \ref{convergence data corollary} and in the third we used the convergence of $(\lambda_{M_i}, f_{M_i})$. Taking limits in \eqref{c0} concludes the proof.
\end{proof}
We now move on to the convergence result in the large dictionary limit. {\red{The following theorems include the additional assumption that the limit of the eigenfunctions is nonzero. This does not always hold, as mentioned at the beginning of the section.}}

\begin{theorem}[Weak convergence of eigensystem  in dictionary limit] \label{convergence eigenfunctions dictionary theorem}
	Suppose there exists a sequence $(\lambda_N, f_N)$ of eigenvalue and normalized eigenfunction pairs of $\Aa_N$, i.e.,
	\begin{equation*}
		\norm{f_N}_{\Dd}=1, \quad\Aa_N f_N=\lambda_Nf_N.
	\end{equation*}
	Then there exists a subsequence of eigenvalue and normalised eigenfunction pairs $\left(\lambda_{N_i}, f_{N_i}\right)_{i\in \N}$ such that
	\begin{equation*}
		\lambda_{N_i}\to\lambda, \quad f_{N_i} \stackrel{w}{\to} f, \quad\text{when } N_i\to\infty,
	\end{equation*}
	where $f\in \Dd$ and the weak convergence is with respect to the inner product on $\Dd$. Furthermore, if $f\ne 0$ and Assumption \ref{projection assumption} holds, then $(\lambda,f)$ is an eigenvalue and eigenfunction pair of $\Aa$.
\end{theorem}

\begin{proof}
	It holds that $\left(f_{N}\right)$ is a bounded subsequence of $\Dd$. Additionally, for each $N$, $\Aa_N$ is bounded. Thus, the sequence of eigenvalues $(\lambda_{N})$ is also bounded and has a convergent subsequence. Consequently, there exists a subsequence of eigenvalue and normalised eigenfunction pairs $\left(\lambda_{N_i}, f_{N_i}\right)$ such that
	\begin{equation*}
		\lambda_{N_i}\to\lambda, \quad f_{N_i} \stackrel{w}{\to} f, \quad\text{when } N_i\to\infty.
	\end{equation*}
	Supposing that Assumption \ref{projection assumption} holds, we want to show that $\Aa f=\lambda f$ or, equivalently, that $\langle\Aa f,g\rangle_{\Ff}=\langle\lambda f,g\rangle_{\Ff}$ for all $g\in\Dd$. We have
	\begin{equation} \label{c0d}
		\br{\Aa f,g}_{\Ff}=\br{\Aa( f-f_{N_i}),g}_{\Ff}+\br{\Aa f_{N_i},g}_{\Ff}.
	\end{equation}
	Now, on the one hand, since $f_{N_i} \stackrel{w}{\to} f$ in $\Dd $ and $\Aa\colon\Dd\to\Ff$ is bounded, $\Aa$ maps weakly convergent sequences in $\Dd$ to weakly convergent sequences in $\Ff$, so
	\begin{equation} \label{c1d}
		\lim_{i\to\infty}\br{\Aa( f-f_{N_i}),g}_{\Ff}=0.
	\end{equation}
	On the other hand,
	\begin{align} \label{c2d}
		\begin{split}
			\lim_{i\to\infty}\br{\Aa f_{N_i},g}_{\Ff} & =\lim_{i\to\infty}\br{\qty(\rm{Id}-\Pp_{\Ff_{N_i}})\Aa f_{N_i},g}_{\Ff}+\lim_{i\to\infty}\br{\Aa_{N_i} f_{N_i},g}_{\Ff} \\
			                                          & =\lim_{i\to\infty}\br{\lambda_i f_{N_i},g}_{\Ff}=\br{\lambda f,g}_{\Ff},
		\end{split}
	\end{align}
	where in the first equality we used that by definition $\Aa_{N_i}=\restr{\Pp_{\Ff_{N_i}} \Aa}{\Ff_{N_i}}$, in the second we used that $\br{(\rm{Id}-\Pp_{\Ff_{N_i}})\Aa f_{N_i},g}_\Ff=\br{\Aa f_{N_i},(\rm{Id}-\Pp_{\Ff_{N_i}})g}_\Ff\to 0$ by the self-adjointness of $\Pp_{\Ff_{N_i}}$, the boundedness of $(\Aa f_{N_i})$ in $\Ff$, and Assumption \ref{projection assumption}, and in the third we used the convergence of $(\lambda_{N_i},f_{N_i})$. Taking limits in \eqref{c0d} and using \eqref{c1d} and \eqref{c2d} concludes the proof.
\end{proof}

\red{Conditions under which such a sequence exists can be found in, for example, \cite[Theorem 3.2]{beattie2000galerkin}, which rely on $\Aa _n \approx \Aa $ and $\Dd_n \approx \Dd$.}

From Theorem \ref{joint limit theorem}, we finally obtain the following result regarding the convergence of eigenvalues and eigenfunctions in the joint large data and dictionary limit.

\begin{theorem}[Joint data and dictionary limit of eigensystem]\label{convergence eigenfunctions joint theorem} Let $\Psi $ satisfy Assumption \ref{continuous} and let $({\lambda}_{N M}, f_{N M})_{(N,M)\in\N^2} $ be a sequence of eigenvalue and normalized eigenfunction pairs of $\wh{\Aa}_{NM}$.
	Then there almost surely exists a subsequence of eigenvalue and normalised eigenfunction pairs $\left(\lambda_{N_iM_{N_i}}, f_{N_iM_{N_i}}\right)$ such that, for any sequence $M'_{N_i} \geq M_{N_i}$, almost surely
	\begin{equation*}
		\lambda_{N_i M'_{N_i}}\to\lambda, \quad f_{N_iM'_{N_i}} \stackrel{w}{\to} f,\quad\text{when $N_i\to\infty$},
	\end{equation*}
	where $f\in \Dd$ and the weak convergence is with respect to the inner product on $\Dd$. Furthermore, if $f\ne 0$ and Assumptions \ref{assumption}, \ref{moments assumption}, \ref{projection assumption}, and \ref{projection operator assumption} hold, then $(\lambda,f)$ is an eigenvalue and eigenfunction pair of~$\Aa$.
\end{theorem}

\begin{proof}
	The first part follows from $\left(\lambda_{NM}\right)$ and $\big(f_{NM}\big)$ being bounded sequences of $\C$ and $\Ff_N\subset \Dd$, respectively.
	Let $\big(\lambda_{N_i,M_i},f_{N_i,M_i}\big)$ be such a convergent sequence and suppose that Assumptions \ref{assumption}, \ref{moments assumption}, \ref{projection assumption}, and \ref{projection operator assumption} hold. Then, by Theorem \ref{joint limit theorem}, there exists a subsubsequence $({N_i,M_{N_i}})$ such that, for any $M'_{N_i} \geq M_{N_i}$, almost surely, for any $g\in \Dd$,
	\begin{equation*}
		\left(\wh{\Aa}_{N_iM'_{N_i}}-\Aa\right) \colon (\Ff_{N_i},\norm{.}_\Dd)\to(\Ff,\norm{.}_\Ff)
	\end{equation*}
	converges to zero pointwise when $N_i\to\infty$. It holds that
	\begin{equation*}
		\br{\Aa f,g}_\Ff =\br{\Aa\left(f-f_{N_i,M'_{N_i}}\right),g}_\Ff+\br{(\Aa-\Aa_{N_i,M'_{N_i}})f_{N_i,M'_{N_i}},g}_\Ff+\br{\Aa_ {N_i,M'_{N_i}}f_{N_i,M'_{N_i}},g}_\Ff.
	\end{equation*}
	For $N_i\to\infty$, the first summand converges to zero because $\Aa$ is bounded and $f_{N_i,M'_{N_i}}\stackrel{w}{\to} f$, the second summand does the same because of how $({N_i,M'_{N_i}})$ was chosen, and the third summand converges to $\br{\lambda f, g}$ because $\lambda_{N_iM'_{N_i}}\to\lambda$. Thus,
	\begin{equation*}
		\br{\Aa f,g}_\Ff =\br{\lambda f, g}_\Ff
	\end{equation*}
	for all $g\in \Dd $, which concludes the proof.
\end{proof}

\section{Numerical experiments} \label{simulations section}

We will now illustrate the derived convergence results and error bounds and test their sharpness.

\subsection{Benchmark problems}\label{benchmark problems section}

We analyse transfer operators associated with deterministic and stochastic systems of the form \eqref{SDE}. In particular, we consider:
\begin{enumerate} \setlength{\itemsep}{0ex}
	\item The ODE defined by
	      \begin{equation} \label{ODE1}
		      \bm{b}({\bm{x}}) = \begin{bmatrix}
			      \gamma x_1
			      \\
			      \delta (x_2 -x_1^2)
		      \end{bmatrix},
	      \end{equation}
	      with $\gamma =-0.8, \delta = -0.7$.
	\item The overdamped Langevin dynamics corresponding to the two-dimensional double-well potential $V(x)=(x_1^2-1)^2 + x_2^2$ with anisotropic diffusion, whose drift and diffusion terms are given by
	      \begin{equation}\label{double well}
		      \bm{b}(\bm{x})=\left[\begin{array}{c}
				      4 x_1-4 x_1^3 \\
				      -2 x_2
			      \end{array}\right], \quad  \bm{\sigma(x)}=\left[\begin{array}{cc}
				      0.7 & x_1 \\
				      0   & 0.5
			      \end{array}\right],
	      \end{equation}
	      respectively.
	\item The one-dimensional Ornstein--Uhlenbeck process given by
	      \begin{equation} \label{OU}
		      b(x)=- \alpha x , \quad \sigma(x)= \sqrt{\frac{1}{2 \beta}},
	      \end{equation}
	      This SDE has a unique invariant distribution $\mathcal{N}\qty(0,\frac{1}{2\beta})$. We choose $\alpha=1$ and $\beta=2$.
\end{enumerate}
For each of the dynamical systems, we aim to reconstruct the Koopman generator, the Perron Frobenius generator or the Koopman operator. We will specify the precise setup below. In each case, we compute the normalized error
\begin{equation} \label{normalized error}
	\varepsilon_N := \frac{\big\|\bm{A}_N-\wh{\bm{A}}_{NM}\big\|}{\big\|\bm{A}_N\big\|}
\end{equation}
of the data-driven matrix $\wh{\bm{A}}_{NM}$ with respect to the true Galerkin projection $\bm{A}_N$ in \eqref{Koopman generator}, where we consider as a proxy for $\bm{A}_N$ the result of gEDMD with a large number of data points. In the two-dimensional case, we use the domain $\X =[-2,2] \times [-1,1]$ and in the one-dimensional case, we define $\X =[-2,2]$. In all cases, we consider the Lebesgue measure on $\X$. For the three dynamical systems introduced above, we now study three different scenarios: the large data limit, the large dictionary limit, and the large data limit when using noisy data.

\subsection{Numerical results as the number of data points tends to infinity} \label{simulations data limit section}

In our first experiment, we choose as basis functions monomials of order up to $k$, defined by
\begin{equation*}
	\Psi^{\rm{MON}} = \set{\bm{x}^\alpha }_{\abs{\alpha} \leq k},
\end{equation*}
and Gaussians centred at a collection of equidistant grid points $\{\bm{p}_{n}\}_{n=1}^N$, i.e.,
\begin{equation*}
	\Psi^{\rm{GSN}} = \set{\exp\qty(-\frac{\norm{\bm{x}-\bm{p}_n}^2}{2 \theta^2})}_{n=1}^N,
\end{equation*}
where for two-dimensional problems we use
\begin{equation*}
	\big\{\bm{p}_{n}\big\}_{n=1}^N=\set{\qty(\frac{i}{2}-2,\frac{j}{2}-1), i=0,\dots,8, j=0,\dots,4},
\end{equation*}
and for one-dimension problems
\begin{equation*}
	\{\bm{p}_{n}\}_{n=1}^N=\set{\frac{i}{2}-2, i=0,\dots,8}.
\end{equation*}

\red{These basis functions are independent and smooth, thus satisfying Assumption \ref{assumption}. Moreover, since $\X$ is compact,  they satisfy the boundedness assumption of Assumption~\ref{bn assumption}. All these basis functions are dense in $H^k(\X)$ for any $k \in \N_{\geq 0}$ , see \cite{bernardi1992polynomial} and \cite[Theorem 6]{wu1993local}. All the operators considered have $\Dd $ as a Sobolev space and as a result, Assumptions~\ref{projection assumption} and \ref{projection operator assumption} are satisfied.}

We also use a finite element method (FEM) basis of piecewise linear functions with $0$ boundary condition on a uniform mesh. That is, if we write $\set{\bm{v}_j}_{j=1}^N$ for the non-boundary vertices of the mesh, we have that $\psi^{\rm{FEM}}$ are the only continuous piecewise linear functions that satisfy
\begin{equation*}
	\psi^{\rm{FEM}}_i(\bm{v}_j)= \delta _{ij}, \quad\forall i,j=1,\dots,N,
\end{equation*}
and write $\Psi^{\rm{FEM}} = \set{\psi^{\rm{FEM}}_n}_{n=1}^N$.
The basis functions in $\Psi^{\rm{FEM}}$ are only once weakly differentiable and as a result do not belong to $\Dd$ when the operator is a second-order differential operator. However, the structure matrix $\bm{C}$ can still be calculated for second-order operators using Observation~\ref{smoothness observation}.

We set the maximum degree of the monomials to $k=8$ so that $N = \binom{k+d}{k}$. That is, $N=45$ when the dimension is $d=2$ and $N=9$ when $d=1$. Additionally, we take a uniform mesh with $45$ and $9$  non-boundary nodes in the $2D$ and $1D$ cases, respectively. In this way, there are the same number of observables in $\Psi^{\rm{MON}},\Psi^{\rm{GSN}}$ and $\Psi^{\rm{FEM}}$ . We then set $\theta = \frac{1}{2N}$ so that the basis functions are well separated. That is, $\theta  < \frac{1}{2}\min_{i\ne j}\norm{\bm{p}_i-\bm{p}_j}$. The points $\bm{x}_1,\ldots,\bm{x}_M$ are sampled uniformly and independently from $\X$ with the Lebesgue measure.

\red{Since we know the dynamical system, the operator $\Aa $ and its action on the basis functions is known exactly. This allows for the exact computation of $\wh{\b{C}}_{NM}, \wh{\b{G}}_{NM}$ and $\wh{\b{A}}_{NM}$ in \eqref{approximate matrices}--\eqref{approximate Galerkin}. We apply this} to approximate the Koopman semigroup, and the Koopman and Perron--Frobenius generators associated with \eqref{ODE1}, \eqref{double well}, and \eqref{OU} for an increasing number of data points using $M=2^8,2^9,\dots,2^{19}$ and compute \eqref{normalized error}, where, since in general we do not have access to $\bm{A}_N$, we approximate it by $\wh{\bm{A}}_{NM}$ with $M=2^{20}$. We repeat this process $50$ times for each $M$ to calculate the average normalized operator error $\varepsilon$. In Figure \ref{figDataConv}, we plot in log-log scale the relationship between $M$ and $\varepsilon$, including a 95\% confidence interval for the error. To serve as a reference, we show dashed lines with slope $-\frac{1}{2}$ and $-1$, respectively.

As can be seen, for all choices of basis functions and all systems, the error has a slope of approximately $-\frac{1}{2}$. This is in accordance with Theorem~\ref{OC theorem} as when $N$ is fixed we obtain $\varepsilon = \mathcal{O}(M^{-\nicefrac12})$. In Figures \ref{imgDoubleWell} and \ref{imgDoubleWellEDMD}, we see that the error of the approximation using FEM basis functions is quite large. This is to be expected as for small $M$ it is possible for an element of the mesh to have few points. If this is the case, the empirical Gram matrix is close to singular (see the comments around Example \ref{FEM example}). This also results in initially large confidence interval for the error, which, when represented on a log-log plot, creates a strong visual effect. The error decays at the expected rate. In Figure \ref{imgOU}, the error using monomial basis functions becomes zero. This is because the subspace spanned by the monomial basis functions is invariant under the Koopman generator of the OU system, see Corollary~\ref{exact approximation}.

\begin{figure}
	\centering
	\begin{subfigure}{.32\linewidth}
		\centering
		\caption{}
		\includegraphics[width=\linewidth]{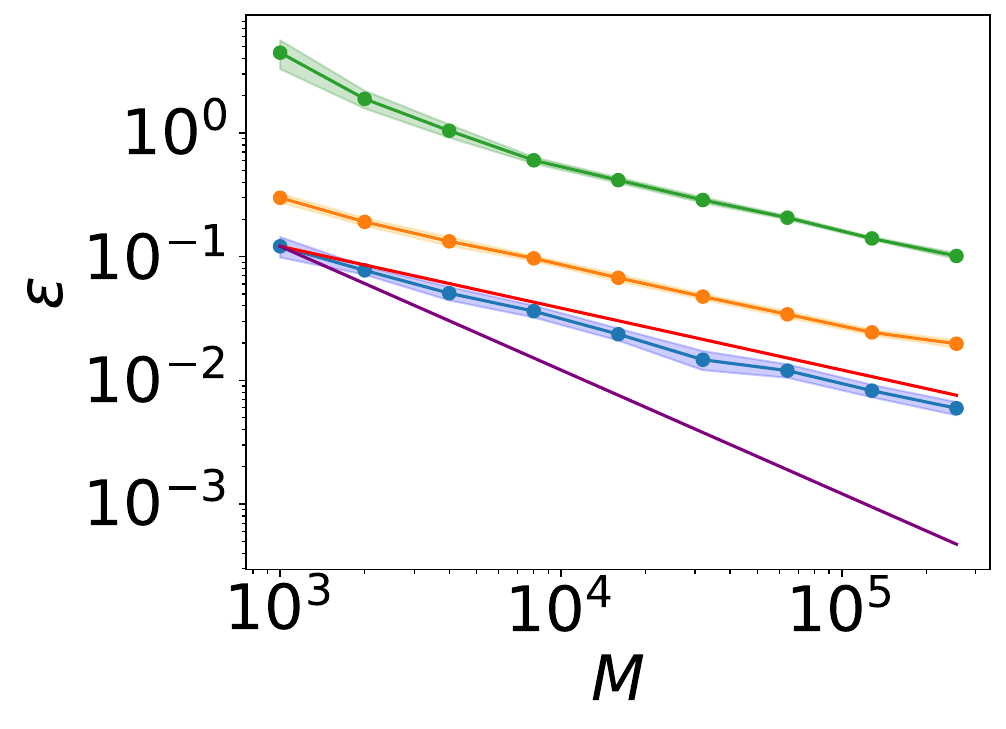}
		\label{imgODE}
	\end{subfigure}
	\begin{subfigure}{.32\linewidth}
		\centering
		\caption{}
		\includegraphics[width=\linewidth]{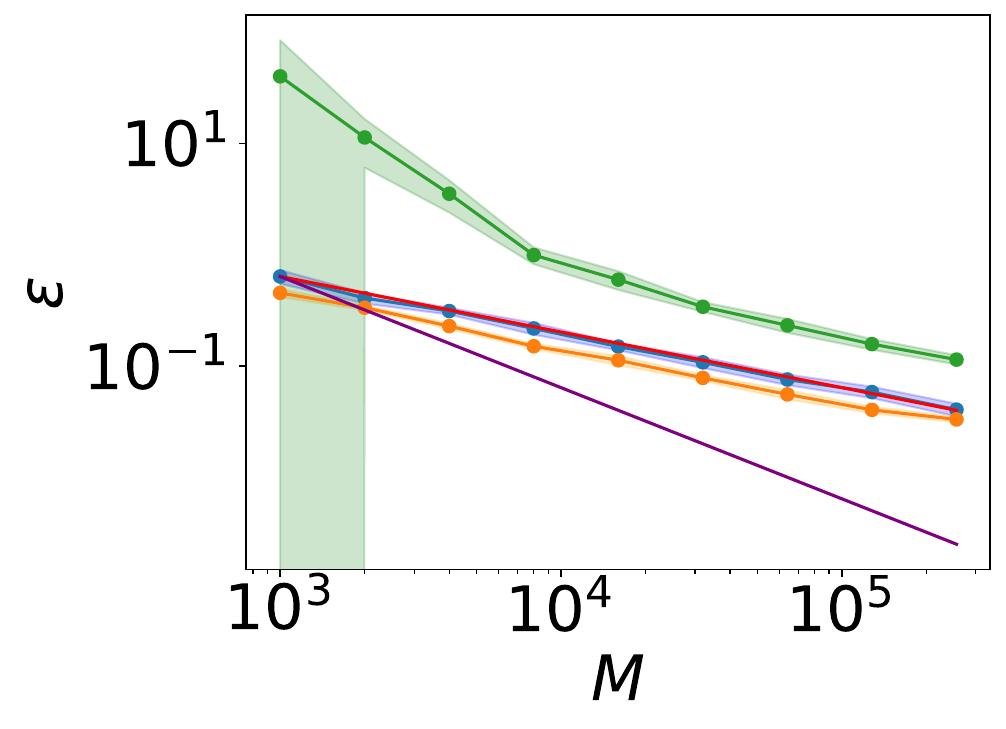}
		\label{imgDoubleWell}
	\end{subfigure}
	\begin{subfigure}{.32\linewidth}
		\centering
		\caption{}
		\includegraphics[width=\linewidth]{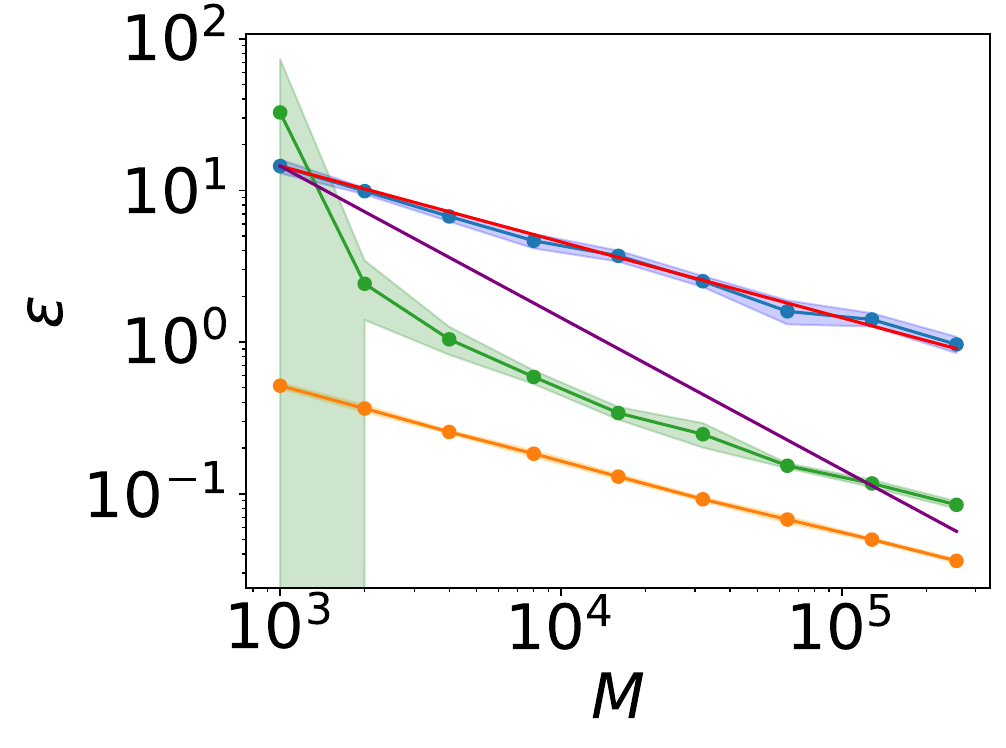}
		\label{imgDoubleWellEDMD}
	\end{subfigure}\\
	\begin{subfigure}{.32\linewidth}
		\centering
		\caption{}
		\includegraphics[width=\linewidth]{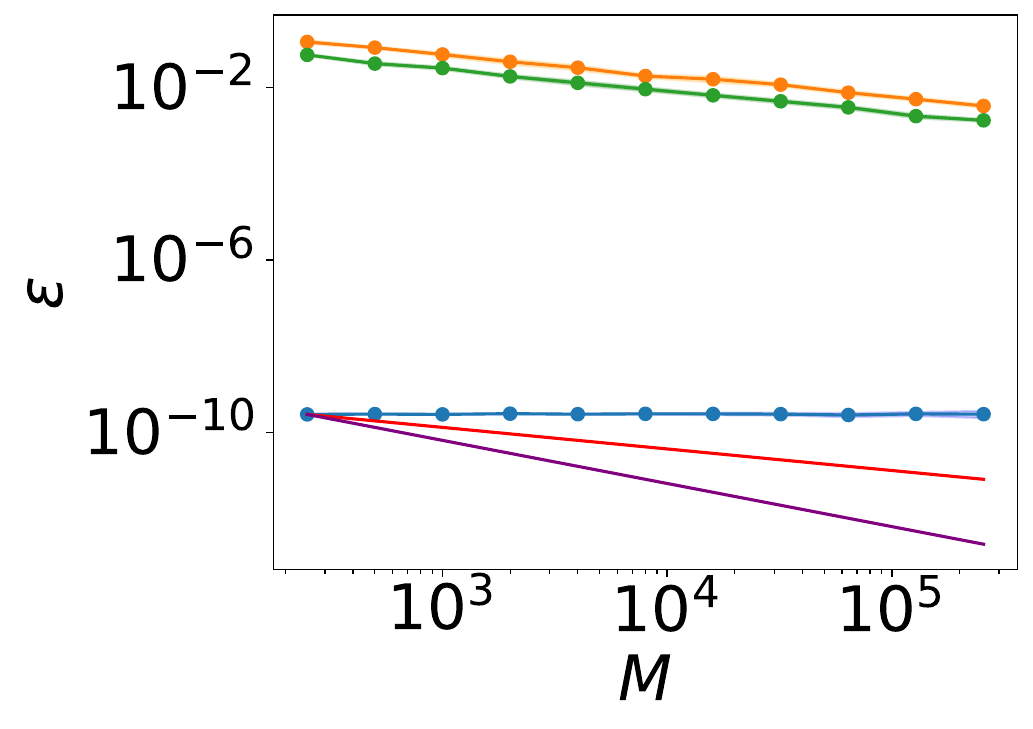}
		\label{imgOU}
	\end{subfigure}
	\begin{subfigure}{.32\linewidth}
		\centering
		\caption{}
		\includegraphics[width=\linewidth]{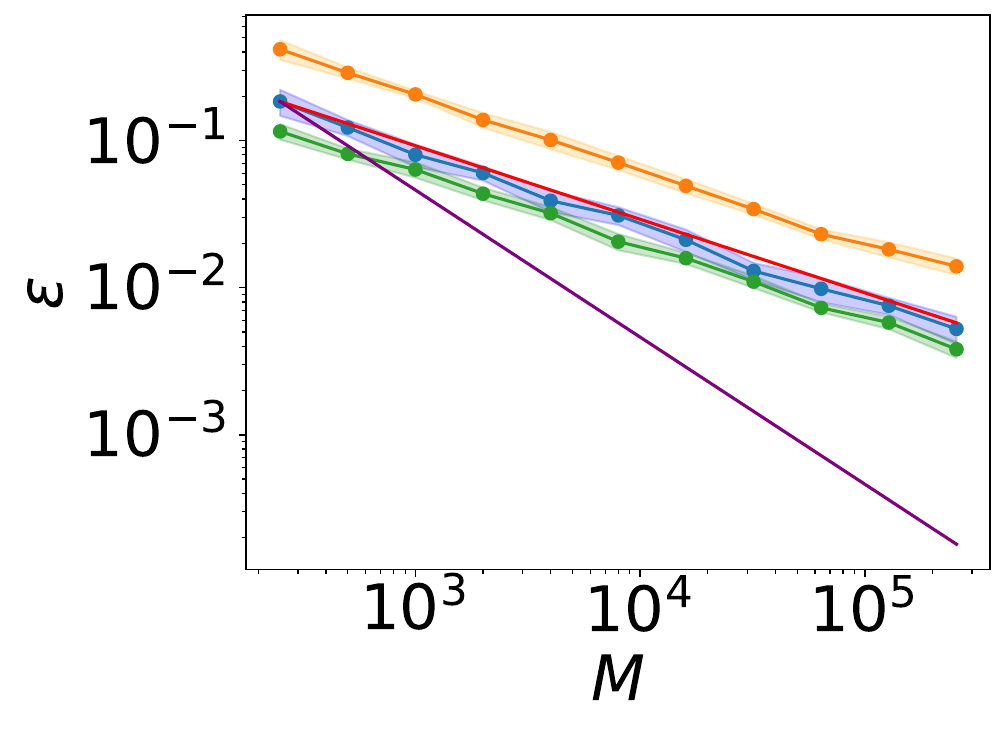}
		\label{imgOU_PF}
	\end{subfigure}
	\begin{subfigure}{.32\linewidth}
		\centering
		\caption{}
		\includegraphics[width=\linewidth]{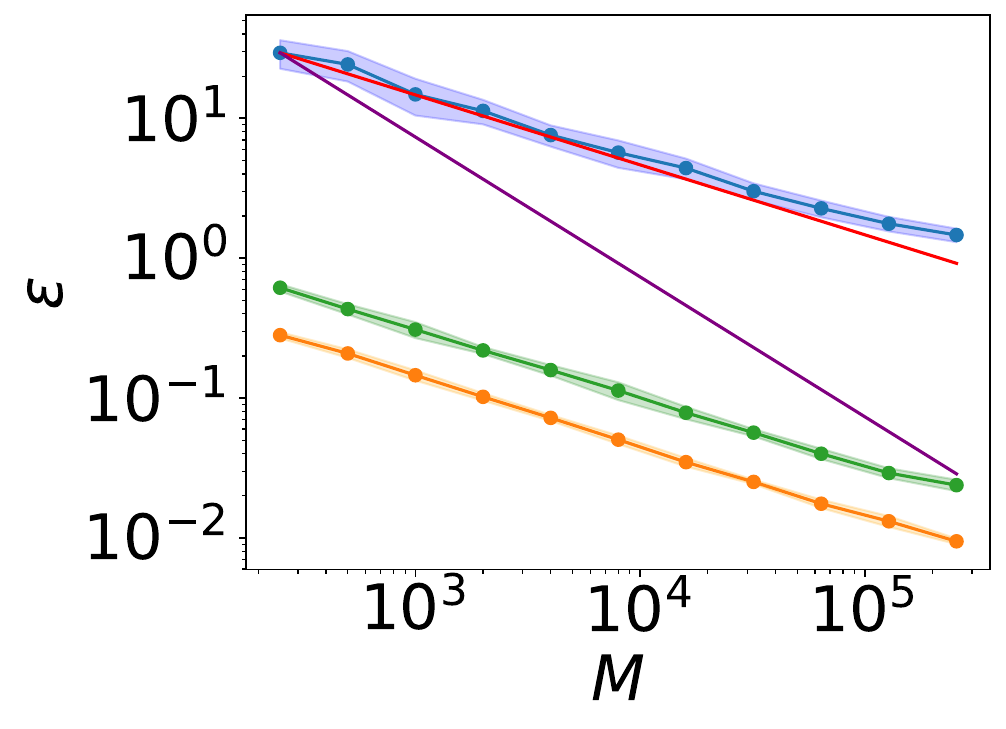}
		\label{imgOU_EDMD}
	\end{subfigure}\\
	\begin{subfigure}{\linewidth}
		\centering
		\includegraphics[width=0.8\linewidth]{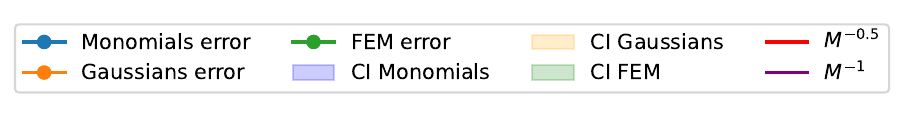}
		\label{legend}
	\end{subfigure}
	\caption{Average normalized error $\varepsilon:=\E\qty[\norm{\smash{\wh{\bm{A}}_{NM}-\bm{A}_N}}/\norm{\bm{A}_N}]$ as a function of the number of data points $M$ for:  the Koopman generator of the ODE \eqref{ODE1} in Figure~\ref{imgODE},  the Koopman generator and Koopman operator for the double-well potential \eqref{double well} in Figures~\ref{imgDoubleWell} and \ref{imgDoubleWellEDMD}, and  the Koopman generator, the Perron--Frobenius generator and the Koopman operator for the OU process in Figures~\ref{imgOU}, \ref{imgOU_PF}, and \ref{imgOU_EDMD}. In all cases, monomials up to order $8$ and the same number of Gaussian observables and FEM basis functions are used. The red and purple lines represent the slopes $-\frac{1}{2}$ and $-1$, respectively. The blue, red and green lines represent the average error over $50$ simulations of the above approximations. The shaded areas represent the 95\% confidence intervals for the respective errors.}
	\label{figDataConv}
\end{figure}

\subsection{Numerical results as the number of dictionary elements tends to infinity}

We now study the effect of the number of dictionary elements on the operator error. To do so, we consider as our basis functions the dictionary comprising Gaussian functions
from the previous subsection. We partition the domain into $N$ equally sized quadrants and define $\{\bm{p}_n\}_{n=1}^N$ to be the centres of these quadrants. We apply the data-driven algorithms with $M=10^4$ data points. We repeat this process $50$ times and calculate the average normalized operator error where again $\bm{A}_N$ is approximated by $\wh{\bm{A}}_{NM}$ with $M=10^{5}$. We then increase $N$ from $4$ to $1024$. In this case, the monomial basis functions are not chosen as for higher orders, the matrix $\bm{G}_N$ becomes ill-conditioned (for example, for monomials of order $10$ in two dimensions, we have that $\kappa(\bm{G}_N)\geq 10^{28}$). When approximating the Koopman operator for the OU process, we also use the piecewise linear FEM basis functions $\Psi^{\rm{FEM}}$. However, these are not used for gEDMD as the theoretical error requires the calculation of $\bm{T}_{ij}= \br{\Ll \psi_i,\Ll \psi_j  }$. For gEDMD, $\Ll $ is a second order operator and $\psi_i^{\rm{FEM}}$ are not twice differentiable.  Since we plot normalised errors, we cap the theoretical bound in Theorem \ref{OC theorem} in the plots at one. As can be seen in Figure \ref{figDictConv}, the error increases with the number of observables as expected in view of Theorem~\ref{OC theorem}. We observe that the theoretical bound on the error increases faster than the simulation error. This indicates that perhaps there is some underlying structure which could make the bounds tighter. Additionally, we see that the confidence interval for the FEM basis functions becomes larger as the number of basis functions increases. This is because of the increased likelihood that very few particles end up in some interval of the partition. This leads to close to a singular mass matrix and a large condition number. See the comments around Example~\ref{FEM example}.

\begin{figure}
	\centering
	\begin{subfigure}{.32\linewidth}
		\centering
		\caption{}
		\includegraphics[width=\linewidth]{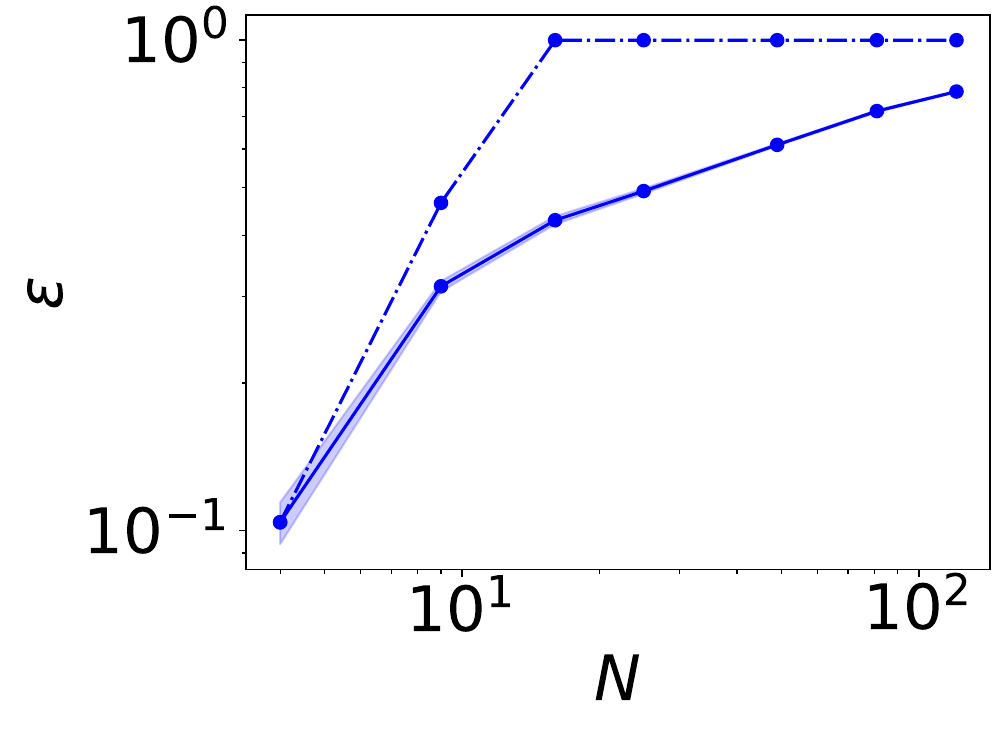}
		\label{imgODEdict}
	\end{subfigure}
	\begin{subfigure}{.32\linewidth}
		\centering
		\caption{}
		\includegraphics[width=\linewidth]{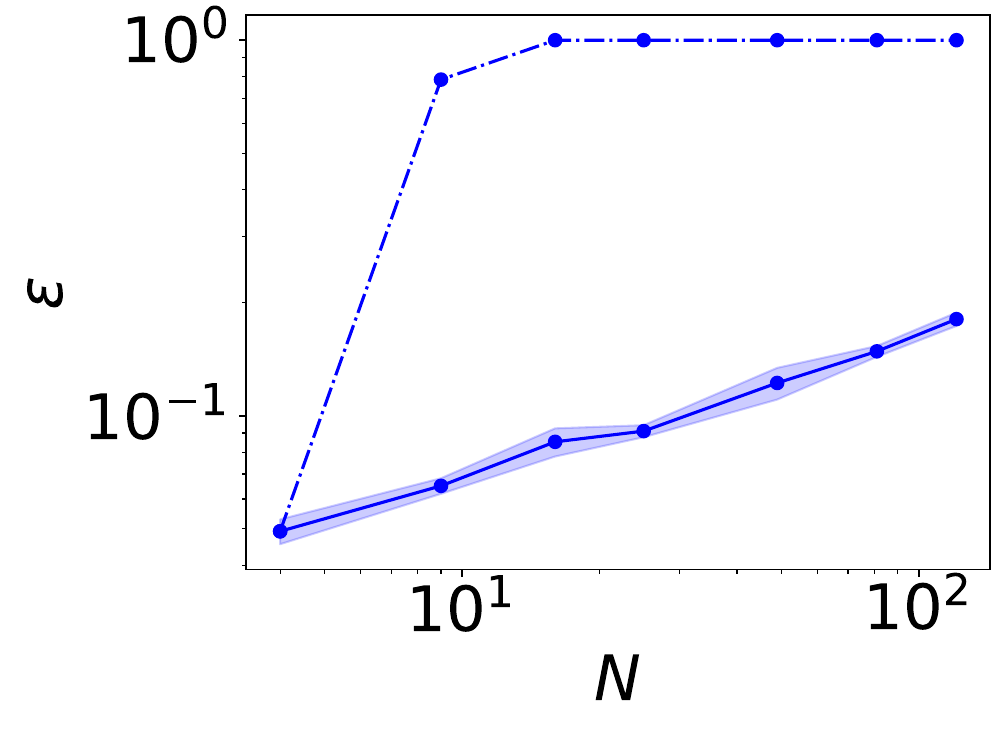}
		\label{imgDoubleWelldict}
	\end{subfigure}
	\begin{subfigure}{.32\linewidth}
		\centering
		\caption{}
		\includegraphics[width=\linewidth]{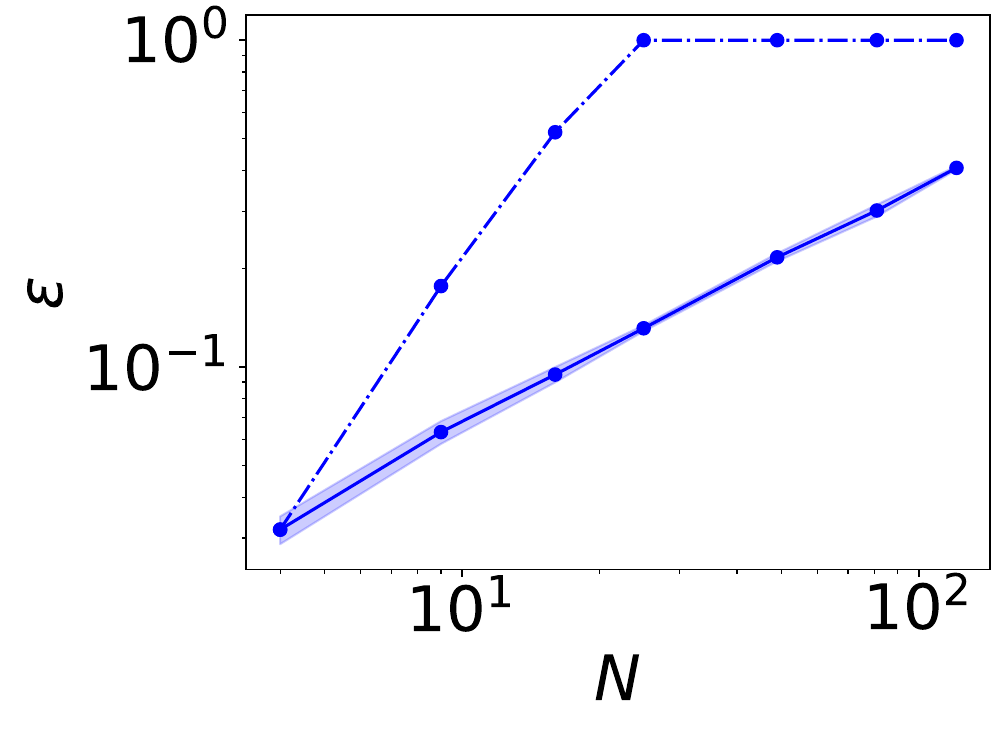}
		\label{double_well_EDMDdict}
	\end{subfigure}
	\begin{subfigure}{.32\linewidth}
		\centering
		\caption{}
		\includegraphics[width=\linewidth]{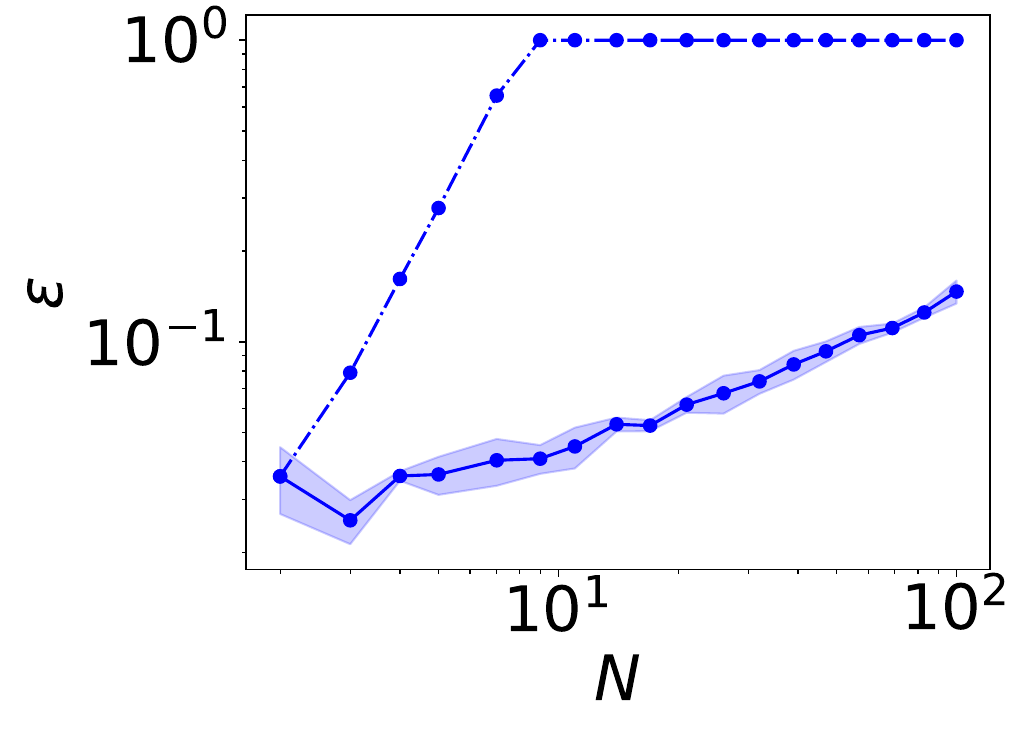}
		\label{imgOUdict}
	\end{subfigure}\begin{subfigure}{.32\linewidth}
		\centering
		\caption{}
		\includegraphics[width=\linewidth]{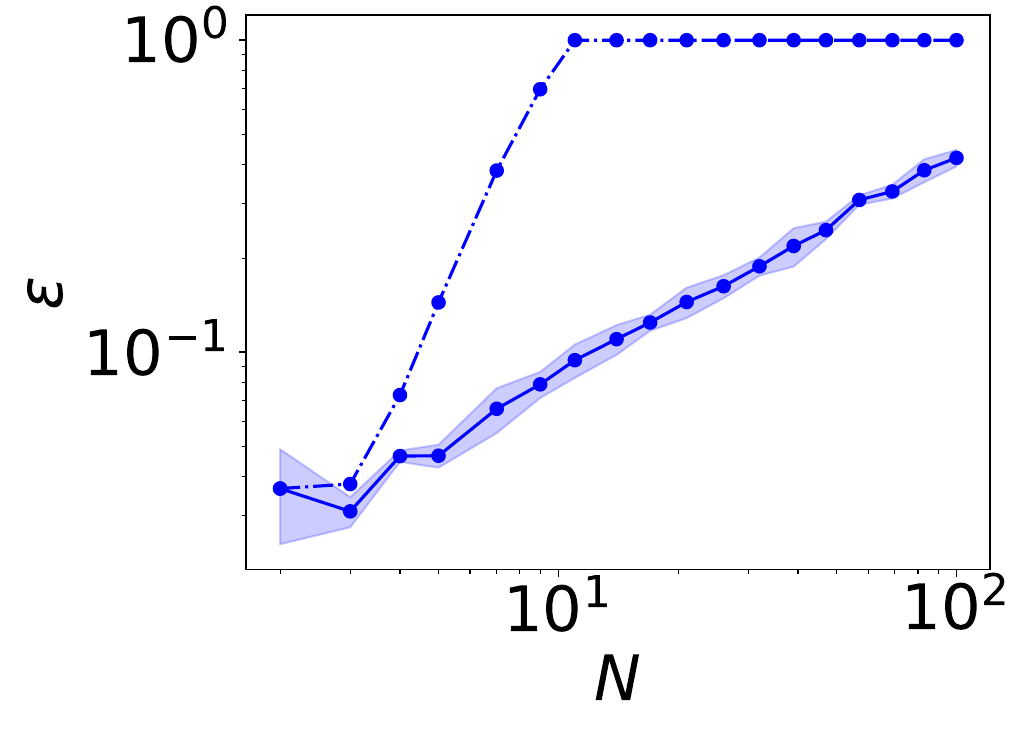}
		\label{imgOU_PFdict}
	\end{subfigure}
	\begin{subfigure}{.32\linewidth}
		\centering
		\caption{}
		\includegraphics[width=\linewidth]{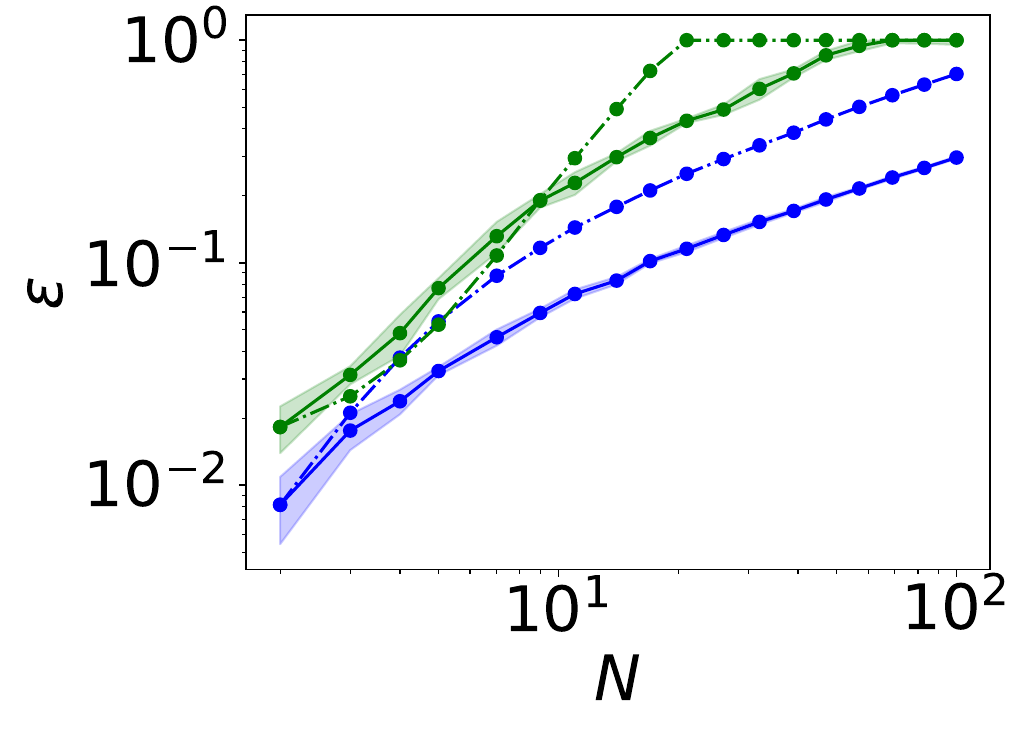}
		\label{imgOU_EDMDdict}
	\end{subfigure}
	\begin{subfigure}{\linewidth}
		\centering
		\includegraphics[width=0.8\linewidth]{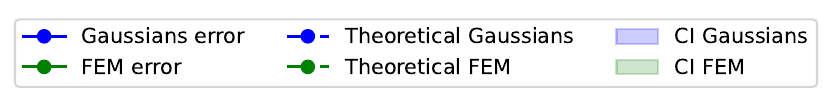}
		\label{legenddict}
	\end{subfigure}
	\caption{Average normalized error $\varepsilon:=\E\qty[\norm{\smash{\wh{\bm{A}}_{NM}-\bm{A}_N}}/\norm{\bm{A}_N}]$ and the theoretical error bound in Theorem \ref{OC theorem} as a function of the number of observables $N$ for the Koopman generator of the ODE \eqref{ODE1} in Figure~\ref{imgODEdict}, the Koopman generator and Koopman operator for the double-well system \eqref{double well} in Figures \ref{imgDoubleWelldict} and \ref{double_well_EDMDdict}, and  the Koopman generator, the Perron--Frobenius, and Koopman operator for the OU process \eqref{OU} using up to $1024$ Gaussian functions in Figures \ref{imgOUdict}, \ref{imgOU_PFdict}, and \ref{imgOU_EDMDdict}.}
	\label{figDictConv}
\end{figure}

\subsection{Numerical results with noise}

In this section, we repeat the experiments carried out in Section \ref{simulations data limit section} with the addition of a noise term as in \eqref{data2}, where we take normal i.i.d.\ noise, i.e.,
\begin{equation*}
	\qty(\bm{\eta}_n,\bm{\xi}_n) \sim \mathcal{N}\qty(0,\sigma^2 \bm{I}_{2n \times 2n}).
\end{equation*}
We increase the noise using $\sigma = 10^{-3}, 10^{-2}, 10^{-1}$ and study its effect on the normalised error. To limit the number of plots, we restrict ourselves to gEDMD for the ODE \eqref{ODE1} and gEDMD for the Perron--Frobenius operator associated with the Ornstein--Uhlenbeck process.

As we can see, the Gaussian and FEM observables are more resilient to increased noise. The FEM basis functions perform the best in the presence of noise. This is expected as we evaluated the FEM basis functions exactly to $0$ on all points outside of their support and only added noise to the non-zero ones. The lower resilience of the monomials compared to the other two is due to the fact that the condition number of the Gram matrix of the monomials is larger and thus the inverse of the Gram matrix is very sensitive to noise. In the second and third rows of Figure~\ref{figNoiseError}, the monomials are slightly more resilient to the noise. This, however, is due to the fact that, since the dimension of the domain is $1$ instead of $2$, there are fewer monomials and thus the condition number of the Gram matrix is smaller. It can be seen that even when the error without noise is exactly zero (see Figure~\ref{imgOU}), the monomial basis functions are not resilient to noise.
\begin{figure}
	\centering
	\begin{subfigure}{.32\linewidth}
		\centering
		\caption{}
		\includegraphics[width=\linewidth]{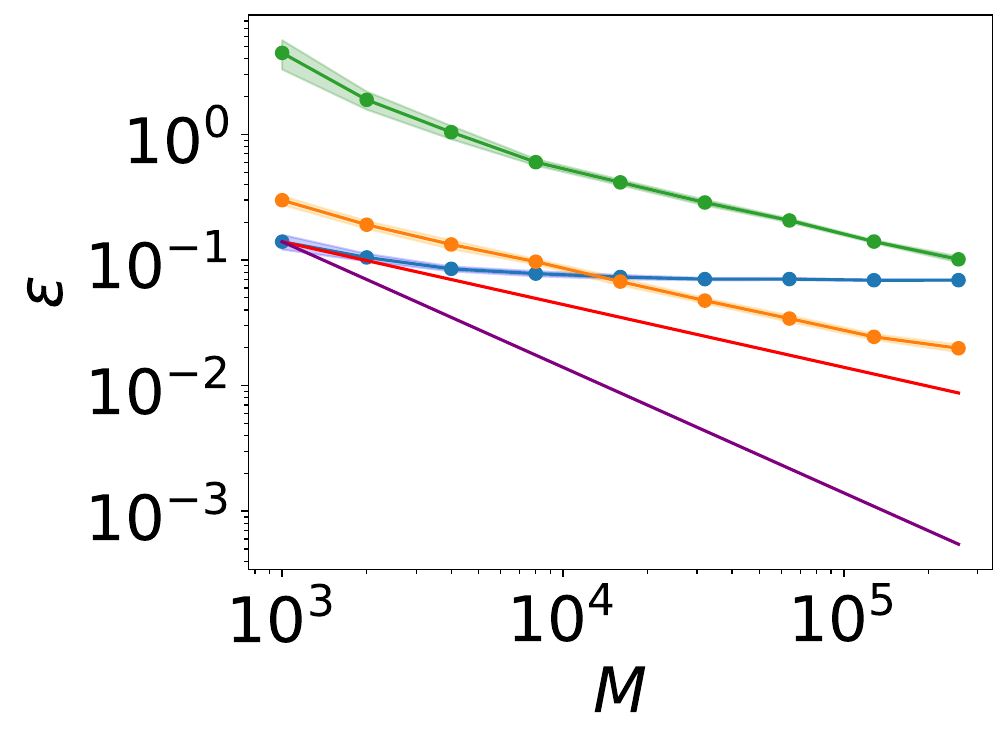}
		\label{ODE_small_noise}
	\end{subfigure}
	\begin{subfigure}{.32\linewidth}
		\centering
		\caption{}
		\includegraphics[width=\linewidth]{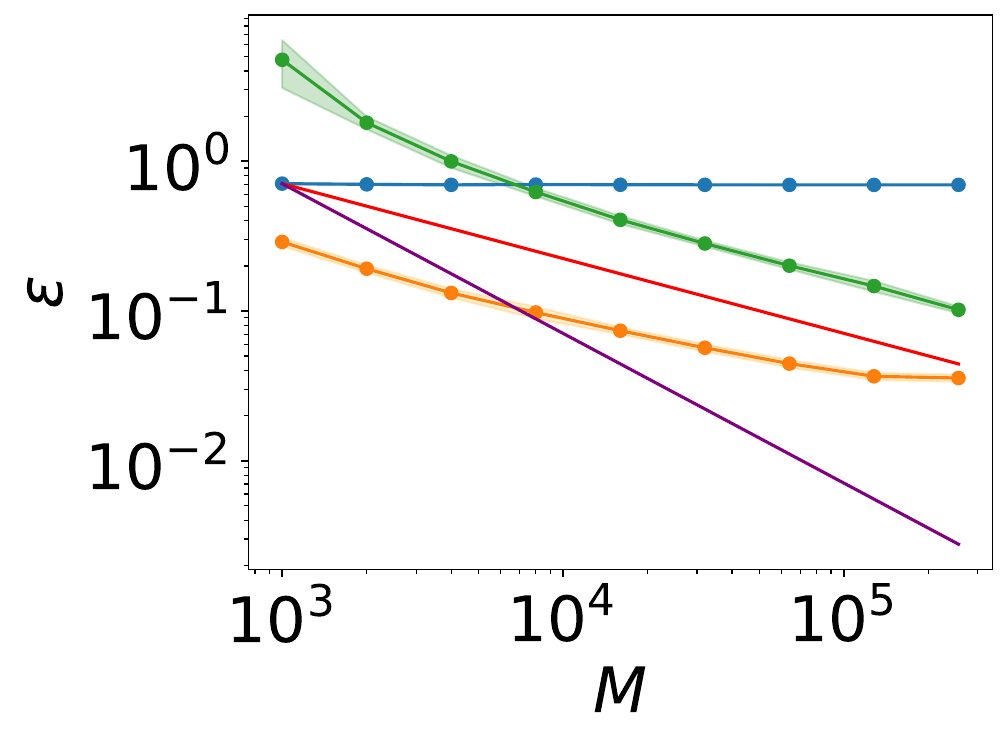}
		\label{ODE_medium_noise}
	\end{subfigure}
	\begin{subfigure}{.32\linewidth}
		\centering
		\caption{}
		\includegraphics[width=\linewidth]{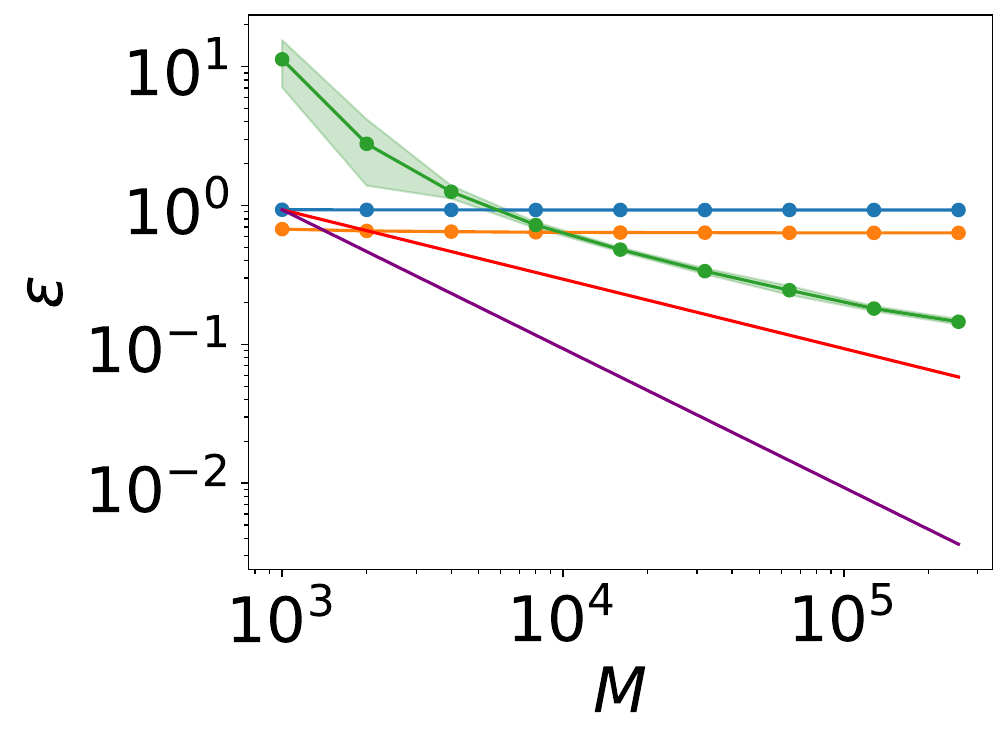}
		\label{ODE_large_noise}
	\end{subfigure}
	\begin{subfigure}{.32\linewidth}
		\centering
		\caption{}
		\includegraphics[width=\linewidth]{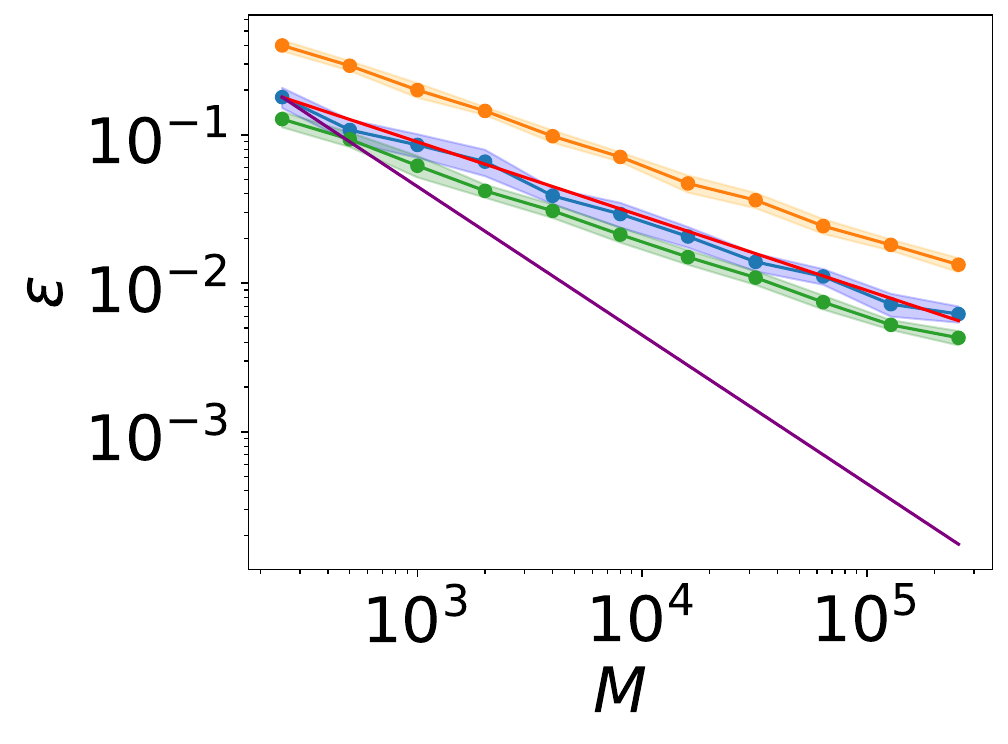}
		\label{OU_PF_small_noise}
	\end{subfigure}
	\begin{subfigure}{.32\linewidth}
		\centering
		\caption{}
		\includegraphics[width=\linewidth]{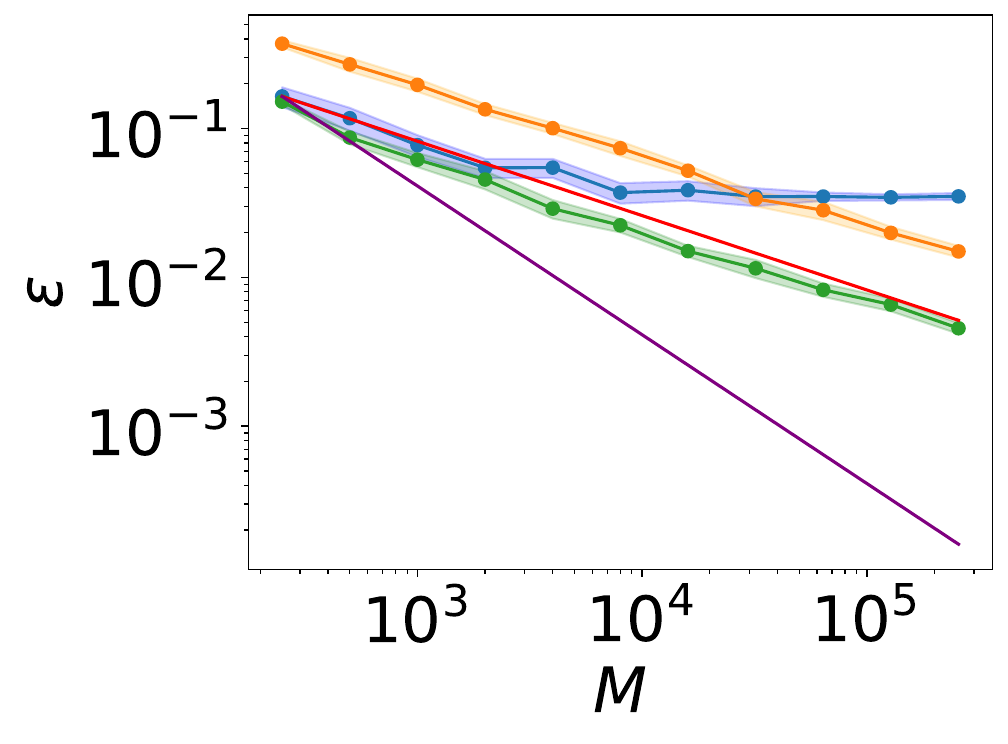}
		\label{OU_PF_medium_noise}
	\end{subfigure}
	\begin{subfigure}{.32\linewidth}
		\centering
		\caption{}
		\includegraphics[width=\linewidth]{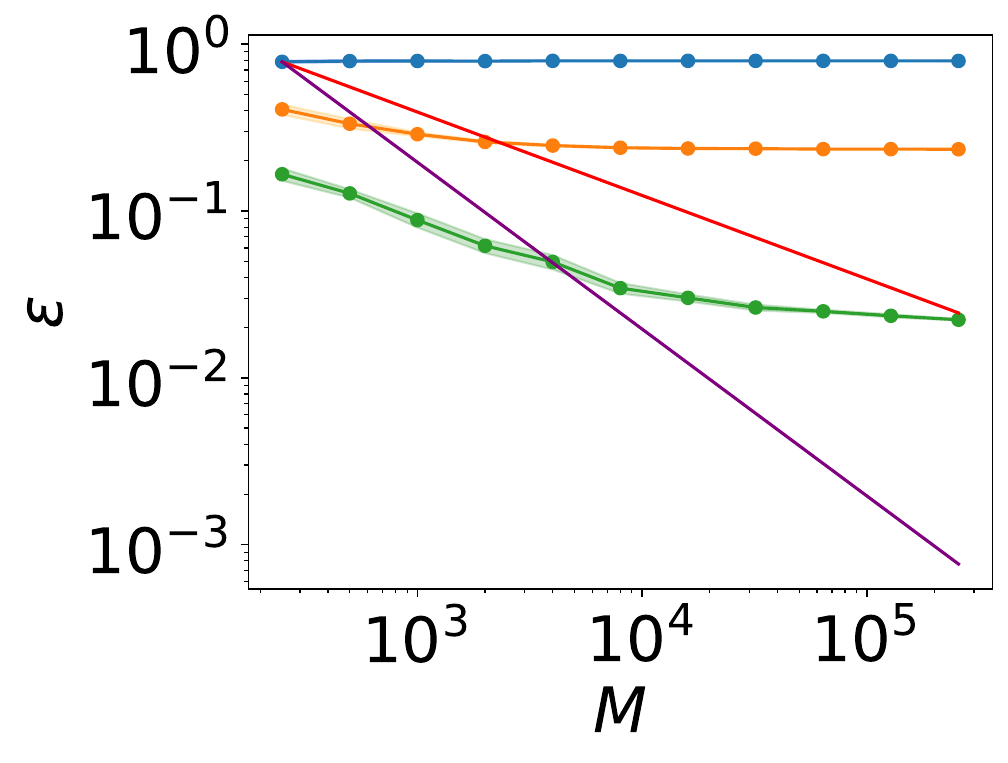}
		\label{OU_PF_large_noise}
	\end{subfigure}
	\begin{subfigure}{.32\linewidth}
		\centering
		\caption{}
		\includegraphics[width=\linewidth]{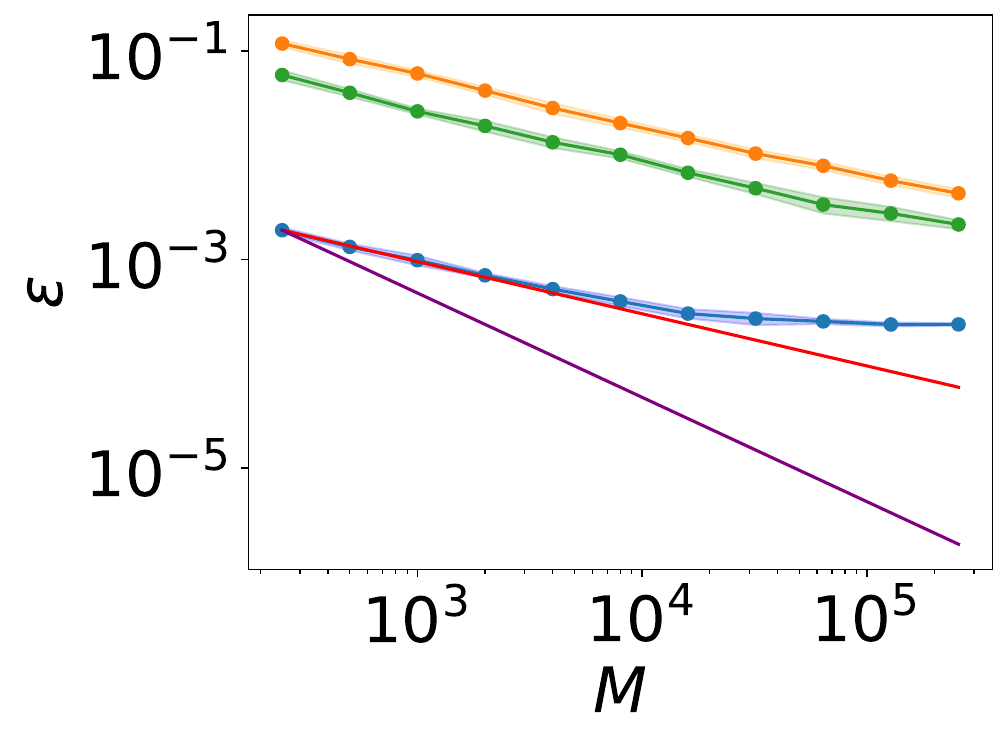}
		\label{OU_small_noise}
	\end{subfigure}
	\begin{subfigure}{.32\linewidth}
		\centering
		\caption{}
		\includegraphics[width=\linewidth]{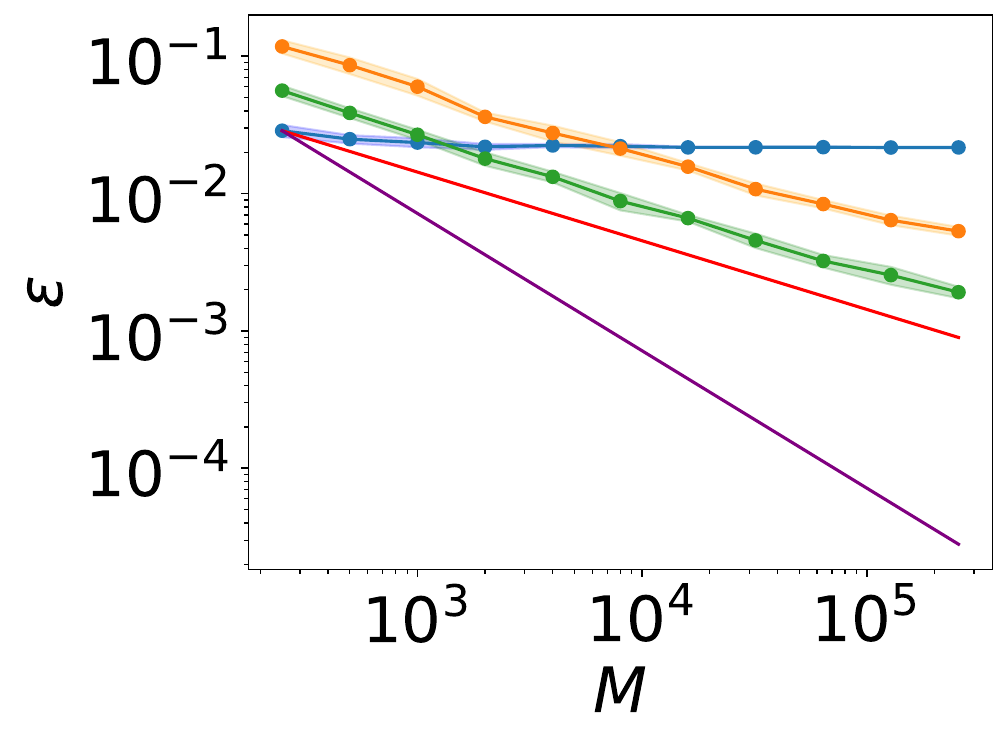}
		\label{OU_medium_noise}
	\end{subfigure}
	\begin{subfigure}{.32\linewidth}
		\centering
		\caption{}
		\includegraphics[width=\linewidth]{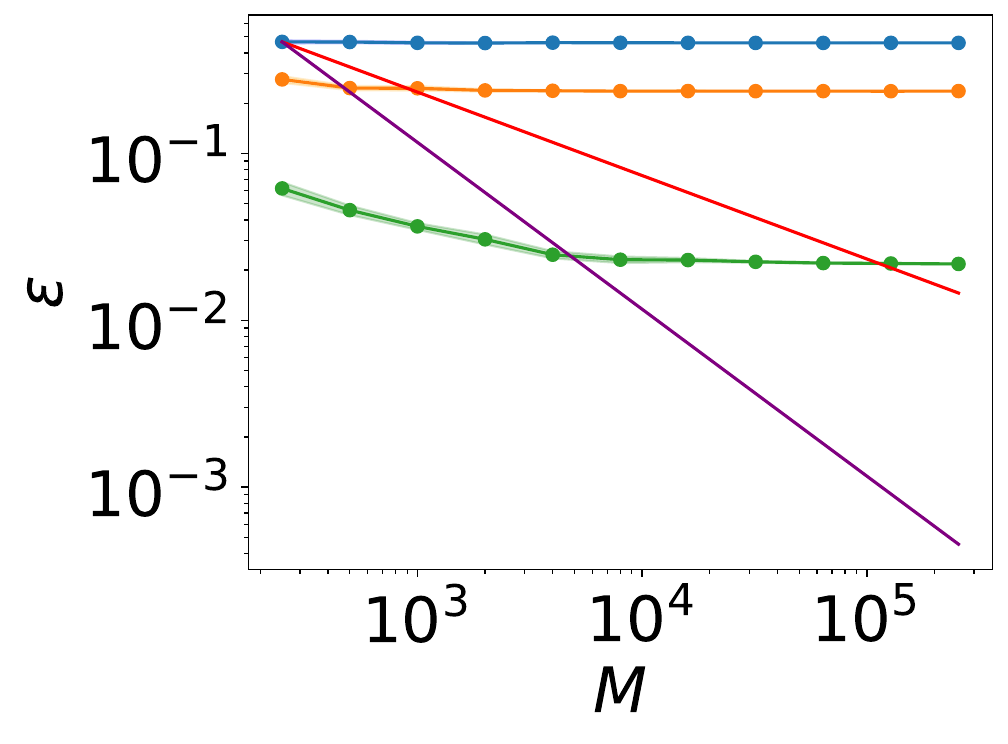}
		\label{OU_large_noise}
	\end{subfigure}
	\begin{subfigure}{\linewidth}
		\centering
		\includegraphics[width=0.8\linewidth]{Simulation_figures/Data_Limit/legend.pdf}
	\end{subfigure}
	\caption{Average normalized error $\varepsilon:=\E\qty[\norm{\smash{\wh{\bm{A}}_{NM}-\bm{A}_N}}/\norm{\bm{A}_N}]$ as a function of the number of data points $M$. In Figures~\ref{ODE_small_noise}, \ref{ODE_medium_noise}, and \ref{ODE_large_noise}, we take $\sigma =10^{-3},10^{-2},10^{-1}$, respectively, and approximate the Koopman generator of the ODE. In Figures~\ref{OU_PF_small_noise}, \ref{OU_PF_medium_noise}, \ref{OU_PF_large_noise} and in \ref{OU_small_noise}, \ref{OU_medium_noise}, \ref{OU_large_noise} we also take $\sigma =10^{-3},10^{-2},10^{-1}$ and now approximate the Perron--Frobenius operator and Koopman generator of \eqref{OU}, respectively. In all cases, monomials up to order $8$ and the same number of Gaussian observables and FEM basis functions are used. The red and purple lines represent the slopes $-\frac{1}{2}$ and $-1$, respectively. The blue, red, and green lines represent the error averaged over $50$ simulations of the above approximations. The shaded areas represent the 95\% confidence intervals for the respective errors.}
	\label{figNoiseError}
\end{figure}
\newpage
\red{\subsection{Numerical results on spectral convergence}
The results contained in Section \ref{eigen section} do not provide bounds on the speed of convergence of the spectrum. In this section, we numerically compute the difference between the eigenvalues of $\wh{\b{A}}_{NM}$ and those of $\wh{\b{A}}_N$. As an example, we consider the Ornstein--Uhlenbeck process \eqref{OU} since, when using the monomials as basis functions, the eigenvalues of their projection onto the space of monomials are given by}
\begin{align*}
	\mathrm{spec}(\Ll_N )&= -n \alpha \quad \quad  n=0,1,\ldots,N-1, \\
	\mathrm{spec}(\Kk_N^t )&= e^{-n \alpha t}, \quad n=0,1,\ldots,N-1.
\end{align*}
We use the monomial, Gaussian, and FEM basis functions of Section \ref{simulations data limit section} and calculate the spectral error 
\begin{align}\label{spectral error}
\varepsilon_{\mathrm{spec}}:=	\norm{\mathrm{spec}(\wh{\b{A}}_{NM})-\mathrm{spec}(\wh{\b{A}}_N)}_{\R^N} = \left(\sum_{n=1}^N \abs{\lambda^{(n)}_{NM} - \lambda^{(n)}_{N}}^2\right)^{1/2},
\end{align}
where $\lambda^{(n)}_{NM}$ and $\lambda^{(n)}_{N}$ are the ordered $n$-th eigenvalues of $\wh{\b{A}}_{NM}$ and $\wh{\b{A}}_{N}$, respectively.

As before, we use an increasing number of data points using $M=2^8,2^9,\dots,2^{19}$, compute \eqref{spectral error} and repeat the approximation $50$ times for each $M$ to calculate the average spectral error $\varepsilon_{\mathrm{spec}}$. In Figure \ref{figDataConv_spec}, we plot in log-log scale the relationship between $M$ and $\varepsilon_{\mathrm{spec}}$, including a 95\% confidence interval for the error. To serve as a reference, we show dashed lines with slope $-\frac{1}{2}$ and $-1$, respectively.

Since the subspace spanned by the monomial basis functions is invariant under the Koopman generator, the error using monomial basis functions becomes zero, see Corollary~\ref{exact approximation}. Otherwise, the error has a slope of approximately $-\frac{1}{2}$. This suggests that, in some particular cases, error bounds similar to those of Section \ref{convergence Galerkin section} may hold for the spectrum.

\begin{figure}[h]
	\centering
	\begin{subfigure}{.32\linewidth}
		\centering
		\caption{}
		\includegraphics[width=\linewidth]{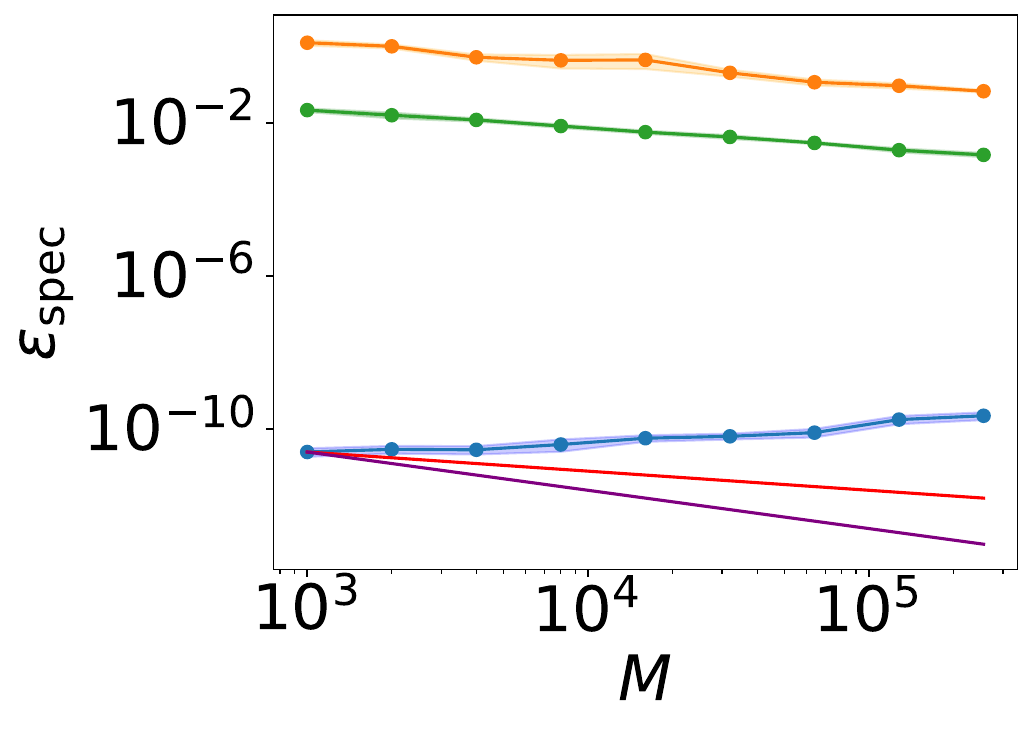}
		\label{imgOU_spec}
	\end{subfigure}
	\begin{subfigure}{.32\linewidth}
		\centering
		\caption{}
		\includegraphics[width=\linewidth]{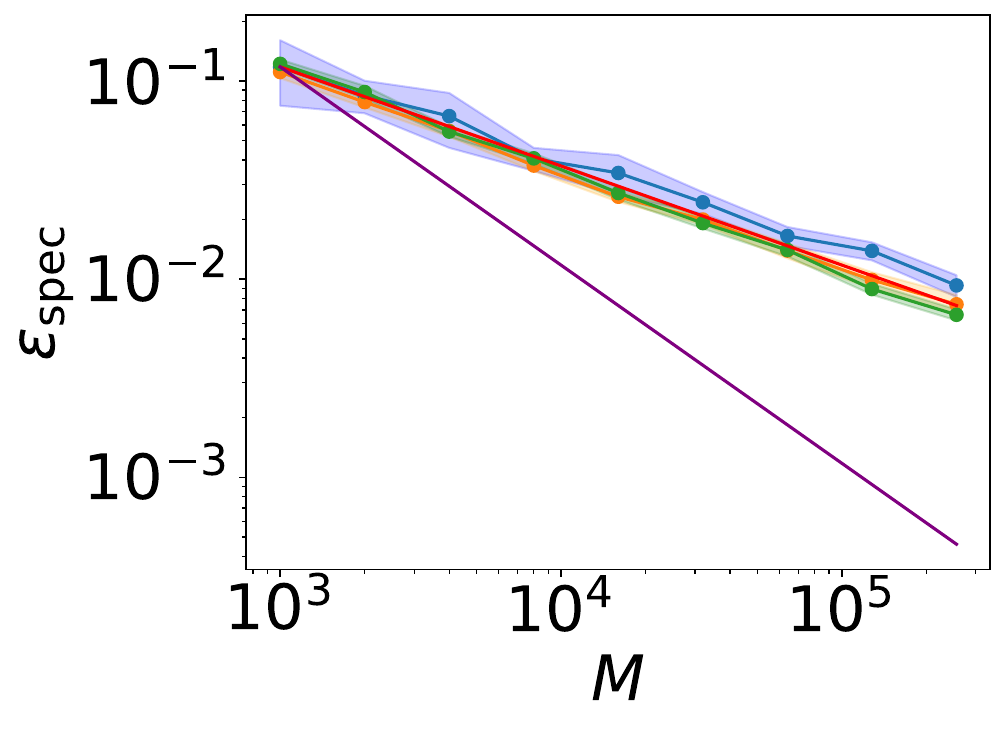}
		\label{imgOU_EDMD_spec}
	\end{subfigure}\\
	\begin{subfigure}{\linewidth}
		\centering
		\includegraphics[width=0.8\linewidth]{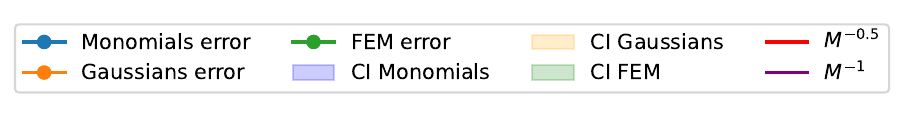}
		\label{legend_spec}
	\end{subfigure}
	\caption{Average spectral error $\varepsilon_{\mathrm{spec}}$ as a function of the number of data points $M$ for:  the Koopman generator, and the Koopman operator for the OU process \eqref{OU} in Figures~\ref{imgOU_spec}, and \ref{imgOU_EDMD_spec}. In all cases, monomials up to order $8$ and the same number of Gaussian observables and FEM basis functions are used. The red and purple lines represent the slopes $-\frac{1}{2}$ and $-1$, respectively. The blue, red and green lines represent the average error over $50$ simulations of the above approximations. The shaded areas represent the 95\% confidence intervals for the respective errors.}
	\label{figDataConv_spec}
\end{figure}
\section{Conclusion}\label{conclusion section}

In this article, we have investigated the approximation of an operator $\Aa$ from point evaluations of a dictionary to which the operator has been applied. That is, we have assumed that we have training data of the form
\begin{equation*}
	\big\{\psi_n(\bm{x}_m), \Aa \psi_n (\bm{x}_m)\big\}_{m,n=1}^{M,N},
\end{equation*}
where the evaluations could additionally be subject to random noise. After describing the estimation procedure for linear operators from data, we have presented a thorough convergence and error analysis in the dictionary limit ($N \rightarrow \infty$), the data limit ($M \rightarrow \infty$), and their joint limit ($N, M \rightarrow \infty$). We have studied the convergence of the full operators as well as their spectra.

Throughout this work, we have usually thought of the approximation of transfer operators in the context of dynamical systems, such as Koopman operators, Perron--Frobenius operators, and their generators. The framework we have presented is clearly not limited to such operators. Indeed, it generalises approximation techniques for transfer operators, such as EDMD and gEDMD, to general operators. EDMD and gEDMD fall naturally into our framework and we have shown significant new convergence results and error analyses for these methods. Additionally, we have verified our analytical results in numerical experiments.

\bibliography{biblio}

\appendix

\section{When does a matrix determine an operator?}

As discussed in Section \ref{notation section}, a matrix may not always define an operator when it is interpreted as acting on a set of vectors that are not linearly independent. The following lemma gives necessary and sufficient conditions for such an operator to be well-defined.
\begin{lemma} \label{operato lemma}
	Let $\bm{T} \in \C^{N_2\times N_1},\Psi_1=\set{\psi_n}_{n=1}^{N_1},\Psi_2=\set{\phi_n}_{n=1}^{N_2}, V_1= \mspan(\Psi_1)$ and $V_2= \mspan(\Psi_2)$. Given
	\begin{equation*}
		v=\sum_{j=1}^{N_1}c_j \psi_j \in V_1,
	\end{equation*}
	define
	\begin{equation} \label{matrix gives operator}
		\Tt v := \sum_{i=1}^{N_2}\sum_{j=1}^{N_1}c_j\bm{T}_{ij}\phi_i,
	\end{equation}
	then $\Tt \colon V_1 \to V_2 $ is a well-defined linear operator if and only if for all $(c_1, \dots,c_{N_1}) \in \C^{N_1}$
	\begin{equation} \label{well-defined}
		\sum_{j=1}^{N_1}c_j \psi_j=0 \implies \sum_{i=1}^{N_2}\sum_{j=1}^{N_1} c_j \bm{T}_{ij}\phi_i=0 .
	\end{equation}
	Then, $\bm{T}=\bm{T}^{\Psi_1 \to \Psi _2}$, that is, $\bm{T}$ is a matrix representation of $\Tt$ with respect to $\Psi_1, \Psi _2$.
\end{lemma}

\begin{proof}
	Suppose that condition \eqref{well-defined} holds. To see that $\Tt$ is well-defined, we check that, if $v \in V_1$ has two representations
	\begin{equation*}
		v=\sum_{j=1}^{N_1} b_j \psi_j= \sum_{j=1}^{N_1}b _j' \psi_j,
	\end{equation*}
	then $\Tt$ is equal on both representations. That is,
	\begin{equation*}
		\sum_{i=1}^{N_2}\sum_{j=1}^{N_1}b _j \bm{T}_{ij}\phi_i-\sum_{i=1}^{N_2}\sum_{j=1}^{N_1}b'_j \bm{T}_{ij}\phi_i=\sum_{i=1}^{N_2}\sum_{j=1}^{N_1}(b _j-b'_j) \bm{T}_{ij}\phi_i = 0.
	\end{equation*}
	This holds by \eqref{well-defined} with $c_j := b_j-b _j'$. The fact that $\Tt \colon V_1 \to V_2$ is linear follows from the definition in \eqref{matrix gives operator}.
	The reverse implication follows from the linearity of $\Tt$ as, if $v=\sum_{j=1}^{N_1}c_j \psi_j=0$, then
	\begin{equation*}
		0=\Tt v =\sum_{i=1}^{N_2}\sum_{j=1}^{N_1} c_j\bm{T}_{ij}\phi_i.
	\end{equation*}
	Finally, to check that if \eqref{well-defined} holds, then $\bm{T}= \bm{T}^{\Psi_1 \to \Psi _2}$ it suffices to take $v= \psi_j$ in \eqref{matrix gives operator}.
\end{proof}

\section{Bernstein inequality}

The following results can be found in \cite[page 8]{tropp2015introduction}.
\begin{theorem} \label{Bernstein theorem}
	Let $\bm{S}_1, \dots, \bm{S}_M \in \C^{N \times N}$ be independent, random matrices such that
	\begin{equation*}
		\mathbb{E}[\bm{S}_m]=\bm{0} \text { and }\left\|\bm{S}_m\right\| \leq L, \quad\forall m \in \set{1, \dots, M}.
	\end{equation*}
	Consider the sum
	\begin{equation*}
		\bm{Z}=\sum_{m=1}^M \bm{S}_m,
	\end{equation*}
	and let $v(Z)$ denote the matrix variance statistic of the sum:
	\begin{equation*}
		\nu(\bm{Z}) = \max \left\{\big\|\mathbb{E}\left(\bm{Z} \bm{Z}^\dagger\right)\big\|, \big\|\mathbb{E}\left(\bm{Z}^\dagger \bm{Z}\right)\big\|\right\}.
	\end{equation*}
	Then for every $\delta >0$
	\begin{equation*}
		\mathbb{P}\{\|\bm{Z}\| \geq \delta \} \leq 2N \exp \left(\frac{-\delta ^2 / 2}{\nu(\bm{Z})+L \delta / 3}\right), \quad\forall \delta \geq 0.
	\end{equation*}
\end{theorem}

\begin{corollary}[Bernstein inequality for the covariance] \label{Bernstein Lemma}
	Let $\bm{c}$ and $\bm{g}$ be two random vectors in $\C^n$ such that almost everywhere
	\begin{equation*}
		\abs{\bm{c}}^2\leq \gamma, \quad \abs{\bm{g}}^2\leq \gamma.
	\end{equation*}
	Let $\set{\bm{c}_m}_{m=1}^M, \set{\bm{g}_m}_{m=1}^M $ be copies of $\bm{c}$ and $\bm{g}$ respectively and such that $ \big\{\bm{c}_m \, \bm{g}^\dagger_m\big\}_{m=1}^M$ are independent. Define the matrices,
	\begin{align*}
		\bm{G}   & := \E\big[\bm{g}\bm{g^\dagger}\big], \quad \bm{T} := \E\big[\bm{c}\bm{c}^\dagger\big], \quad \bm{C} := \E[\bm{c}\bm{g}^\dagger], \\
		\bm{S}_m & := \frac{1}{M}\qty( \bm{c}_m\bm{g}_m^\dagger- \bm{C}), \quad \bm{Z} := \sum_{m=1}^M \bm{S}_m.
	\end{align*}
	Then
	\begin{equation*}
		\mathbb{P}\{\|\bm{Z}\| \geq \delta \} \leq 2 N \exp(\frac{-M \delta ^2 /2 }{ \gamma\qty(\max \set{\norm{\bm{T}}, \norm{\bm{G}}} +{2 \delta }/{3})}).
	\end{equation*}
	Furthermore, for all $p \in (0,1)$ and for all
	\begin{equation*}
		M>(3 \max \set{ \norm{\bm{G}}, \norm{\bm{T}}}+2 \delta ) \frac{2 \gamma}{3 \delta ^2}\log \left(\frac{2
			N}{1-p}\right),
	\end{equation*}
	it holds that
	\begin{equation*}
		\mathbb{P}\{\|\bm{Z}\| < \delta \}\geq p.
	\end{equation*}
\end{corollary}

\begin{proof}
	By construction, $\bm{S}_m$ are independent with mean zero so that we can apply Bernstein's inequality~\ref{Bernstein theorem}. We have
	\begin{equation*}
		\norm{\bm{S}_m}\leq \frac{1}{M}\qty(\big\|\bm{c}_m \bm{g}_m^\dagger \big\|+\norm{\bm{C}}) = \frac{1}{M}\qty( \norm{\bm{c}_m}\big\|\bm{g}_m^\dagger\big\|+\big\|\E\big[\bm{c}\bm{g}^\dagger\big]\big\|) \leq \frac{2\gamma}{M} =: L.
	\end{equation*}
	Next, we bound the matrix variance statistic $\nu(\boldsymbol{Z})$. First,
	\begin{align*}
		\mathbb{E} [\bm{S}_m\bm{S}_m^\dagger] & =\frac{1}{M^2} \E\qty[\norm{\bm{g}_m}^2 \bm{c}_m\bm{c}_m^\dagger -\bm{c}_m\bm{g}_m^\dagger \bm{C}^\dagger- \bm{C}\bm{c}_m\bm{g}_m^\dagger+\bm{C}\bm{C}^\dagger]\preccurlyeq\frac{1}{M^2}\qty( \gamma \bm{T} - \bm{C}\bm{C}^\dagger) \preccurlyeq \frac{\gamma}{M^2} \bm{T},
	\end{align*}
	where we used the notation $\boldsymbol{D} \preccurlyeq \boldsymbol{E}$ to signify that $\boldsymbol{E}-\boldsymbol{D}$ is positive semi-definite. Similarly,
	\begin{align*}
		\mathbb{E} [\bm{S}_m^\dagger\bm{S}_m] & =\frac{1}{M^2} \E\qty[\norm{\bm{c}_m}^2 \bm{g}_m\bm{g}_m^\dagger -\bm{C}^\dagger\bm{c}_m\bm{g}_m^\dagger - \bm{c}_m\bm{g}_m^\dagger\bm{C}+\bm{C}\bm{C}^\dagger]\preccurlyeq\frac{1}{M^2}\qty( \gamma \bm{G} - \bm{C}\bm{C}^\dagger) \preccurlyeq \frac{\gamma}{M^2} \bm{G} .
	\end{align*}
	Now, since $\bm{S}_m$ are independent with mean zero, we obtain that
	\begin{equation*}
		\nu(\bm{Z}) = \max \left\{\big\|\mathbb{E}\big(\bm{Z} \bm{Z}^\dagger\big)\big\|, \big\|\mathbb{E}\big(\bm{Z}^\dagger \bm{Z}\big)\big\|\right\} \leq \frac{\gamma}{M}\max \set{\norm{\bm{T}}, \norm{\bm{G}}}.
	\end{equation*}
	Applying Bernstein's inequality \ref{Bernstein theorem}, we obtain
	\begin{equation*}
		\mathbb{P}\{\|\bm{Z}\| \geq \delta \} \leq 2 N \exp(\frac{-M \delta ^2 /2 }{ \gamma\qty(\max \set{\norm{\bm{T}}, \norm{\bm{G}}} +{2 \delta }/{3})}).
	\end{equation*}
	Setting the right-hand side of the above to $1-p$ and solving for $M$ concludes the proof.
\end{proof}

\begin{lemma} \label{norm matrix operator lemma}
	Given $\Tt \colon \Ff_N\to\Ff_N$ it holds that $\norm{\Tt}\leq\sqrt{\kappa(\bm{G}_N)}\norm{\bm{T}^\Psi}$.
\end{lemma}

\begin{proof}
	To establish a bound, we begin by orthonormalising $\Psi_N$ by considering
	\begin{equation} \label{orthonormalise0}
		\tll{\Psi}_N := \bm{G}_N^{-\nicefrac12} \Psi_N.
	\end{equation}
	Now, given $\psi \in \Ff_N,$ we can write $\psi =\tll{\bm{c}}\cdot \tll{\Psi}_N$ for some $\tll{\bm{c}}\in \C^N$. Using that $\tll{\Psi}_N$ is orthonormal, and the expression for the change of basis matrix given by \eqref{orthonormalise0} shows that
	\begin{equation} \label{ortho norm0}
		\norm{\Tt \psi}_\Ff = \abs{\bm{T}^{\tll{\Psi}_N} \tll{\bm{c}}} = \abs{\bm{G}_N^{-\nicefrac12}\bm{T}^{\Psi_N}\bm{G}_N^{\nicefrac12} \, \tll{\bm{c}}}\leq \norm{\bm{G}_N^{-\nicefrac12}}\norm{\bm{T}^{\Psi_N}}\norm{\bm{G}_N^{\nicefrac12}}\abs{\bm{\tll{c}}}.
	\end{equation}
	Now, given a matrix $\bm{B}$, its operator norm in the Euclidean metric is
	\begin{equation*}
		\norm{\bm{B}}^2 = \lambda_{\max}(\bm{B}\bm{B}^\dagger).
	\end{equation*}
	Applying this to $\bm{G}_N^{\nicefrac12}$ and $\bm{G}_N^{-\nicefrac12}$ and substituting back into \eqref{ortho norm0} completes the proof.
\end{proof}

\section{Table of notation} \label{notation table section}

The notation used throughout the manuscript is summarised in Table~\ref{tab: notation}.

\begin{table}
	\centering
	\caption{Overview of the notation.}
	\renewcommand{\arraystretch}{1.2}
	\begin{tabular}{c|p{12cm}}
		\textbf{Symbol}                                               & \textbf{Description}                                                          \\ \hline
		$\mathbb{X}$                                                  & state space of the dynamical system                                           \\
		$\mu$                                                         & probability measure on $\mathbb{X}$                                           \\
		$\Ff = L^2(\X,\mu)$                                           & ambient space                                                                 \\
		$\Aa \colon \Dd \subset \Ff \to \Ff $                         & target linear operator                                                        \\
		$\mathcal{D}$                                                 & domain of the operator $\mathcal{A}$                                          \\ $\|\cdot\|_{\mathcal{F}},\|\cdot\|_{\mathcal{D}}$   & norms on the function spaces $\mathcal{F} $ and $\mathcal{D}$  														\\
		$\{\psi_n\}_{n=1}^N$                                          & dictionary or set of functions used to approximate $\Aa $                     \\
		$\Psi_N$                                                      & first $N$ elements $\psi_1 ,\dots, \psi_N $ of the dictionary                 \\
		$\set{\bm{x}_m}_{m=1}^M$                                      & data sampled i.i.d.\ from $\mu$ used to approximate $\Aa $                    \\
		$\mathcal{F}_N = \mspan(\Psi_N)$                              & finite-dimensional subspace on which $\Aa $ is approximated                   \\
		$\Ff_\infty = \bigcup_{n=1}^\infty  \Ff_n$                  & infinite-dimensional space spanned by the dictionary                          \\
		$\Pp _{\Ff_N} $                                               & projection operator onto $\mathcal{F}_N$ using inner product on $\Ff $        \\
		$\Pp_{\Dd_N}$                                                 & projection operator onto $\mathcal{D}_N$ using inner product on $\Dd $        \\
		$\Aa_N = \restr{\Pp_{\Ff_N}\Aa}{\Ff _N} $                     & Galerkin projection of $\Aa $ onto $\mathcal{F}_N$                            \\
		$\wh{\Aa}_{NM}$                                               & approximation of $\Aa$ and $\Aa_N $ using $M$ samples and $N$ basis functions \\
		$\bm{T}^\Psi$                                                 & matrix representation of the operator $\Tt $ using the basis given by $\Psi$  \\
		$\bm{C}_N, \bm{G}_N$                                          & structure matrix of $\Aa $  and Gram matrix w.r.t.\ the basis $\Psi_N $       \\
		$\wh{\bm{C}}_{NM}, \wh{\bm{G}}_{NM}$                          & empirical structure and Gram matrices                                         \\
		$\wh{\mu }_M = \frac{1}{M}\sum_{m=1}^M \delta_{\bm{x}_m} $    & empirical measure associated with the $M$ samples                             \\
		$\wh{\Ff}_{M} = L^2(\X, \wh{\mu }_M )$                        & empirical space associated with the $M$ samples                               \\
		$\wh{\phi} = \sum_{m=1}^M \phi (\bm{x}_m) \delta _{\bm{x}_m}$ & function $\phi \in \Ff $ when viewed in $\wh{\Ff }_{M} $                      \\
		$\wh{\Psi}_N$                                                 & first $N$ elements of the dictionary when viewed in $\wh{\Ff }_{M}$           \\
		$\wh{\Ff}_{NM} = \mspan(\wh{\Psi}_N) \subset \Ff_M $          & span of dictionary in empirical space                                         \\
		$\wh{\Tt }$                                                   & operator $\Tt $ viewed as acting on $\wh{\Ff }_M $                            \\
		$\bm{\eta }_N, \bm{\xi }_N$                                   & additive noise in the samples and observations                                \\
		$\co{\Aa}_{NM}, \co{\bm{C}}_{NM}, \co{\bm{G}}_{NM}$           & approximations of $\Aa_N, \bm{C}_N, \bm{G}_N$ when noise is present           \\
		$\Phi \colon \mathbb{X} \to \mathbb{X} $                      & flow of the dynamical system                                                  \\
		$\mathcal{K}$, $\mathcal{K}_*$                                & Koopman operator and Perron--Frobenius operator                               \\
		$\mathcal{L}$, $\mathcal{L}^*$                                & infinitesimal generator of the Koopman operator and its adjoint               \\
	\end{tabular}
	\label{tab: notation}
\end{table}

\end{document}